\documentclass[a4paper, 11pt]{amsart}
\usepackage{longtable, stackrel, multicol, fancybox, stackrel, tcolorbox, mathtools}
\usepackage{bm,amsfonts,amsmath,amssymb,mathrsfs,bbm,amsthm} 
\usepackage{graphicx, color,hyperref,eurosym, eucal,palatino}
\usepackage{scrextend}
\usepackage[normalem]{ulem} 
\usepackage{float} 
\changefontsizes{9.5pt}

\definecolor{darkblue}{rgb}{0.2,0.2,0.6}
\definecolor{superdarkblue}{rgb}{0.2,0.2,0.3}
\hypersetup{colorlinks=true, linkcolor=darkblue,citecolor=darkblue, urlcolor = superdarkblue}
\definecolor{citegreen}{rgb}{0.2,0.2,0.6}
\definecolor{green2}{rgb}{0.,0.6,0.2}

\DeclarePairedDelimiter{\norm}{\lVert}{\rVert}

\newcommand\V{\mathbb{V}}

\newcommand{\beq}{\begin{equation} \begin{split}}
\newcommand{\eeq}{\end{split} \end{equation}}

\newcommand\Omg{\Omega}

\makeatletter
\def\section{\@startsection{section}{1}\z@{.9\linespacing\@plus\linespacing}%
	{.7\linespacing} {\fontsize{13}{14}\selectfont\bfseries\centering}}
\def\paragraph{\@startsection{paragraph}{4}%
	\z@{0.3em}{-.5em}%
	{$\bullet$ \ \normalfont\itshape}}

\@addtoreset{equation}{section}
\makeatother

\renewcommand\and{\qquad\text{and}\qquad}

\newcommand\sm{\setminus}

\newcommand\dl{\delta}

\newcommand{\comm}[1]{}

\def\sfH{\mathsf{H}}\def\sfS{\mathsf{S}}
\def\sfS{\mathsf{S}}

\def\bm1{\mathbbm{1}}
\def\G{\Gamma}

\def\s{\sigma}

\def\p{\partial}

\def\sfW{\mathsf{W}}

\def\omg{\omega}

\newcommand{\spn}{\mathrm{span}\,}

\def\Re{{\rm Re}\,}

\def\arr{\rightarrow}

\def\tt{\theta}
\def\aa{\alpha}
\def\lm{\lambda}

\def\s{\sigma}

\def\ii{{\mathsf{i}}}
\def\p{\partial}

\def\sfH{\mathsf{H}}

\def\sfP{\mathsf{P}}
\def\dd{{\mathsf{d}}}
\def\uhr{\upharpoonright}

\def\omg{\omega}

\def\sfS{\mathsf{S}}

\def\sfW{\mathsf{W}}

\def\sfP{\mathsf{P}}
\def\sfU{\mathsf{U}}
\def\sfV{\mathsf{V}}

\newcommand*{\medoplus}{\mathbin{\scalebox{1.5}{\ensuremath{\oplus}}}}%

\newcounter{counter_a}
\newenvironment{myenum}{\begin{list}{{\rm(\roman{counter_a})}}%
		{\usecounter{counter_a}
			\setlength{\itemsep}{1.ex}\setlength{\topsep}{0.8ex}
			\setlength{\leftmargin}{5ex}\setlength{\labelwidth}{5ex}}}{\end{list}}

\usepackage[latin1]{inputenc}
\usepackage[T1]{fontenc}

\newcommand{\eg}{{\it e.g.}\,}
\newcommand{\ie}{{\it i.e.}\,}
\newcommand{\cf}{{\it cf.}\,}

\numberwithin{figure}{section}
\numberwithin{equation}{section}
\theoremstyle{plain}
\newtheorem*{thm*}{Theorem}
\newtheorem{thm}{Theorem}[section]

\newtheorem{lem}[thm]{Lemma}
\newtheorem{prop}[thm]{Proposition}

\newtheorem{cor}[thm]{Corollary}

\newtheorem{dfn}[thm]{Definition}
\theoremstyle{remark}
\newtheorem{remark}[thm]{Remark}
\theoremstyle{plain}

\newcommand{\beu}{\begin{equation*}}
\newcommand{\eeu}{\end{equation*}}
\newcommand{\besu}{\begin{equation*}
	\begin{aligned}}
\newcommand{\eesu}{\end{aligned}
	\end{equation*}}
\newcommand{\bes}{\begin{equation}
	\begin{aligned}}
\newcommand{\ees}{\end{aligned}
	\end{equation}}

\newcommand\cD{\mathcal D}
\newcommand\cF{\mathcal F}
\newcommand\cG{\mathcal G}
\newcommand\cH{\mathcal H}
\newcommand\cK{\mathcal K}

\newcommand\cL{\mathcal L}
\newcommand\cM{\mathcal M}
\newcommand\cN{\mathcal N}
\newcommand\cP{\mathcal P}

\newcommand\cS{\mathcal S}
\newcommand\CC{\mathbb C}

\newcommand\ZZ{\mathbb Z}

\newcommand\ov{\overline}
\newcommand\wt{\widetilde}
\newcommand\wh{\widehat}

\newcommand\void[1]{}

\def\ov{\overline}

      \def\dC{{\mathbb C}}

   \def\dN{{\mathbb N}}   
      \def\dR{{\mathbb R}}
\def\dS{{\mathbb S}}      
\def\dV{{\mathbb V}}      
   \def\dZ{{\mathbb Z}}

\def\sfA{{\mathsf A}}   \def\sfB{{\mathsf B}}   
\def\sfD{{\mathsf D}}      
   \def\sfH{{\mathsf H}}   \def\sfI{{\mathsf I}}
   \def\sfK{{\mathsf K}}   
      
\def\sfP{{\mathsf P}}   \def\sfQ{{\mathsf Q}}   
\def\sfS{{\mathsf S}}   \def\sfT{{\mathsf T}}   \def\sfU{{\mathsf U}}
\def\sfV{{\mathsf V}}   \def\sfW{{\mathsf W}}

      \def\cC{{\mathcal C}}
\def\cD{{\mathcal D}}   \def\cE{{\mathcal E}}   \def\cF{{\mathcal F}}
\def\cG{{\mathcal G}}   \def\cH{{\mathcal H}}   \def\cI{{\mathcal I}}
   \def\cK{{\mathcal K}}   \def\cL{{\mathcal L}}
\def\cM{{\mathcal M}}   \def\cN{{\mathcal N}}   
\def\cP{{\mathcal P}}      
\def\cS{{\mathcal S}}   \def\cT{{\mathcal T}}   \def\cU{{\mathcal U}}
\def\cV{{\mathcal V}}      
   \def\cZ{{\mathcal Z}}

\newcommand{\dom}{\mathrm{dom}\,}

\title[Self-adjointness for the MIT bag model on a cone]{Self-adjointness for the MIT bag model on an unbounded cone}

\author[B. Cassano]{Biagio Cassano}
\address{Dipartimento di Matematica e Fisica \\ Universit\`a degli Studi
  della Campania  ``Luigi Vanvitelli", Viale Lincoln 5, 81100 Caserta, Italy \\
  E-mail: {biagio.cassano@unicampania.it}}
\author[V. Lotoreichik]{Vladimir Lotoreichik}
\address{
Department of Theoretical Physics\\
Nuclear Physics Institute, Czech Academy of Sciences, 
25068 \v{R}e\v{z}, Czech Republic\\
E-mail: {lotoreichik@ujf.cas.cz }
}


\begin{document}
\begin{abstract}
	We consider the massless Dirac operator with the MIT bag boundary conditions on an unbounded three-dimensional circular cone. For convex cones, we prove that this operator is self-adjoint defined on four-component $H^1$--functions satisfying the MIT bag boundary conditions. The proof of this result
	relies on separation of variables
	and spectral estimates for one-dimensional fiber Dirac-type operators. Furthermore, we provide a numerical evidence for the self-adjointness on the same domain also for non-convex cones. Moreover, we prove a Hardy-type inequality for such a Dirac operator on convex cones, which, in particular, yields stability of self-adjointness under perturbations by a class of unbounded potentials. Further extensions of our results to Dirac operators with quantum dot boundary conditions are also discussed.
\end{abstract}

\maketitle
\section{Introduction}
In the present paper, we consider the three-dimensional massless
Dirac operator $\sfD_\omg$ on an unbounded Euclidean circular cone of half-aperture $\omg\in(0,\pi)$
\begin{equation}\label{eq:cone}
	\cC_\omg := 
	\left\{(x_1,x_2,x_3)\in\dR^3\colon x_3> \cot\omg\,(x_1^2+x_2^2)^{\frac12}
	\right\}
\end{equation}
with the MIT bag boundary conditions. These boundary conditions are
related to the celebrated MIT bag model~\cite{C75, CJJT74, CJJTW74,
  DJJK75, J75}, which describes the confinement of quarks in hadrons. The cone $\cC_\omg$ on which the Dirac operator is defined can be viewed 
as a Lipschitz hypograph in the sense of~\cite[Eq. (3.26)]{McL} and clearly $\cC_\omg$ is non-smooth due to the conical point at the origin $\mathbf{0}\in\dR^3$.
The operator domain of $\sfD_\omg$ 
consists of functions in the space $C^\infty_0(\ov{\cC_\omg}\sm\mathbf{0};\dC^4) :=\{u|_{\cC_\omg}\colon u\in C^\infty_0(\dR^3\setminus\mathbf{0};\dC^4)\} $
satisfying the MIT bag boundary conditions on the boundary of $\cC_\omg$. This densely defined operator can be shown to be symmetric. The natural question that arises whether this operator is essentially self-adjoint
or whether it has non-zero deficiency indices.   

In the existing literature self-adjointness of the
three-dimensional Dirac operators
with the MIT bag boundary
conditions was first shown for $C^2$-smooth domains with
compact boundaries~\cite{OV18, BHM20}; see also \cite{ALR17, ALR20}
for related approaches. 
It is possible to regard Dirac operators with MIT bag boundary
conditions in the framework of 
Dirac operators with Lorentz-scalar $\delta$-shell interactions; see \eg~\cite{AMV14, AMV15, AMV16, BEHL16, BHSS21, 
HOP18, OV18, CLMT21, B21b, R2022}, the review papers~\cite{BEHL19, OP21}, and the references 
therein. In particular, in~\cite{B21b}
self-adjoint\-ness for interactions with non-compact supports less regular than
$C^2$-smooth is shown, but the regularity assumptions there still
exclude the conical surface. 
On the other hand, in~\cite{BHSS21} self-adjointness for interactions supported on boundaries of general bounded Lipschitz domains that can have conical points is proved: a difference with the approach of~\cite{BHSS21} is that they define the Dirac operator with the operator domain contained in the Sobolev space $H^{1/2}$ whereas in our construction we 
verify (essential) self-adjointness for a more regular
operator domain, contained in the Sobolev space $H^1$. We remark that the unboundedness of the conical domain is not a real issue: away from the vertex singularity the boundary of the domain is smooth, so it can be handled as in \cite{B21b, R2020}.
Considering a more regular operator domain can in many cases lead to non-zero deficiency indices and, in fact, this happens
in a related setting of two-dimensional polygons analysed in~\cite{LO18,PV19}.

We prove that the  Dirac operator $\sfD_\omg$ is unitarily equivalent to an infinite orthogonal sum of Dirac operators on the half-line with off-diagonal Coulomb-type potentials, whose coefficients
are given by the eigenvalues of an effective essentially self-adjoint spin-orbit operator $\sfK_\omg$ on the spherical cap $\cC_\omg\cap\dS^2$, where $\dS^2\subset\dR^3$ is the two-dimensional unit sphere centred at the origin. We
characterise the eigenvalues of $\sfK_\omg$
in terms of solutions of transcendental equations and  show that for
convex cones $\omega < \pi/2$ all the eigenvalues
of $\sfK_\omg$ escape the critical interval $[-\frac12,\frac12]$.
This construction implies that $\sfD_\omg$ is
essentially self-adjoint in this case. Moreover, we show that for $\omg < \pi/2$ the domain of the self-adjoint closure $\ov{\sfD_\omg}$ of $\sfD_\omg$
	consists of four-component $H^1$-functions satisfying the MIT bag boundary conditions in the sense of traces.
We complement our
analysis by numerical results describing the eigenvalue distribution
 of $\sfK_\omg$ for $\pi/2 < \omega < \pi$, supporting the claim that
 $\sfD_\omg$ is essentially self-adjoint for any $\omega \in (0,\pi)$: we remark that
 the numerical approach is only used to find the roots of explicit transcendental equations associated to the problem.
Finally, we study for $\omg < \pi/2$ the stability of
the self-adjointness $\ov{\sfD_\omg}$ under perturbations
by Hermitian matrix-valued potentials that can have Coulombic singularity at the origin. 
 
Our analysis is inspired by the closely related 
considerations for the Laplace operator on planar domains with corners and three-dimensional domains with conical points
and for the Dirac operator on planar domains with corners. Let us briefly review the existing literature on the subject.  The two-dimensional Laplace operator 
with Dirichlet boundary conditions on a curvilinear polygon $\cP\subset\dR^2$ was first considered in~\cite{BS62} and then further analysed in~\cite[Chap. 4]{G85}, \cite[Chap. 2]{G92}, \cite{D88} and more recently in \cite{P13}. 
In particular, it is proved that the deficiency indices of such a Laplace operator with the domain $H^2(\cP) \cap H^1_0(\cP)$
are $(n,n)$, where $n$ is the number of non-convex corners of $\cP$; \ie corners with the opening angle larger than $\pi$.
A similar effect occurs for the Robin Laplacian, where the coefficient in the boundary conditions having a specific singularity at a boundary point
creates deficiency indices~\cite{ES88, MR09, NP18}.	
The three-dimensional Laplace operator with Dirichlet boundary conditions on a bounded domain $\Omega\subset\dR^3$ with a conical point is analysed first in~\cite{HS67,K67}, see also \cite[\S 8.2.2]{G85} and \cite{BDY99,D88}.
The appearance of deficiency indices here is more subtle and we discuss it in detail in Remark~\ref{rem:Laplace_cone}. It should be emphasized that for the Dirichlet Laplacian
on a domain that has a conical point and which is smooth outside it, the transition between self-adjointness and existence of non-trivial deficiency indices indeed occurs for some critical opening angle of the cone. In this perspective the phenomenon for the Dirac operator is substantially different.
As for Dirac operators, analysis of deficiency indices for boundaries with singular points has only been carried out
in two dimensions~\cite{FL21, LO18, PV19, CL20} so far. The deficiency indices on a curvilinear polygon here are characterised in the same way as for the Dirichlet Laplacian~\cite{LO18, PV19}. 
For the discontinuous infinite mass boundary conditions on a sector the deficiency indices can be unequal \cite{CL20}
while for Lorentz-scalar $\delta$-shell interactions supported on star-graphs~\cite{FL21} various scenarios are possible and, in particular, the deficiency indices depend on the strength of the interaction.

The structure of the paper is as follows. In Section~\ref{sec:main_res} we rigorously define the Dirac operator on an unbounded cone with MIT bag boundary conditions and formulate all the main results of the paper. Section~\ref{sec:pre} contains preliminary material used throughout the paper. In this section we recall the concept of boundary triples, analyse self-adjointness of a model Dirac operator on the half-line, and find the form of the spin-orbit differential expression in the spherical coordinates. Further, in Section~\ref{sec:Dirac_model} we study the spectrum of a model Dirac-type operator on the interval $(0,\omg)$. Relying on this spectral analysis we decompose the operator $\sfD_\omg$ in Section~\ref{sec:decomp} in spherical coordinates and using this decomposition we prove all the main results of the paper. Finally, the paper is complemented by
Appendix~\ref{app:pderiv} with the derivation of expressions for partial derivatives in spherical coordinates, Appendix~\ref{app:Dirac} on Dirac systems, and Appendix~\ref{app:plots} with numerical results for non-convex cones. 

\section{Setting and the main results}
\label{sec:main_res}
In this section, we will first rigorously define the operator $\sfD_\omg$ clarifying the MIT bag boundary conditions. 
Then we will formulate all the main results of the present paper.
\subsection{Definition of the operator}
\label{ssec:def}
The problem is better described in  spherical coordinates  $(r,\tt,\phi) \in \dR_+\times[0,\pi] \times [0,2\pi)$ on $\dR^3$, which are connected with the Cartesian coordinates $(x_1,x_2,x_3)$ by the well-known formulae
\begin{equation}\label{eq:x123}
\begin{cases}
x_1  = r\cos\phi\sin\tt,\\
x_2  = r\sin\phi\sin\tt,\\
x_3  = r\cos\tt.
\end{cases}
\end{equation}
In this convention the north pole of the unit sphere $\dS^2$ has coordinates 
	$r =1, \tt = 0$, and $\phi = 0$.
The cone with the half-aperture $\omg  \in (0,\pi)$ defined in~\eqref{eq:cone} can be alternatively characterised in the spherical coordinates as
\begin{equation}\label{eq:cone_spherical}
	\cC_\omg = \big\{(r\sin\tt\cos\phi,r\sin\tt\sin\phi,r\cos\tt)\colon r\in\dR_+, \tt \in [0,\omg), \phi\in[0,2\pi)\big\}.
\end{equation}
The unit normal vector to the boundary $\p\cC_\omg$ of the cone
$\cC_\omg$ pointing outwards is explicitly given as
\[
	\nu_\omg(r,\phi) = (\cos\omg\cos\phi,\cos\omg\sin\phi,-\sin\omg),\qquad r\in\dR_+, \phi\in[0,2\pi).
\]

After having described the geometric setting, we can
proceed to the definition of the Dirac operator on $\cC_\omg$. To this aim we recall the expressions for the $\dC^{2\times2}$ Pauli matrices
\[
	\s_1 := \begin{pmatrix}0 &  1 \\ 1& 0\end{pmatrix},\qquad
	\s_2 := \begin{pmatrix}0 &  -\ii \\ \ii& 0\end{pmatrix},\qquad \s_3 := \begin{pmatrix}1 & 0 \\ 0& -1\end{pmatrix},
\]
and for the $\dC^{4\times4}$ Dirac matrices $\aa_j$, $j=1,2,3$, and $\beta$
\[
	\aa_j := \begin{pmatrix} 0 & \s_j\\ \s_j &0\end{pmatrix},\quad j=1,2,3,\qquad\quad
	\beta = \begin{pmatrix} I_2 & 0\\ 0 & -I_2\end{pmatrix},
\]
where $I_2$ is the identity matrix in $\dC^2$. The Dirac matrices are Hermitian and they satisfy the anti-commutation relations
\[
	\aa_k\aa_j+ \aa_j\aa_k = 2\dl_{kj}\qquad\text{and}\qquad \aa_k\beta+\beta\aa_k = 0,\quad k,j\in\{1,2,3\},
\]
where $\dl_{kj}$ is the Kronecker symbol.
We define the triple of the Dirac matrices $\aa := (\aa_1,\aa_2,\aa_3)$ and define $\aa\cdot b$ with $b\in\dC^3$ as
$\aa\cdot b = \aa_1b_1+\aa_2b_2+\aa_3 b_3$.
With all the above preparations, the Dirac differential expression in $\dR^3$ is given by 
\[
	\cD := -\ii\aa\cdot\nabla = -\ii(\aa_1\p_1+\aa_2\p_2 +\aa_3\p_3).
\]
We consider the following Dirac operator in the Hilbert space $(L^2(\cC_\omg;\dC^4),(\cdot,\cdot)_{\cC_\omg})$ (here and in the following all the inner products are linear in the first entry)
\begin{equation}\label{eq:Dirac_operator}
\begin{aligned}
	\sfD_\omg u& := \cD u,\\
	\dom\sfD_\omg& := 
	\big\{u\in C^\infty_0(\ov{\cC_\omg}\sm \mathbf{0};\dC^4)\colon u|_{\p\cC_\omg} + \ii\beta(\aa\cdot\nu_\omg) u|_{\p\cC_\omg} = 0\big\},
\end{aligned}
\end{equation}
where $C^\infty_0(\ov{\cC_\omg}\sm \mathbf{0};\dC^4) = \{u|_{\cC_\omg}\colon u\in C^\infty_0(\dR^3\setminus\mathbf{0};\dC^4)\}$ is the space of smooth four-component functions on $\ov{\cC_\omg}$, the support of which is a compact subset of the closure of the cone with the removed tip $\ov{\cC_\omg}\sm\mathbf{0}$.
The operator $\sfD_\omg$ is symmetric.
Indeed, integrating by parts we find for $u,v\in\dom\sfD_\omg$
\[
\begin{aligned}
	(-\ii\aa\cdot\nabla u,v)_{\cC_\omg} - ( u,-\ii\aa\cdot\nabla v)_{\cC_\omg}
	& \!=\! \int_{\p \cC_\omg} \langle-\ii(\aa\cdot\nu_\omg)u|_{\p\cC_\omg},  v|_{\p \cC_\omg}\rangle_{\dC^4}\dd\s\\
	& \!=\! \int_{\p\cC_\omg} \langle(\aa\cdot\nu_\omg)(\beta(\aa\cdot\nu_\omg))u|_{\p\cC_\omg},  \ii \beta (\aa\cdot\nu_\omg) v|_{\p\cC_\omg}\rangle_{\dC^4}\dd\s\\
	&\!=\!
	\int_{\p\cC_\omg} \langle\ii(\aa\cdot\nu_\omg)^2(\beta^2(\aa\cdot\nu_\omg))u|_{\p\cC_\omg},   v|_{\p\cC_\omg}\rangle_{\dC^4}\dd\s\\
	&\!=\!
	\int_{\p\cC_\omg} \langle\ii(\aa\cdot\nu_\omg)u|_{\p\cC_\omg},   v|_{\p\cC_\omg}\rangle_{\dC^4}\dd\s =0,
\end{aligned}
\]
where we used that $\beta^2 = (\aa\cdot\nu_\omg)^2 = 1$ in the penultimate step.

It can be shown by the same integration by parts that the extension
\begin{equation}\label{eq:extension}
\{u\in H^1(\cC_\omg;\dC^4)\colon u|_{\p\cC_\omg} + \ii\beta(\aa\cdot\nu_\omg)u|_{\p\cC_\omg} = 0\}\ni u\mapsto \cD u
\end{equation}
of the Dirac operator $\sfD_\omg$ is symmetric, where the boundary conditions are understood in the sense of traces in the Sobolev space $H^{1/2}(\p\cC_\omg;\dC^4)$.

The MIT bag boundary condition
$u|_{\p\cC_\omg} + \ii\beta(\aa\cdot\nu_\omg) u|_{\p\cC_\omg} = 0$ can be written in terms of components $u=(u_1,u_2,u_3,u_4)^\top$ of the vector-valued function $u$  
\[
	\begin{pmatrix}
	1 & 0 & -\ii\sin\omg & \ii\cos\omg e^{-\ii \phi}\\
	0 & 1 & \ii\cos\omg  e^{\ii\phi}  & \ii\sin\omg\\
 \ii\sin\omg & -\ii\cos\omg e^{-\ii \phi} &1 &0\\
 -\ii\cos\omg  e^{\ii\phi}  & -\ii\sin\omg &0 &1
 \end{pmatrix}
 	\begin{pmatrix}
 u_1(r,\omg,\phi)\\
 u_2(r,\omg,\phi)\\
 u_3(r,\omg,\phi)\\
 u_4(r,\omg,\phi)
 \end{pmatrix} =0,
\]
where $r\in\dR_+$ and $\phi\in[0,2\pi)$.
This condition is equivalent to
\[
\begin{cases}
	u_1(r,\omg,\phi) - \ii\sin\omg u_3(r,\omg,\phi)+\ii\cos\omg e^{-\ii\phi} u_4(r,\omg,\phi) = 0,\\
	u_2(r,\omg,\phi)+\ii\cos\omg e^{\ii\phi} u_3(r,\omg,\phi)+\ii\sin\omg u_4(r,\omg,\phi) = 0.
\end{cases}
\]

\subsection{Main results}

In order to formulate our first main result
we need some preparation. We introduce for $\omg\in(0,\pi)\sm\{\frac{\pi}{2}\}$ the following matrix differential expressions on
the interval $\cI_\omg := (0,\omg)$ 
\[
	\tau_k := \begin{pmatrix}
	k+1 & -\frac{\dd}{\dd\tt}-\left(k+\frac12\right)\cot\tt\\
	\frac{\dd}{\dd\tt}-\left(k+\frac12\right)\cot\tt &-k\end{pmatrix},\qquad \cT_k := \begin{pmatrix}
	\tau_k &0\\
	0 &\tau_k\end{pmatrix},\qquad k\in\dZ,
\]
and 
associate to them the differential operators $\sfT_k$, $k\in\dZ$, in the Hilbert space $L^2(\cI_\omg;\dC^4)$
defined by
\begin{equation}\label{eq:Tkintro}
\begin{aligned}
	\sfT_k\psi &:=\cT_k\psi,\\
	\dom\sfT_k &:= \bigg\{\psi\in{\rm AC}_{\rm loc}(\cI_\omg;\dC^4)\colon \psi,\cT_k\psi \in L^2(\cI_\omg;\dC^4),
	 \left(\begin{smallmatrix}
	\psi_1(\omg)\\ \psi_2(\omg)\end{smallmatrix}\right) = \sfA_\omg \left(\begin{smallmatrix}
	\psi_3(\omg)\\ \psi_4(\omg)\end{smallmatrix}\right)\bigg\},
\end{aligned}
\end{equation}
where the matrix $\sfA_\omg$ is given by
\[
	\sfA_\omg := 
	\begin{pmatrix} 
	\ii\sin\omg &-\ii\cos\omg\\
	-\ii\cos\omg
	&-\ii\sin\omg\end{pmatrix}.
\]
We check that the operators $\sfT_k$, $k\in\dZ$, are self-adjoint, that their spectra are discrete, simple and symmetric about the origin and that zero is not an eigenvalue of any of these operators. Moreover, it turns out that $\s(\sfT_k) = \cZ_k\cup(-\cZ_k)$, where
the discrete subsets $\cZ_k$, $k\in\dZ$, of the real axis are defined by
\begin{equation}\label{eq:Zk0}
\begin{aligned}
	\cZ_{k} &:= \big\{\lm\in\dR\colon (\lm+k+1)\sfP_\lm^{-k-1}(\cos\omg) = \sfP_{\lm-1}^{-k}(\cos\omg)\big\}, & k&\in\dN_0,\\
	\cZ_k &:=\big\{\lm\in\dR\colon (\lm+k)\sfP_{\lm-1}^{k}(\cos\omg) = -\sfP_{\lm}^{k+1}(\cos\omg)\big\}, & -k&\in\dN,
\end{aligned}
\end{equation}	
where $\sfP_\nu^\mu(x)$, $x\in(-1,1)$, is the Ferrers function of the first kind~\cite[\S 5.15]{O97} and \cite[\S 8.7 and 8.8]{GR07}. Here, we drop the dependence on $\omg$ in the notation
$\sfT_k$ and $\cZ_k$ for the sake of brevity.

For $\lm\in\dR$, we introduce the auxiliary one-dimensional Dirac operator in the Hilbert space $L^2(\dR_+;\dC^2)$ by
\begin{equation}\label{eq:fibre}
		\mathbf{d}_\lm\psi := 
		\begin{pmatrix} 0 & -\frac{\dd}{\dd x} - \frac{\lm}{x}\\
		\frac{\dd}{\dd x} - \frac{\lm}{x} & 0\end{pmatrix}
		\begin{pmatrix}\psi_1\\ \psi_2\end{pmatrix},\qquad \dom\mathbf{d}_\lm := C^\infty_0(\dR_+;\dC^2).
\end{equation}
In the first main result we decompose $\sfD_\omg$ into an orthogonal sum of the fiber operators.
\begin{thm}\label{thm:main1}
  Let $\omg\in(0,\pi)\sm\{\frac{\pi}{2}\}$ and let the Dirac operator $\sfD_\omg$ be as in~\eqref{eq:Dirac_operator}.
  Let the operators $\sfT_k, k\in\dZ$, be as in~\eqref{eq:Tkintro}, the discrete sets $\cZ_k$ be as in~\eqref{eq:Zk0} and let the fiber operators $\mathbf{d}_\lm$ be as in~\eqref{eq:fibre}. Then the following unitary equivalences hold
  \begin{equation}\label{eq:Dirac_ortho}
    \sfD_\omg\simeq
    	\bigoplus_{k\in\dZ}
    \bigoplus_{\lm\in\s(\sfT_k)\cap(0,\infty)} \mathbf{d}_\lm\simeq \bigoplus_{k\in\dZ}\bigoplus_{\lm\in\cZ_k} \mathbf{d}_\lm.
  \end{equation}
\end{thm}
The proof of Theorem~\ref{thm:main1} relies on the representation of
$\sfD_\omg$ in spherical coordinates, in which the spin-orbit operator $\sfK_\omg$
acting in the Hilbert space $L^2([0,\omg)\times
[0,2\pi);\sin\tt\dd\tt\dd\phi)$ arises. The spin-orbit operator $\sfK_\omg$ is defined in~\eqref{eq:Komg} below. The operators $\sfT_k$, $k\in\dZ$,
appear naturally in the analysis
upon separation of variables for the spin-orbit operator $\sfK_\omg$.
We decompose the
Hilbert space $L^2(\cC_\omg;\dC^4)$ into an orthogonal sum of
subspaces, whose construction is based on the eigenfunctions of the aforementioned spin-orbit operator $\sfK_\omg$. This orthogonal decomposition of $L^2(\cC_\omg;\dC^4)$ enables to characterise the reducing subspaces for $\sfD_\omg$. It turns out that the compressions of $\sfD_\omg$ to these reducing subspaces are unitarily equivalent to the one-dimensional fiber operators in~\eqref{eq:fibre} with specific $\lm\in\dR$.  
\begin{remark}
	 The case $\omg = \frac{\pi}{2}$ does not fit to our general scheme of reduction and requires a separate consideration within our approach.
	 However, in this case the cone $\cC_{\pi/2}$ is the half-space and the theory is already established. The reader is referred to \cite{ALR20, B21a}, where it is shown that the Dirac operator on $\cC_{\pi/2}$ is self-adjoint defined on functions in the Sobolev space $H^1(\cC_\omg;\CC^4)$ satisfying the MIT bag boundary conditions.
	For this reason, we exclude the case $\omega = \frac{\pi}{2}$ from our results.
\end{remark}

Relying on the analysis of the fiber operators we get the following second main result.
\begin{thm}\label{thm:main2}
Let $\sfD_\omg$ be defined as 
in~\eqref{eq:Dirac_operator}. If $\omega \in (0,\pi/2)$ then $\sfD_\omg$ is essentially
self-adjoint and
\begin{equation}\label{eq:domain.closure}
  	\dom\overline{\sfD_\omg} := 
	\big\{u\in H^1(\cC_\omg ;\dC^4)\colon u|_{\p\cC_\omg} + \ii\beta(\aa\cdot\nu_\omg) u|_{\p\cC_\omg} = 0\big\},
\end{equation}
	where the trace $u|_{\p\cC_\omg}$ is well defined as a function in the Sobolev space $H^{1/2}(\p\cC_\omg;\dC^4)$.
\end{thm}
The proof of Theorem~\ref{thm:main2} is based on the decomposition in Theorem~\ref{thm:main1} and reduces to showing that the eigenvalues of
	the spin-orbit operator $\sfK_\omg$
escape the critical interval $[-\frac12,\frac12]$ or, equivalently, that $\s(\sfT_k)\cap[-\frac12,\frac12]=\emptyset$ for all $k\in\dZ$.

\begin{remark}\label{rmk:conjecture}
  We conjecture that Theorem \ref{thm:main2} is true for all the values $\omega \in (0,\pi)$ of the aperture angle. In Appendix \ref{app:plots} we provide numerical evidence to support our conjecture, 
 showing that for $\pi/2 < \omega < \pi$ all the eigenvalues
of the spin-orbit operator $\sfK_\omg$ escape the critical interval $[-\frac12,\frac12]$. This implies essential self-adjointness of $\sfD_\omg$ and the validity of \eqref{eq:domain.closure} for all $\omg\in(0,\pi)$.
\end{remark}
\begin{remark}
	The operator $\sfK_\omg$ can be viewed as a Dirac-type operator on the manifold $\dS^2\cap\cC_\omg$. Geometric spectral bounds in the spirit of~\cite{HMZ01} can also be useful to show that there are no eigenvalues of $\sfK_\omg$ in the interval $[-\frac12,\frac12]$. This approach is also suitable to deal with
	general non-circular cones, for which separation of variables can not be used any more. 
\end{remark}

\begin{remark}\label{rmk:mit_minus}
In the literature, sometimes also the following operator $\sfD_\omg^{-}$ is called Dirac operator with MIT bag boundary conditions:
\begin{equation}\label{eq:Dirac_operator_-}
\begin{aligned}
	\sfD_\omg^{-} u& := \cD u,\\
	\dom\sfD_\omg^{-}& := 
	\big\{u\in C^\infty_0(\ov{\cC_\omg}\sm \mathbf{0};\dC^4)\colon u|_{\p\cC_\omg} - \ii\beta(\aa\cdot\nu_\omg) u|_{\p\cC_\omg} = 0\big\},
\end{aligned}
\end{equation}
the difference between $\sfD_\omg$ and $\sfD_\omg^{-}$ standing in the minus sign in the boundary conditions of $\sfD_\omg^{-}$. 
The operator $\sfD_\omg^{-}$ is unitarily equivalent to $\sfD_\omg$ through the unitary and self-adjoint matrix
\[
\gamma_5 
= \begin{pmatrix}
0 & I_2 \\
I_2 & 0
\end{pmatrix},
\]
\ie~$\sfD_\omg^{-} = \gamma_5 \sfD_\omg \gamma_5$, see \cite[Proof of Prop.~3.4]{BHM20}. As a consequence, Theorem \ref{thm:main2} immediately implies that for $\omg \in (0, \pi/2)$ the Dirac operator $\sfD_\omg^{-}$ is essentially self-adjoint and 
\begin{equation*}
  	\dom\overline{\sfD_\omg^{-}} := 
	\big\{u\in H^1(\cC_\omg ;\dC^4)\colon u|_{\p\cC_\omg} - \ii\beta(\aa\cdot\nu_\omg) u|_{\p\cC_\omg} = 0\big\}.
\end{equation*}
\end{remark}
\begin{remark}\label{rem:Laplace_cone}
  It is interesting to compare Theorem~\ref{thm:main2} with related results for the Laplace operator with Dirichlet boundary condition.
	Let $\Omega\subset\dR^3$ be a bounded domain such that $\mathbf{0}\in\p\Omg$, that $\p\Omg\sm\{\mathbf{0}\}$ is $C^2$-smooth and that there exists a neighbourhood of $\cV$ of $\mathbf{0}$ such that in $\cV$ the domain $\Omega$ coincides with an unbounded circular cone $\cC_\omg$ with half-aperture $\omg\in(0,\pi)$.  We consider the densely defined symmetric Dirichlet Laplacian on $\Omg$
	\[
		\sfH^\Omg_{\rm D} u := -\Delta u,
		\qquad \dom\sfH^\Omg_{\rm D} := H^2(\Omg)\cap H^1_0(\Omg).
	\]
	Let $t_0$ be the zero of $\sfP_{\frac12}^0$ on the interval $(-1,1)$, and set $\omg_{\rm cr} := \arccos(t_0)\approx 0.726\pi$.
	The operator $\sfH_{\rm D}^\Omg$ is closed by \cite[Thm. 8.2.2.1]{G85} and \cite[Lem. 4.1]{HS67} for all $\omg\ne \omg_{\rm cr}$.
	Combining \cite[Thm. 8.2.2.6]{G85} and \cite[Sec. 4]{HS67} one gets that if the angle $\omg > \omg_{\rm cr}$ then the deficiency indices of $\sfH_{\rm D}^\Omg$ are $(1,1)$ while if $\omg < \omg_{\rm cr}$ the operator $\sfH_{\rm D}^\Omg$ is self-adjoint. 
	Similar results are expected to hold also for the Dirichlet Laplacian on $\cC_\omg$, but to the best of our knowledge they are not stated in the existing literature in terms of deficiency indices. In this perspective the result for the Dirac operator $\sfD_\omg$ is qualitatively different. The transition to non-trivial deficiency indices for large half-aperture does not happen for $\sfD_\omg$ as suggested by our numerical analysis of the problem.
\end{remark}
It is possible to consider more general boundary conditions for the Dirac operator. For $\theta \in [0,2\pi)$, the Dirac operator $\sfD_\omg^\theta$ with \emph{quantum dot} bondary conditions is defined as follows: 
\begin{equation}\label{eq:Dirac_operator_quantumdot}
\begin{aligned}
	\sfD_\omg^{\theta} u& := \cD u,\\
	\dom\sfD_\omg^{\theta}& := 
	\big\{u\in C^\infty_0(\ov{\cC_\omg}\sm \mathbf{0};\dC^4)\colon u|_{\p\cC_\omg} = [(\cos\theta) \ii(\aa\cdot\nu_\omg)\beta + (\sin\theta)\beta] u|_{\p\cC_\omg} 
\big\}.
\end{aligned}
\end{equation}
One sees that $\sfD_\omg^{0} = \sfD_\omg$ and $\sfD_\omg^{\pi} = \sfD_\omg^{-}$ are the Dirac operators with MIT bag boundary conditions defined in \eqref{eq:Dirac_operator} and \eqref{eq:Dirac_operator_-} respectively; the operators $\sfD_\omg^{\frac{\pi}2}$ and $\sfD_\omg^{\frac{3\pi}2}$ are called Dirac operators with \emph{zig-zag} boundary conditions. From Theorem \ref{thm:main2} it is immediate to give a result on essential self-adjointness and self-adjointness for Dirac operators  with quantum dot boundary conditions, except for the case of zig-zag boundary conditions.
\begin{prop}\label{prop:quantumdot}
Let $\theta \in [0,2\pi) \setminus \{\frac{\pi}2,\frac{3\pi}2\}$ and let $\sfD_\omg^\theta$ be defined as 
in~\eqref{eq:Dirac_operator_quantumdot}. If $\omega \in (0,\pi/2)$ then $\sfD_\omg^\theta$ is essentially
self-adjoint and
\begin{equation}\label{eq:domain.closure2}
  	\dom\overline{\sfD_\omg^\theta} := 
	\big\{u\in H^1(\cC_\omg ;\dC^4)\colon u|_{\p\cC_\omg} = [(\cos\theta) \ii(\aa\cdot\nu_\omg)\beta + (\sin\theta)\beta] u|_{\p\cC_\omg} \big\},
\end{equation}
	where the trace $u|_{\p\cC_\omg}$ is well defined as a function in the Sobolev space $H^{1/2}(\p\cC_\omg;\dC^4)$.
\end{prop}
\begin{remark}
The Dirac operators $\sfD_\omg^{\frac{\pi}2}$ and $\sfD_\omg^{\frac{3\pi}2}$ with zig-zag boundary conditions represent different phenomena and are not considered in this work. We refer to \cite{H21} and references therein. 
\end{remark}

We complete our manuscript with a result on stability of essential self-adjointness and self-adjointness under perturbations of $\sfD_\omg$ by regular potentials.
\begin{thm}\label{thm:perturbation}
Let $\omega \in (0,\pi/2)$ and $\V\colon\cC_\omg \to \dC^{4\times 4}$ such that $\V(x)$ is Hermitian for all $x \in \cC_\omg$ and 
\begin{equation}\label{eq:bnd_V}
\sup_{x \in \cC_\omg} |x| |\V(x)| = \nu.
\end{equation}
If $\nu \le \frac{\pi}{4\omg}$, 
then the Dirac operator $\sfD_\omg + \V$ is essentially self-adjoint, with 
\begin{equation*}
\begin{split}
&( \sfD_\omg+ \V ) u = (\cD + \V) u, \\
& \dom ( \sfD_\omg + \V ) = \dom ( \sfD_\omg) = 
	\big\{u\in C^\infty_0(\ov{\cC_\omg}\sm \mathbf{0};\dC^4)\colon u|_{\p\cC_\omg} + \ii\beta(\aa\cdot\nu_\omg) u|_{\p\cC_\omg} = 0\big\}.
\end{split}
\end{equation*}
Moreover, if
$\nu < \frac{\pi}{4\omg}$, then $\overline{\sfD_\omg + \V} = \overline{\sfD_\omg} + \V$ is self-adjoint, with 
\begin{equation*}
\begin{split}
&( \overline{\sfD_\omg+ \V} ) u
= (\cD + \V) u, \\
& \dom(\ov{\sfD_\omg + \V}) = \dom \ov{\sfD_\omg} 
=
	\big\{u\in H^1(\cC_\omg ;\dC^4)\colon u|_{\p\cC_\omg} + \ii\beta(\aa\cdot\nu_\omg) u|_{\p\cC_\omg} = 0\big\}.
\end{split}
\end{equation*}
\end{thm}
Theorem~\ref{thm:perturbation} is worth to compare with the perturbation results for the Dirac operator in $\dR^3$. By~\cite[Thm. V 5.10]{Kato} the operator
\[
	H^1(\dR^3;\dC^4)\ni u\mapsto -\ii(\aa\cdot\nabla)u+ \dV u
\]
is self-adjoint in $L^2(\dR^3;\dC^4)$ if the Hermitian matrix-valued function $\dV\colon\dR^3\arr\dC^{4\times4}$ satisfies the condition $\sup_{x\in\dR^3}|x||\dV(x)| <\frac12$. For the Coulomb potential $\dV_{\rm C}(x) = \frac{\gamma}{|x|}I_4$ finer results are available, and, in particular, by~\cite[Thm. 2.1]{GR73} the above Dirac operator with $\dV = \dV_{\rm C}$ is self-adjoint provided that $|\gamma| <\frac{\sqrt{3}}{2}$. More detailed discussion and further references can be found in~\cite{CP18}. According to Theorem~\ref{thm:perturbation}, self-adjointness of the Dirac operator on convex cones with MIT bag boundary conditions is stable under adding
general matrix-valued potentials with stronger singularities at the origin, than in the case of the full space. In particular, we can add Coloumb potentials
(having singularity at the tip of the cone)
with arbitrary large coefficient $\gamma$ provided that the opening angle of the cone is sufficiently small.

Theorem \ref{thm:perturbation} is a consequence of the Hardy inequality in the following proposition, that we state because it has an independent interest.
Thanks to it, the proof of Theorem~\ref{thm:perturbation} descends immediately
from the Kato-Rellich theorem and the W\"{u}st theorem.
\begin{prop}\label{prop:estimate.norm.resolvent}
Let $\sfD_\omg$ be defined as 
in~\eqref{eq:Dirac_operator}. If $\omega \in (0,\pi/2)$ then
\begin{equation}\label{eq:estimate.norm.resolvent}
\int_{\cC_\omg}|(\sfD_\omg u)(x)|^2\dd x \ge \left(\frac{\pi}{4\omg}\right)^2\int_{\cC_\omg}\frac{|u(x)|^2}{|x|^2}\dd x,\qquad \text{for all}\,u\in\dom\sfD_\omg.
\end{equation}
\end{prop}
\begin{remark}
The constant in the inequality \eqref{eq:estimate.norm.resolvent}   follows from the estimate \eqref{eq:not_sharp}. Numerical evidence in Appendix~\ref{app:plots} shows that this constant is not sharp. However,
it can also be seen from numerics that the difference between the sharp constant and $\left(\frac{\pi}{4\omg}\right)^2$ tends to zero as $\omg\arr\frac{\pi}{2}$.
\end{remark}

\section{Preliminaries}\label{sec:pre}
In this section we provide preliminary material that will be used in the proofs of our main results. In Subsection~\ref{ssec:bt} we recall the concept of boundary triples and provide some basic facts related to it. Further, in Subsection~\ref{ssec:1D_Dirac} we study
	self-adjointness of a class of one-dimensional Dirac operators. Finally, in Subsection~\ref{ssec:spin-orbit} we represent the spin-orbit operator on the spherical cap in the spherical coordinates. 

\subsection{Boundary triples for adjoints of symmetric operators}\label{ssec:bt}
In this subsection we recall the concept of ordinary boundary triples. This concept is introduced in \cite{Br76, Ko75}, further developed in \eg \cite{DM91, DM95} and presented in detail in the monographs~\cite{BHdS20, S12}.
Throughout this subsection $\sfA$ denotes a densely defined symmetric operator in a Hilbert space $(\cH,(\cdot,\cdot)_\cH)$.

First, we define the notions of deficiency subspaces and deficiency indices of a symmetric operator.
\begin{dfn}
	The deficiency subspaces of the symmetric operator $\sfA$ are defined as
	\[
		\cN_\ii(\sfA) := \ker(\sfA^* -\ii)\and
		\cN_{-\ii}(\sfA) := \ker(\sfA^* +\ii).
	\]
	The deficiency indices of $\sfA$ are given by $n_\pm(\sfA) := \dim\cN_{\pm \ii}(\sfA)$.
\end{dfn}
Next, we recall the definition of the concept of a boundary triple for the adjoint of a symmetric operator.
\begin{dfn}\label{dfn:boundary_triple}
	A boundary triple for $\sfA^*$ is a triple 
	$\{\cG,\G_0,\G_1\}$ of a Hilbert space $(\cG,(\cdot,\cdot)_\cG)$ and linear mappings
	$\G_j\colon\dom\sfA^*\arr\cG$, $j=0,1$ such that
	\begin{myenum}
		\item $(\sfA^*f,g)_\cH - (f,\sfA^*g)_\cH = (\G_1 f,\G_0g)_\cG - (\G_0f,\G_1g)_\cG$
		for all $f,g\in\dom\sfA^*$;
		\item the mapping $\dom\sfA^*\ni f\mapsto (\G_0f,\G_1f)\in \cG\times\cG$ is surjective.
	\end{myenum}
\end{dfn} 
In the next proposition we provide a necessary and sufficient condition for the boundary triple for the adjoint of a symmetric operator to exist.
\begin{prop}\cite[Prop. 14.5]{S12}
	There exists a boundary triple $\{\cG,\G_0,\G_1\}$ for $\sfA^*$ if and only if the symmetric operator has equal deficiency indices. We then have $n_+(\sfA) = n_-(\sfA) = \dim\cG$.
\end{prop}

For $\sfA$ a densely defined symmetric operator in a Hilbert
space $\cH$, the knowledge of a boundary triple for the operator $\sfA^*$ allows to parametrize
its self-adjoint restrictions.  
Let $\Theta$ be a linear operator in  $\cG$.
We define the operator $\sfA_{\Theta} := \sfA^* \upharpoonright \dom
\sfA_{\Theta}$, where
\begin{equation*}
	\dom \sfA_{\Theta} 
	:= \{f \in \dom \sfA^* \colon
	\G_0 f \in \dom \Theta,  \G_1 f = \Theta \G_0 f
	\}.
\end{equation*}
\begin{prop}[{\cite[Prop. 14.7\,(v)]{S12}}]
	\label{prop:bt_sa}
	The operator $\sfA_{\Theta}$ is self-adjoint in $\cH$
	if and only if $\Theta$ is self-adjoint in $\cG$.
\end{prop}	 
\begin{remark}
	One can not parametrize all self-adjoint extensions of $\sfA$ by self-adjoint linear operators $\Theta$ as described above. In order to cover all self-adjoint extensions self-adjoint linear relations in the boundary conditions should be considered; \cf~\cite[Thm. 14.10]{S12}. This most general construction is not necessary for our purposes. 
\end{remark}
\subsection{Deficiency indices and self-adjoint extensions for a class of 1-D Dirac operators}
\label{ssec:1D_Dirac}
In this subsection we consider a model one-dimensional Dirac operator on the interval $(0,b)$ with $b\in (0,\infty]$.
We find the deficiency indices and characterise the self-adjoint extensions. In the analysis we rely on the general approach to one-dimensional Dirac operators briefly outlined in Appendix~\ref{app:Dirac}.

Let the parameter $\aa \in\dR$ be fixed. We consider the following Dirac differential expression on the interval $\cI := (0,b)$ with $b\in (0,\infty]$
\begin{equation}\label{eq:tau_alpha}
	\tau_\aa f :=\begin{pmatrix} 0 & 1\\ -1 & 0\end{pmatrix}f' + \begin{pmatrix} 0 & \frac{\aa}{x}\\
	\frac{\aa}{x} & 0\end{pmatrix} f,
\end{equation}
where $f = (f_1,f_2)^\top\colon \cI\arr\dC^2$.
The differential expression $\tau_\aa$ is of the type~\eqref{eq:differential_expression} 
with the potential
\begin{equation}\label{eq:q_alpha}
	q(x) = \begin{pmatrix} 0 & \frac{\aa}{x}\\
	\frac{\aa}{x} & 0\end{pmatrix}.
\end{equation}
Next we classify the endpoints of the interval $\cI$ for $\tau_\aa$ in the sense of Definition~\ref{dfn:singular_regular}.
The differential expression $\tau_\aa$ is singular at $x=0$ for $\aa \ne 0$, because the matrix norm of the potential $q$ in~\eqref{eq:q_alpha} has a non-integrable singularity at the origin, while $\tau_0$ is regular at $x = 0$. The differential expression $\tau_\aa$ is singular at the right endpoint $x = b$ if $b  = \infty$ and regular at $x = b$ if $b < \infty$.

In view of~\eqref{eq:maximal}, the maximal operator associated with $\tau_\aa$ is given by
\begin{equation}\label{eq:maximal_alpha}
\begin{split}
	\sfT_\aa f & := \tau_\aa f,\\
	\dom\sfT_\aa & := 
	\left\{f\in L^2(\cI;\dC^2)\colon f\in {\rm AC}_{\rm loc}(\cI;\dC^2), \tau_\aa f\in L^2(\cI;\dC^2)\right\}.
\end{split}	
\end{equation}
For $z \in\dC\sm\dR$, we define the deficiency subspaces by
\begin{equation}\label{eq:def_subspaces}
	\cN_{\aa,z} := \ker(\sfT_\aa - z).
\end{equation}
According to~\eqref{eq:preminimal} the densely defined preminimal operator is given by
\begin{equation}\label{eq:preminimal_alpha}
	\sfT_{0,\aa}'' f := \tau_\aa f,\qquad
	\dom\sfT_{0,\aa}''  := C^\infty_0(\cI;\dC^2).
\end{equation}
By Proposition~\ref{prop:domT0} the closed minimal symmetric operator $\sfT_{0,\aa} := \ov{\sfT_{0,\aa}''}$ is characterised by
\begin{equation}\label{eq:minimal_alpha}
\begin{split}
	\sfT_{0,\aa} f & = \tau_\aa f,\\
	\dom\sfT_{0,\aa} & = 
	\{f\in\dom\sfT_\aa\colon [g,f]_0 = [g,f]_b = 0,\,\, \forall\, g\in\dom\sfT_\aa\}\\
	& =
	\{f\in\dom\sfT_\aa\colon [g,f]_0 = [g,f]_b = 0,\,\, \forall\, g\in\cN_{\aa,\ii}+ \cN_{\aa,-\ii}\},
\end{split}	
\end{equation}
where the boundary values of Lagrange brackets $[g,f]_0$ and $[g,f]_b$ are defined via~\eqref{eq:bracket} and Proposition~\ref{prop:Green}.
By Proposition~\ref{prop:T0adjoint} we have
\begin{equation}\label{eq:adjoints}
	(\sfT_{0,\aa}'')^* = (\sfT_{0,\aa})^* = \sfT_\aa.
\end{equation}
Further, we characterise the deficiency indices and the self-adjoint
extensions of $\sfT_{0,\aa}''$ under different assumptions on $b$ and
$\aa$. This analysis is reminiscent of the one
in~\cite[Lem. 2.5]{LO18} and~\cite[Thms. 1.1 and 1.2]{CP18}, \cite{CP19}. We provide it in full detail for convenience of the reader and also because we need to cover some additional aspects. 
We remark that item (iii) of the below proposition will not be used in the proofs of our main results and is provided only for completeness.
\begin{prop}\label{prop:1D_Dirac}
	Let the maximal operator $\sfT_\aa$, the preminimal operator $\sfT_{0,\aa}''$
	and the minimal operator $\sfT_{0,\aa}$ be as in	\eqref{eq:maximal_alpha},~\eqref{eq:preminimal_alpha} and~\eqref{eq:minimal_alpha}, respectively.
	\begin{myenum}
	\item For $|\aa| \ge \frac12$ and $b < \infty$ the preminimal operator has deficiency indices $(1,1)$ and the triple $\{\dC,\G_0,\G_1\}$ with the well-defined mappings
	\begin{equation}\label{eq:bt}
	\G_0,\G_1\colon\dom\sfT_\aa\arr\dC,\qquad \G_0f := f_1(b),\quad \G_1f := f_2(b)
	\end{equation}
	is a boundary triple for $\sfT_\aa = (\sfT_{0,\aa}'')^*$. Moreover, any self-adjoint extension of $\sfT_{0,\aa}''$ has purely discrete spectrum.
	\item	For $|\aa| \ge \frac12 $ and $b =\infty$ the preminimal operator $\sfT_{0,\aa}''$  is essentially self-adjoint. 
	If, moreover, $|\aa| > \frac12$, then 
	$\dom \sfT_{0,\aa} = H_0^1 (\mathbb{R}_+;\mathbb{C}^2)$.	
	\item	For $\aa \in (-\frac12,\frac12)$ and $b = \infty$ the preminimal operator $\sfT_{0,\aa}''$ has deficiency indices $(1,1)$ and the deficiency subspaces are
	given by
	\[
		\cN_{\aa,\pm \ii} = \ker(\sfT_\aa \mp \ii)  =
		{\rm span}\,\left\{
		f_\aa^\pm
		\right\},\qquad 
		f_\aa^\pm(x) := \begin{pmatrix}
		x^{1/2}K_{\frac12 -\aa}(x)\\
		\mp \ii x^{1/2}K_{\frac12+\aa}(x)
		\end{pmatrix},
	\]
	where $K_\nu$ stands for the modified Bessel function of the second kind and order $\nu$.
	All self-adjoint extensions of $\sfT_{0,\aa}''$ are
        characterised by
	\begin{equation}\label{eq:extensions}
	\sfT_{\gamma,\aa}f = \tau_\aa f,
	\qquad\dom\sfT_{\gamma,\aa} =
	\left\{f + c(f_\aa^+ + \gamma f_\aa^-)\colon f\in\dom\sfT_{0,\aa},c\in\dC\right\},
	\end{equation}
	for $\gamma\in\dC$ with $|\gamma| = 1$.
	\end{myenum} 
\end{prop}
\begin{proof}
	It is clear from~\eqref{eq:maximal_alpha},~\eqref{eq:def_subspaces}, and~~\eqref{eq:adjoints} that the deficiency indices of $\sfT_{0,\aa}''$  are given by
	\[
		n_\pm(\sfT_{0,\aa}'') = \dim\cN_{\aa,\pm \ii}.
	\]
	In order to characterise $\cN_{\aa,\pm \ii}$ we need to find all square-integrable solutions of the differential equations $\tau_{\aa}f = \pm \ii f$.
	
	The differential equation
	$\tau_{\aa}f = \pm \ii f$ is equivalent to the system
	\begin{equation}\label{eq:f1f2}
		\begin{cases}
		-f_1'(x) + \frac{\aa}{x} f_1(x) = \pm\ii f_2(x), \\
		+f_2'(x) + \frac{\aa}{x} f_2(x) =  \pm\ii f_1(x).
		\end{cases}
	\end{equation}
	We express $f_2$ through $f_1$ from the first equation in~\eqref{eq:f1f2} and substitute it into the second equation 
	\begin{equation}\label{eq:f1f22}
		\begin{cases}
		f_2(x) = \pm\left(\ii f_1'(x) - \ii \frac{\aa}{x} f_1(x)\right),\\
		\ii f_1''(x) - \ii\left(\frac{\aa}{x}f_1(x)\right)' +\frac{\aa}{x}\left(\ii f_1'(x) - \ii\frac{\aa}{x}f_1(x)\right) -\ii f_1(x) =0.
		\end{cases}
	\end{equation}
	The second equation in the above system can be transformed into
	\[
	f_1''(x) + \left(\frac{\aa}{x^2}f_1(x)-\frac{\aa}{x}f_1'(x)\right) +\frac{\aa}{x}\left( f_1'(x) - \frac{\aa}{x}f_1(x)\right) - f_1(x) =0
	\]
	and then simplified as
	\begin{equation}\label{eq:f1}
		f_1''(x) - \frac{\aa(\aa-1)}{x^2}f_1(x) - f_1(x) =0.
	\end{equation}
	By making the substitution
	$f_1(x) = g(x)x^{1/2}$	
	in~\eqref{eq:f1} we find that $g$ satisfies
	the ordinary differential equation
	\begin{equation}\label{eq:ODE_g}
		x^2g''(x) + xg'(x) - \left(\left(\aa-\frac12\right)^2 + x^2\right)g(x)  =0.
	\end{equation}
	According to~\cite[\S 9.6]{AS64} there
	are exactly
	two linearly independent solutions of the above equation given by the modified Bessel functions $K_\nu(x)$ and $I_\nu(x)$ of the order 
	\[
		\nu = \left|\aa-\frac12\right|,
	\]
	and any solution of~\eqref{eq:ODE_g} is a linear combination of them.
	Recall also that the modified Bessel functions are smooth on $\dR_+$.
	
	Hence, the two linearly independent solutions of~\eqref{eq:f1} are given by 
	\begin{equation}\label{eq:f1pm}
		f^+_1(x) = x^{1/2} I_\nu(x)\and f^-_1(x) =  x^{1/2}K_\nu(x).
	\end{equation}	
	Using the first equation in~\eqref{eq:f1f22} we also find the second component of the solution to the system~\eqref{eq:f1f2} associated with $f^+_1$
	\begin{equation}\label{eq:f2p}
	\begin{split}
		f_2^+(x) 
		& = \pm\ii \left((x^{1/2} I_\nu(x))' -\frac{\aa}{x^{1/2}}I_\nu(x)\right)\\
		& = \pm \ii\left(x^{1/2} I_{\nu - 1}(x) - \frac{\nu+\aa-1/2}{x^{1/2}}I_\nu(x)\right),
		\end{split}	
	\end{equation}  
	where we used that $I_\nu'(x) = I_{\nu - 1}(x) -\frac{\nu}{x} I_\nu(x)$; \cf \cite[9.6.26]{AS64}.
	The second component associated with $f_1^-$ can be recovered from the first formula in~\eqref{eq:f1f22}
	\begin{equation}\label{eq:f2m}
	\begin{split}
	f_2^-(x) & = \pm\ii\left((x^{1/2}K_\nu(x))'
	-\frac{\aa}{x^{1/2}} K_\nu(x)\right)\\
	&=\pm\ii \left(
	-x^{1/2}K_{\nu-1}(x)-\frac{\nu +\aa -1/2}{x^{1/2}}K_\nu(x)
	 \right),
	\end{split}	
	\end{equation}
	where we used that $K_\nu'(x) = -K_{\nu-1}(x) -\frac{\nu}{x}K_\nu(x)$; \cf~\cite[9.6.26]{AS64}.
	We have found that the system of differential equations $\tau_\aa f = \pm \ii f$ has two linearly independent solutions
	$f^+ = (f_1^+,f_2^+)$ and $f^- = (f_1^-,f_2^-)$ with respective components given by~\eqref{eq:f1pm},~\eqref{eq:f2p}, and~\eqref{eq:f2m}.
	
	Now, we develop the asymptotic expansions
	of the components $f_2^+$, $f_1^-$ and $f_2^-$ in the limit $x\arr0^+$.
	For the second component of the solution $f^+$ we get using~\cite[Eq. 9.6.7 and Eq. 9.6.10]{AS64}
	and taking that $I_{-1}(x) = I_1(x)$ into account for $\aa = \frac12$
	\begin{equation}\label{eq:f2p_asymp}
		f_2^+(x)\sim_{x\arr 0^+} 
		\begin{cases}
		\pm\frac{\ii x^{\nu+3/2}}{\nu(\nu+1)2^{\nu+1}\G(\nu)}, & \aa > \frac12,\\
		\pm \frac{\ii x^{3/2}}{2} , & \aa = \frac12, \\
		\pm  \frac{\ii x^{\nu-1/2}}{2^{\nu-1}\G(\nu)}, & \aa < \frac12.
	\end{cases}	
	\end{equation}
	For the first component of the solution $f^-$ we find using the asymptotics $K_\nu$ as $x\arr 0^+$
	(see \cite[9.6.8, 9.6.9]{AS64})  we find that
	\begin{equation}\label{eq:f1m_asymp}
	f_1^-(x) \sim_{x\arr 0^+}
	\begin{cases}
	 -x^{1/2}\ln(x),& \aa = \tfrac12,\\
	 2^{\nu-1}\Gamma(\nu)x^{1/2-\nu},\quad &\aa \ne \tfrac12.
	\end{cases}	
	\end{equation} 	 
	Finally, for the second component of the solution $f^-$ we find again using
	the asymptotics $K_\nu$ as $x\arr 0^+$
	that
	\begin{equation}\label{eq:f2m_asymp}
		f_2^-(x) \sim_{x\arr 0^+} 
		\begin{cases}
		\pm (-\ii) 2^\nu\Gamma(\nu+1)x^{-\nu-\frac12}, &\aa > \tfrac12,\\
		\pm (-\ii) 2^{-\nu}\Gamma(1-\nu) x^{\nu-\frac12},
		&\aa \in\big(-\tfrac12, \tfrac12\big],\\
		\pm\ii x^{1/2}\ln(x),& \aa = -\tfrac12,\\
		\pm (-\ii) 2^{\nu- 2} \Gamma(\nu - 1)x^{3/2-\nu},& \aa <-\tfrac12.
	\end{cases}	
	\end{equation}
	The remaining part of the proof is split into the analysis of the three cases 
	\begin{myenum}
	\item $b < \infty$, $| \aa | \ge \frac12$.
	\item $b = \infty$, $|\aa| \ge \frac12$.	  
	\item $b = \infty$, $\aa \in (-\frac12,\frac12)$.
	\end{myenum}
	\noindent \underline{\emph{The case $b < \infty$ and $|\aa| \ge \tfrac12$.}}
	Since $I_\nu$ is a bounded function on $\ov\cI$, we obtain that $f^+_1\in L^2(\cI)$. The asymptotics of $f_2^+$ in~\eqref{eq:f2p_asymp} yields that 	  $f_2^+\in L^2(\cI)$ and hence $f^+\in L^2(\cI;\dC^2)$. On the other hand, according to the asymptotics~\eqref{eq:f2m_asymp} we have $f_2^-\notin L^2(\cI)$ for $\aa \ge \frac12$ and according to~\eqref{eq:f1m_asymp} we have $f_1^-\notin L^2(\cI)$ for $\aa \le -\frac12$. Hence, $f^-\notin L^2(\cI;\dC^2)$ and we conclude that the deficiency indices of $\sfT_{0,\aa}''$ are $(1,1)$. 
	
	Notice that the solution
	$f^+$ lies left and right in $L^2(\cI;\dC^2)$ while the solution $f^-$ lies right in $L^2(\cI;\dC^2)$ and does not lie left in $L^2(\cI;\dC^2)$ in the sense of Definition~\ref{dfn:leftright}. Hence, we conclude from Proposition~\ref{prop:Weyl} and Definition~\ref{dfn:lplc} that $\tau_\aa$ is in this case limit-point at $x = 0$ and limit-circle at $x = b$. 
	Moreover, recall that $\tau_\aa$ is also regular at the
        endpoint $x = b$. Hence, the Green formula in
        Proposition~\ref{prop:Green}, combined with
        Proposition~\ref{prop:regular} and
        Proposition~\ref{prop:lplc}\,(i) yield that
        $\{\dC,\G_0,\G_1\}$ defined as in~\eqref{eq:bt} is a boundary
        triple for $\sfT_\aa$ in the sense of
        Definition~\ref{dfn:boundary_triple}.

        We turn now to show that any self-adjoint extension of $\sfT_{0,\aa}''$ has purely
          discrete spectrum.
        To do so, we show that  $\dom \sfT_\aa \subset
        H^1(\cI;\dC^2)$ in the case that $|\alpha|>1/2$ and that $\dom \sfT_\aa \subset
        H^{\frac{1}{2}}(\cI;\dC^2)$ in the case that $|\alpha|=1/2$: these give immediately the claim thanks to the compactness of the embedding of $H^1(\cI;\dC^2)$
        	and of $H^\frac12(\cI;\dC^2)$ into $L^2(\cI;\dC^2)$.
        Let $f=(f_1,f_2)^\top \in \dom \sfT_\aa$: from \eqref{eq:maximal_alpha} we
        have that $f_1,f_2 \in L^2(\cI)$, $f_1,f_2 \in {\rm AC}_{\rm
          loc}(\cI)$, $f_1'- \tfrac{\alpha}{x}f_1, f_2'+ \tfrac{\alpha}{x}f_2
        \in L^2(\cI)$. Let $\xi \in C_0^{\infty}([0,b))$ be such that $\xi(x)=1 $ if 
        $x \in [0,b/2]$ and for $j=1,2$ define the function $g_j\colon\dR_+\arr\dC^2$ by
        \begin{equation*}
          g_j(x):= \begin{cases}
            f_j(x) \xi(x), \quad &\text{ if }x\in \cI,\\ 
            0, \quad &\text{otherwise}.
          \end{cases}
        \end{equation*}
        It is clear that $g_j \in L^2(\dR_+), g_j \in {\rm AC}_{\rm loc}(\dR_+)$
        and $g_1'- \tfrac{\alpha}{x}g_1, g_2'+ \tfrac{\alpha}{x}g_2
        \in L^2(\dR_+)$.
        In the following we assume without
        loss of generality that $\alpha \geq 1/2$, since
        the proof in the case that $\alpha \leq -1/2$ can be done analogously switching the
        roles of $f_1$ and $f_2$. 
        We observe that
        \begin{equation}\label{eq:conditions.maximal.radial}
          g_1' - \frac{\alpha}{x}g_1 =
          x^{\alpha} (x^{-\alpha} g_1)' \in L^2(\dR_+),
          \qquad
          g_2' + \frac{\alpha}{x}g_2 =
          x^{-\alpha} (x^{\alpha} g_2)' \in L^2(\dR_+).
        \end{equation}
        From \cite[Prop. 2.2 (i)]{CP18} and \cite[Prop. 2.4 (i)]{CP18}, with $a=-\alpha \leq -1/2<1/2$, there exists a constant $A_2 \in \dC$ such
        that
        \begin{align}
          \label{eq:behaviour.in.the.origin}
        &g_2(x) = \frac{A_2}{x^{\alpha}} + o(\sqrt{x}) \quad \text{ as }x\arr
          0^+,\\
          \label{eq:behaviour.everywhere}
        &\frac{g_2}{x} - \frac{A_2}{x^{1+\alpha}} \in L^2(\dR_+).
        \end{align}
        Since $\alpha \geq 1/2$ and $g_2 \in L^2(\dR_+)$, \eqref{eq:behaviour.in.the.origin}
        gives $A_2 = 0$ and \eqref{eq:behaviour.everywhere} implies
        that ${g_2}/{x}\in L^2(\dR_+)$.
        We now distinguish two cases.
        
	\noindent {\emph{The case $\alpha>1/2$.}}
        Thanks to the Hardy inequality \cite[Prop. 2.4\,(ii)]{CP18}, with $a=\alpha>1/2$, we have ${g_1}/{x} \in L^2(\dR_+)$.
        We finally conclude that
        ${g}/{x} \in L^2(\cI;\dC^2)$, that gives ${f|_{(0,b/2)}}/{x}\in
        L^2((0,b/2);\dC^2)$. Moreover, since  $f|_{(b/2,b)} \in
        L^2((b/2,b);\dC^2)$ and ${1}/{x}$ is bounded with bounded inverse for $x
        \in [b/2, b]$, we conclude that ${f}/{x}\in
        L^2(\cI;\dC^2)$.
        From \eqref{eq:maximal_alpha} we have $f' \in
        L^2(\cI;\dC^2)$, that is the claim.

\noindent {\emph{The case $\alpha = 1/2$.}}
Thanks to the Hardy inequality
\cite[Prop. 2.4 (iii)]{CP18}, with $R = b$,  we
have that $g_{1}/(x \ln(b/x)) \in L^2(\dR_+)$, that implies that 
$g_{1}/\sqrt{x} \in
L^2(\dR_+)$. Since $g_2/x \in L^2(\dR_+)$ implies $g_2/\sqrt{x} \in
L^2(\dR_+)$, we conclude that $g/\sqrt{x} \in
L^2(\dR_+;\dC^2)$ and so $f / \sqrt{x} \in
L^2(\cI;\dC^2)$. From \eqref{eq:maximal_alpha}, we conclude that
$\sqrt{x}f' \in L^2(\cI;\dC^2)$.
Let us consider the Lipschitz domain 
\[
	\Omega := \big\{x = (x_1,x_2) \in \dR^2 \colon |x| <\omega, x_1>0,x_2 > 0 \big\}
\]
and the function
$F \in L^2(\Omega;\dC^2)$ defined as $F(x) =f(|x|)$. Since $\sqrt{x}f' \in
L^2(\cI;\dC^2)$, 
we have $F\in H^1(\Omega;\dC^2)$ and
by the trace theorem~\cite[Thm. 3.38]{McL} for Lipschitz domains
$F|_{\p\Omega
  \cap \{x_2 = 0\}} \in H^{\frac12}(\p\Omega\cap \{x_2 = 0\};\dC^2)$. Since this trace can be identified with $f$, we get $f \in H^{\frac12}(\cI;\dC^2)$, that is the claim.

	\medskip
	\noindent \underline{\emph{The case $b = \infty$ and $|\aa| \ge \tfrac12$.}}
	Taking the asymptotics~\cite[Eq.~9.7.1]{AS64} of $I_\nu(x)$ as $x \arr \infty$ into account we obtain that $f^+_1\notin L^2(\cI)$.
	As in the analysis of the previous case
	we derive from the asymptotics
	~\eqref{eq:f1m_asymp} and~\eqref{eq:f2m_asymp} that $f^-\notin L^2(\cI;\dC^2)$ for $|\aa| \ge \frac12$.
	Hence, we obtain that the operator $\sfT_{0,\aa}''$ is essentially self-adjoint in this case.
	If $|\aa| > \frac12$,
	it follows from \cite[Lem. A.1]{FL21} and its proof that $\dom\sfT_{0,\aa} = H^1_0(\dR_+,\dC^2)$, so the claim of (ii) is completely proven.
	
	\medskip
	
	\noindent \underline{\emph{The case $b = \infty$ and $\aa \in(-\tfrac12,\tfrac12)$.}}
	As in the analysis of the previous case
	the asymptotics of $I_\nu(x)$ as $x \arr \infty$ yields that $f^+_1\notin L^2(\cI)$
	and hence $f^+\notin L^2(\cI;\dC^2)$.
	On the other hand, we derive from the asymptotics~\eqref{eq:f1m_asymp} and~\eqref{eq:f2m_asymp} that $f^-\in L^2(\cI;\dC^2)$ for $\aa\in (-\frac12, \frac12)$.
	Hence, we obtain that the operator $\sfT_{0,\aa}''$ has deficiency indices $(1,1)$.
	
	The defect subspaces can be characterised as
	\[
		\cN_{\aa,\pm\ii} = \spn\left\{
		\begin{pmatrix}
		x^{1/2}K_{\frac12-\aa}(x)\\
		\mp\ii x^{1/2} K_{\frac12+\aa}(x)
		\end{pmatrix}
		\right\}.
	\]
	In view of von Neumann's extensions theory~\cite[Thm. 13.10]{S12} self-adjoint extensions of $\sfT_{0,\aa}''$ are parametrized as in~\eqref{eq:extensions} and thus the claim of (iii) is also proven.
\end{proof}

\subsection{The spin-orbit differential expression in spherical coordinates}
\label{ssec:spin-orbit}
The main goal of this subsection is to compute the representation of the spin-orbit differential expression in the natural coordinates on the unit sphere. The spin-orbit differential expression appears in the representation of the Dirac differential expression $\cD$ in the spherical coordinates; see \cite[\S 20.3]{W03} and~\cite[\S 4.6]{T92}. 	

To this aim we first find that the components of the orbital angular momentum differential expression
\[	\cL = -\ii{\bf x}\times \nabla = (\cL_1,\cL_2,\cL_3)
\]
are given by
\[
	\cL_1 = -\ii(x_2\p_3 - x_3\p_2),\qquad \cL_2 = -\ii(x_3\p_1 - x_1\p_3),\qquad \cL_3 = -\ii(x_1\p_2 - x_2\p_1).
\]
According to Appendix~\ref{app:pderiv} the expressions for $\p_1,\p_2,\p_3$ in terms of $\p_r,\p_\tt,\p_\phi$ are given by
\begin{equation}\label{eq:p123}
	\begin{cases}
	\p_1 = \cos\phi\sin\tt \p_r
	+
	\frac{\cos\phi\cos\tt\p_\tt}{r}  
	- 
	\frac{\sin\phi}{\sin\tt}\frac{\p_\phi}{ r},
	\\
	\p_2 = \sin\phi\sin\tt\p_r + \frac{\sin\phi\cos\tt\p_\tt}{r} + \frac{\cos\phi}{\sin\tt}\frac{\p_\phi}{r},
	\\
	\p_3 = \cos\tt\p_r -\frac{\sin\tt\p_\tt}{r}.
	\end{cases}
\end{equation}
The next step is to express $\cL_1, \cL_2, \cL_3$ in terms of $\p_\tt$
and $\p_\phi$. In order to compute $\cL_1$ we combine the expression
for $x_2$, $x_3$ in~\eqref{eq:x123} with the formulae for $\p_2,\p_3$
in~\eqref{eq:p123}
\[
\begin{aligned}
	\cL_1 & = -\ii r\sin\phi\sin\tt
	\left(\cos\tt\p_r -\frac{\sin\tt\p_\tt}{r}\right)\\
	& \qquad +\ii r\cos\tt\left(\sin\phi\sin\tt\p_r + \frac{\sin\phi\cos\tt\p_\tt}{r} + \frac{\cos\phi}{\sin\tt}\frac{\p_\phi}{r}\right)
	\\
	& = 
	\ii\left(\sin\phi\sin^2\tt+\sin\phi\cos^2\tt\right)\p_\tt+
	\ii \cot\tt \cos\phi\p_\phi  = \ii\sin\phi\p_\tt + \ii\cot\tt\cos\phi\p_\phi.
\end{aligned}
\]
To compute $\cL_2$ we combine the expression for $x_1$, $x_3$ in~\eqref{eq:x123} with the formulae for $\p_1,\p_3$ in~\eqref{eq:p123}
\[
\begin{aligned}
	\cL_2 & = -\ii r\cos\tt\left(	
	\cos\phi\sin\tt \p_r
	+ \frac{\cos\phi\cos\tt\p_\tt}{r}  
		- \frac{\sin\phi}{\sin\tt }\frac{\p_\phi}{r} 
	\right)\\
		& \qquad +\ii r\cos\phi\sin\tt	\left(\cos\tt\p_r -\frac{\sin\tt\p_\tt}{r}\right)\\
	& = 
	-\ii\left(\cos\phi\cos^2\tt + \cos\phi\sin^2\tt\right)\p_\tt+
	\ii\cot\tt\sin\phi\p_\phi	= -\ii\cos\phi\p_\tt+  \ii\cot\tt\sin\phi\p_\phi.	
\end{aligned}	
\]
For the derivation of $\cL_3$ we combine the expression for $x_1$, $x_2$ in~\eqref{eq:x123} with the formulae for $\p_1,\p_2$ in~\eqref{eq:p123}
\[
\begin{aligned}
\cL_3 & = -\ii (x_1 \p_2 - x_2 \p_1) = -\ii r\cos\phi\sin\tt\left(\sin\phi\sin\tt\p_r + \frac{\sin\phi\cos\tt\p_\tt}{r} + \frac{\cos\phi}{\sin\tt}\frac{\p_\phi}{r}\right)\\
& \qquad +\ii r\sin\phi\sin\tt\left(\cos\phi\sin\tt \p_r
+\frac{\cos\phi\cos\tt\p_\tt}{r} - \frac{\sin\phi}{\sin\tt }\frac{\p_\phi}{r} 
\right) \\
& = \ii\left(-\sin\phi\cos\phi\sin\tt\cos\tt +\sin\phi\cos\phi\sin\tt\cos\tt\right)\p_\tt
+\ii\left(-\cos^2\phi-\sin^2\phi\right)\p_\phi = -\ii\p_\phi.
\end{aligned}	
\]
Using the expressions for $\cL_j$, $j=1,2,3$ we compute the auxiliary differential expression  
\[
\begin{aligned}
	\cS :=\s_1\cL_1 +\s_2\cL_2+\s_3\cL_3
		=\begin{pmatrix}-\ii\p_\phi & 
		-e^{-\ii\phi}\p_\tt + \ii \cot\tt e^{-\ii\phi} \p_\phi    \\
		e^{\ii\phi}\p_\tt + \ii\cot\tt e^{\ii\phi}\p_\phi  & \ii\p_\phi \end{pmatrix}.
\end{aligned}	
\]
Finally, the \emph{spin-orbit differential expression} is given by
\begin{equation}\label{eq:spin_orbit}
	\cK = \sfI_4 + \begin{pmatrix} \cS &0\\ 0 & \cS\end{pmatrix},
\end{equation}
where $\sfI_4$ is the identity matrix in $\dC^4$.

For further analysis we introduce also the matrix-valued function
\begin{equation}\label{eq:alpha_r}
	\aa_r := \sum_{j=1}^3 \aa_j \frac{x_j}{r} = 
	\begin{pmatrix} 0 & 0&  \cos\tt & e^{-\ii\phi}\sin\tt\\
	0&0 &e^{\ii\phi}\sin\tt	&		-\cos\tt\\
	\cos\tt & e^{-\ii\phi}\sin\tt &0 &0\\
	e^{\ii\phi}\sin\tt	&		-\cos\tt&0&0
	\end{pmatrix}.
\end{equation}
\section{Spectral analysis of the one-dimensional model Dirac-type operator}
\label{sec:Dirac_model}
In this section we consider a family of  one-dimensional Dirac operators on the interval $\cI_\omg = (0,\omg)$, $\omg \in (0,\pi)\sm \{\frac{\pi}{2}\}$. These operators arise in the orthogonal decomposition of the spin-orbit operator on the spherical cap. We show that these Dirac-type operators are self-adjoint in
the Hilbert space $L^2(\cI_\omg;\dC^4)$ and that their spectrum is discrete. Furthermore, we explicitly find their eigenvalues as solutions of certain transcendental equations and characterise
the respective associated eigenfuctions. Finally, we obtain estimates on the size of the spectral gaps of these model Dirac-type operators.

First, we recall the definitions of the
$2\times2$ matrix differential expression on the interval $\cI_\omg$  
\begin{equation}\label{key}
	\tau_k 
	= 
	\begin{pmatrix} 
	k + 1 &-\frac{\dd}{\dd\tt}
			-\left(k+\frac12\right)\cot\tt\\
	\frac{\dd}{\dd\tt} - \left(k+\frac12\right)\cot\tt & -k
	\end{pmatrix},\qquad k\in\dZ,		
\end{equation}
and the associated $4\times 4$ matrix differential expression on $\cI_\omg$ 
\[
	\cT_k = 
	\begin{pmatrix} \tau_k & 0\\
	0 & \tau_k\end{pmatrix}.
\]
Recall also that $2\times 2$ unitary matrix
$\sfA_\omg$ is defined by
\begin{equation}\label{eq:matrix_A}
	\sfA_\omg = 
	\begin{pmatrix}  
		\ii\sin\omg     & -\ii\cos\omg  \\
		-\ii\cos\omg    & -\ii\sin\omg
	\end{pmatrix},
\end{equation}
and that the coupling of two Dirac operators on the interval $\cI_\omg$ 
is introduced as
\begin{equation}\label{eq:Dk}
\begin{aligned}
	\sfT_k \psi & =\cT_k\psi,\\
	\dom\sfT_k &
	\!= \!
	\left\{
	\psi \!\in\! {\rm AC}_{\rm loc}(\cI_\omg;\dC^4)\colon\!
	\psi,
	\cT_k \psi\in L^2(\cI_\omg;\dC^4),
	\left(\begin{smallmatrix}
		\psi_1(\omg)
		\\
		\psi_2(\omg)
	\end{smallmatrix}\right)\!=\! 
	\sfA_\omg	\left(\begin{smallmatrix}\psi_3(\omg)
		\\
		\psi_4(\omg)
	\end{smallmatrix}\right)
\right\}\!.
\end{aligned}
\end{equation}
In the first proposition of this section we establish the self-adjointness of $\sfT_k$
and show that the spectrum of $\sfT_k$ is purely discrete.
\begin{prop}\label{prop:spectrum.Tk}
	The operator $\sfT_k$, $k\in\dZ$, in~\eqref{eq:Dk} is self-adjoint in $L^2(\cI_\omg;\dC^4)$, and the restriction of $\sfT_k$ to the subspace of $\dom\sfT_k$ 
	\begin{equation}\label{eq:cD}
		\cM := \left\{
		\psi\in C^\infty_0((0,\omg];\dC^4)\colon
		\left(\begin{smallmatrix}
			\psi_1(\omg)
			\\
			\psi_2(\omg)
		\end{smallmatrix}\right)\!=\! 
		\sfA_\omg	\left(\begin{smallmatrix}\psi_3(\omg)
			\\
			\psi_4(\omg)
		\end{smallmatrix}\right)
		\right\}
	\end{equation}
	is essentially self-adjoint. Finally, the spectrum of $\sfT_k$
        is purely discrete.
      \end{prop}
\begin{proof}
	The proof of self-adjointness of $\sfT_k$ is based on a perturbation argument. We decompose $\sfT_k$ into a sum of the Dirac-type operator whose self-adjointness is established by the methods of Subsection~\ref{ssec:1D_Dirac} and a bounded self-adjoint perturbation.
	
	To this aim, we consider the auxiliary differential expression
	\[
		\check{\tau}_k = 
		\begin{pmatrix} 
		0 & 
		\frac{\dd}{\dd\tt}+
			\left(k+\frac12\right)\frac{1}{\tt}\\
		-\frac{\dd}{\dd\tt}+ 	
			\left(k+\frac12\right)\frac{1}{\tt}&
		0\end{pmatrix}.
	\]
	This differential expression is of the type~\eqref{eq:tau_alpha} with $\alpha = k+\frac12$. In particular, we immediately observe that $|\aa| \ge \frac12$ for all $k\in\dZ$.
	
	Let us associate the symmetric preminimal operator to $\check{\tau}_k$ as in~\eqref{eq:preminimal_alpha}
	\[
		\check{\sfT}_{0,k}''\psi 
		:=
		\check{\tau}_k\psi,
		\qquad 
		\dom\check{\sfT}_{0,k}'' 
		:= 
		C^\infty_0(\cI_\omg;\dC^2).
	\]
	According to Subsection~\ref{ssec:1D_Dirac} the adjoint of $\check{\sfT}_k := (\check{\sfT}_{0,k}'')^*$ is given by
	\begin{equation}\label{key}
		\check{\sfT}_k\psi 
		:=
		\check{\tau}_k\psi,
		\qquad
		\dom\check{\sfT}_k 
		:= 
		\big\{\psi\colon 
			\psi,\check{\tau}_k\psi\in L^2(\cI_\omg;\dC^2), \psi\in {\rm AC}_{\rm loc}(\cI_\omg;\dC^2)
		\big\}.		
	\end{equation}
	By Proposition~\ref{prop:1D_Dirac}\,(i), the deficiency indices of $\check{\sfT}_{0,k}''$ are $(1,1)$ and $\{\dC,\check{\G}_0,\check{\G}_1\}$ with $\check{\G}_0,\check{\G}_1\colon\dom\check{\sfT}_k\arr\dC$, $\check{\G}_0\psi := \psi_1(\omg)$, $\check{\G}_1\psi := \psi_2(\omg)$ is a boundary triple for its adjoint $\check{\sfT}_k$. 
	
	Further, consider the auxiliary symmetric operator 
	\begin{equation}\label{eq:auxiliary_symmetric}
	\wh{\sfT}_{0,k}'' = \check{\sfT}_{0,k}''\oplus \check{\sfT}_{0,k}''
	\end{equation}
	acting in the Hilbert space $L^2(\cI_\omg;\dC^4)$. The adjoint $\wh{\sfT}_k$ of $\wh{\sfT}_{0,k}''$ is given by 
	$\wh{\sfT}_k = \check{\sfT}_k\oplus\check{\sfT}_k$.
	It is easy to see that the deficiency indices
	of $\wh{\sfT}_{0,k}''$ are $(2,2)$ and that 
	$\{\dC^2,\wh{\G}_0,\wh{\G}_1\}$ with the mappings $\wh{\G}_0,\wh{\G}_1\colon\dom\wh{\sfT}_k \arr\dC^2$ given by
	\[
		\wh{\G}_0\psi := \begin{pmatrix}\psi_1(\omg)\\ \psi_3(\omg)\end{pmatrix}\and\wh{\G}_1\psi:= \begin{pmatrix} \psi_2(\omg)\\ \psi_4(\omg)\end{pmatrix}
	\]
	is a boundary triple for $\wh{\sfT}_k$ in the sense of Definition~\ref{dfn:boundary_triple}.
	Let the Hermitian matrix $\wh{\sfA}_\omg$ be defined by
	\[
		\wh{\sfA}_\omg := \begin{pmatrix}
		\tan\omg & -\frac{\ii}{\cos\omg}\\
		\frac{\ii}{\cos\omg} & \tan\omg\end{pmatrix}.
	\]
	By Proposition~\ref{prop:bt_sa} the operator 
	\[
		\wt{\sfT}_{k}\psi := \wh{\sfT}_k\psi,\qquad\dom\wt{\sfT}_k := \big\{\psi\in\dom\wh{\sfT}_k\colon \wh{\G}_1\psi = \wh{\sfA}_\omg\wh{\G}_0\psi\big\},
	\]
	is self-adjoint in the Hilbert space $L^2(\cI_\omg;\dC^4)$.
	
	Let us define the multiplication operator
	by a matrix-valued function in the Hilbert space $L^2(\cI_\omg;\dC^2)$
	\[
		\sfB_k\psi = 
		\begin{pmatrix} 
		k+1 & -\left(k+\frac12\right)\left(\cot\tt-\frac{1}{\tt}\right)\\
		-\left(k+\frac12\right)\left(\cot\tt-\frac{1}{\tt}\right) & -k
		\end{pmatrix}\psi.	
	\]
	It is not difficult to see that
	$\sfB_k$ is symmetric and since the entries of $\sfB_k$ are all bounded functions, the operator $\sfB_k$ is bounded in the Hilbert space $L^2(\cI_\omg;\dC^2)$. 

	It remains to notice that 
	\begin{equation}\label{eq:TwhT_k}
		\sfT_k = 
		-\wt{\sfT}_k + 
		\begin{pmatrix}\sfB_k &0\\ 0&\sfB_k\end{pmatrix}
	\end{equation}
	and since $\wt{\sfT}_k$ is self-adjoint and the perturbation $\begin{pmatrix}\sfB_k &0\\ 0&\sfB_k\end{pmatrix}$ is bounded and self-adjoint, the operator $\sfT_k$ is self-adjoint as well.
	
	The operator $\wt{\sfT}_k$ can be viewed as a finite-rank perturbation in the sense of resolvent differences of an orthogonal sum of two self-adjoint extensions of the symmetric operator $\check{\sfT}_k$. Since
	any self-adjoint extensions of $\check{\sfT}_k$ has purely discrete spectrum by Proposition~\ref{prop:1D_Dirac}\,(i), we get
	that $\wt{\sfT}_k$ has purely discrete spectrum as well. Hence, using the representation~\eqref{eq:TwhT_k} we conclude that $\sfT_k$ also has purely discrete spectrum.
	
	In view of the decomposition~\eqref{eq:TwhT_k}, in order to show that the restriction of $\sfT_k$ to the subspace $\cM\subset\dom\sfT_k$ defined in~\eqref{eq:cD} is essentially self-adjoint, it suffices to check that the densely defined operator $\sfS := \wt\sfT_k\uhr\cM$ is essentially self-adjoint. Since $\sfS$ is a restriction of the self-adjoint operator $\wt{\sfT}_{k}$, we conclude that $\sfS$ is symmetric. Moreover, the operator $\sfS$ is an extension of the symmetric operator $\wh{\sfT}_{0,k}''$ defined in~\eqref{eq:auxiliary_symmetric}. Hence, the adjoint $\sfS^*$
	of $\sfS$ is a restriction of $\wh{\sfT}_k$.
	Moreover, for any $\psi\in\dom\sfS$ and 
	$\phi\in\dom\wh{\sfT}_k$ we find via integration by parts that
	\[
	\begin{split}
		&(\sfS \psi,\phi)_{L^2(\cI_\omg;\dC^4)}
		- (\psi,\wh{\sfT}_k\phi)_{L^2(\cI_\omg;\dC^4)}\\
		&\qquad= 
		\int_0^\omg\left(
		\psi_2'\ov{\phi_1} - \psi_1'\ov{\phi_2}
		+\psi_4'\ov{\phi_3} - \psi_3'\ov{\phi_4}
		+\psi_2\ov{\phi_1'} - \psi_1\ov{\phi_2'}
		+\psi_4\ov{\phi_3'} - \psi_3\ov{\phi_4'}
		\right)\dd \tt\\
		&\qquad =
		\psi_2(\omg)\ov{\phi_1(\omg)} 
		-\psi_1(\omg)\ov{\phi_2(\omg)}
		+\psi_4(\omg)\ov{\phi_3(\omg)} - \psi_3(\omg)\ov{\phi_4(\omg)}\\
		&\qquad= -\left(
		\begin{pmatrix}
		\psi_1(\omg)\\
		\psi_3(\omg)
		\end{pmatrix},
		\begin{pmatrix}\phi_2(\omg)\\ \phi_4(\omg)
				\end{pmatrix}
		-\wh\sfA_\omg\begin{pmatrix}\phi_1(\omg)\\ \phi_3(\omg)
		\end{pmatrix}
		\right)_{\dC^2}.
	\end{split}	
	\]
	From the above computation we conclude that the adjoint of $\sfS$
	is characterised by
	\[
		\sfS^*\psi = \wh{\sfT}_k\psi,\qquad \dom\sfS^* = \big\{\psi\in\dom\wh{\sfT}_k\colon\wh{\G}_1\psi = \wh{\sfA}_\omg\wh{\G}_0\psi\big\}.
	\]
	Hence, $\sfS^*$ coincides with $\wt{\sfT}_k$ and it is thus self-adjoint.
	Therefore, the symmetric operator $\sfS$ is essentially self-adjoint.
      \end{proof}
       Further we make use of the following Hardy inequality.
\begin{lem}\label{lem:hardy.for.us}
  Let $\omega \in (0, \pi/2]$.   
For all $f \in C_0^{\infty}((0,\omega])$ the following inequality
holds true:
\begin{equation}\label{eq:hardy.for.us}
  \int_0^{\omega} |f'(\theta)|^2 \,\dd\theta \geq
  \frac14 \int_0^{\omega} \frac{|f(\theta)|^2}{\sin^2 \theta}
  \,\dd\theta
  +
  \frac{\pi^2}{16\omega^2} \int_0^{\omega} |f(\theta)|^2 \,\dd\theta.
\end{equation}
\end{lem}
\begin{proof}
To show \eqref{eq:hardy.for.us} we exploit
the following Hardy inequality: for all $g \in H_0^1((0,\pi))$
\begin{equation}\label{eq:hardy.sin}
  \int_0^{\pi} |g'(\theta)|^2 \,\dd\theta
  \geq  \frac{1}{4}\int_0^{\pi} \frac{|g(\theta)|^2}{\sin^2 \theta}\,\dd\theta
  + \frac{1}{4}\int_0^{\pi} |g(\theta)|^2 \,\dd\theta.
\end{equation}
Such inequality can be derived from the inequality
\begin{equation*}
  0 \leq \int_0^{\pi} \Big\vert g'(\theta) -
    \frac{\cos\theta}{2\sin{\theta}}g(\theta)\Big\vert^2 \, \dd\theta
\end{equation*}
expanding the square and integrating by parts: for a detailed proof 
and interesting details on how this inequality is related to
 Bessel-type operators we refer to \cite{GPS21}.
Let $f \in C_0^{\infty}((0,\omega])$ and define
\begin{equation*}
  \widetilde f \in H_0^1((0,\pi)), \qquad
  \widetilde f(\theta) :=
  \begin{cases}
    f(\frac{2\omega}{\pi}\theta), \quad& \theta\in (0,\frac{\pi}{2}],\\
    f(\frac{2\omega}{\pi}(\pi-\theta)), \quad& \theta\in (\frac{\pi}{2},\pi).
  \end{cases}
\end{equation*}
Since $\widetilde f$ is symmetric in $(0,\pi)$ with respect to the
point $\pi/2$, from \eqref{eq:hardy.sin} we have that
\begin{equation}\label{eq:hardy.ontheway}
  \int_0^{\frac{\pi}2}|\widetilde f'(\theta)|^{2}\,\dd\theta
  \geq  \frac{1}{4}\int_0^{\frac{\pi}2} \frac{|\widetilde f(\theta)|^2}{\sin^2 \theta}\,\dd\theta
  + \frac{1}{4}\int_0^{\frac{\pi}2} |\widetilde f(\theta)|^2 \,\dd\theta.
\end{equation}
After a change of variables, \eqref{eq:hardy.ontheway} gives
\begin{equation}\label{eq:almost.final}
  \int_0^\omega |f'(\theta)|^{2}\,\dd\theta
  \geq
  \frac{1}{4} \int_0^\omega \frac{(\frac{\pi}{2\omega})^2}{\sin^2
    {\frac{\pi \theta}{2\omega}}} |f(\theta)|^2\,\dd\theta
  +
  \frac{1}{4}  \left(\frac{\pi}{2\omega}\right)^2\int_0^\omega |f(\theta)|^2\,\dd\theta.
\end{equation}
In order to conclude \eqref{eq:hardy.for.us} from
\eqref{eq:almost.final}
we show that
\begin{equation}\label{eq:hard(y).inequality}
  \frac{(\frac{\pi}{2\omega})^2}{\sin^2{\frac{\pi \theta}{2\omega}}}
  \geq \frac{1}{\sin^2 \theta}, \quad \text{ for all }\theta \in (0,\omega).
\end{equation}
Since $\theta \in (0,\omega) \subset (0,\pi/2)$, this is equivalent to
show that the functions
\begin{equation}
  c_\theta(\omega):= \sin \theta - \frac{\sin
    (\frac{\pi}{2\omega}\theta)}{\frac{\pi}{2\omega}}, \quad
     \quad \text{ for }\omega \in
  \left(0,\frac{\pi}2\right]
\end{equation}
are non-negative for all $\theta \in (0,\omega)$. We have that $c_\theta(\pi/2) =0$ and
$\frac{\dd}{\dd\omega}c_\theta(\omega) \leq 0$ if and only if
\begin{equation*}
  \frac{\pi}{2\omega}\theta \leq \tan
  \left(\frac{\pi}{2\omega}\theta\right),
   \quad \text{ for }\omega \in
  \left(0,\frac{\pi}2\right], \, \theta \in (0,\omega),
\end{equation*}
that is true since $\tan 0 =0$ and $\frac{\dd}{\dd\varphi}(\tan \varphi)
\geq 1$ for $\varphi \in (0,\pi/2)$.
This shows \eqref{eq:hard(y).inequality} and completes the proof.
\end{proof}

\begin{prop}\label{prop:spectrum.outside}
For all $k \in \ZZ$ and $\omega<\pi/2$
\begin{equation}\label{eq:estimate.below.final}
\norm*{\left(\sfT_k -\frac12\right) \psi}_{L^2(\cI_\omg;\dC^4)} \geq
\frac{\pi}{4\omg}\norm{\psi}_{L^2(\cI_\omg;\dC^4)}, \quad \text{ for all }
\psi \in \cM.
\end{equation}
\end{prop}
\begin{proof}
Let $\psi = (\psi_1,\psi_2,\psi_3,\psi_4)^\top \in \cM$
and for convenience let us denote $\alpha := k+1/2 \in (\mathbb{Z} + 1/2)$.
From \eqref{eq:Dk} we have that
\begin{equation}\label{eq:quadratic.form.1}
  \begin{split}
    &\norm*{\left(\sfT_k -\tfrac{1}{2}\right) \psi}_{L^2(\cI_\omg;\dC^4)}^2
    \!=\! \norm*{ \left(\tau_{k} -\tfrac{1}{2}\right)
            \begin{pmatrix}
        \psi_1 \\ \psi_2
      \end{pmatrix}
    }_{L^2(\cI_\omg;\dC^2)}^2\! +\!
    \norm*{ \left(\tau_{k} -\tfrac{1}{2}\right)
      \begin{pmatrix}
        \psi_3 \\ \psi_4
      \end{pmatrix}
    }_{L^2(\cI_\omg;\dC^2)}^2.
  \end{split}
\end{equation}
Performing a long but elementary computation and thanks to integration by parts,
 we have that for any $f,g \in
C_0^{\infty}((0,\omega])$
\begin{equation}\label{eq:quadratic.form.2}
  \begin{split}
&\left\|\left(\tau_{k} -\tfrac{1}{2}\right)
            \begin{pmatrix}
        f \\ g
      \end{pmatrix}
    \right\|_{L^2(\cI_\omg;\dC^2)}^2\\
   &\qquad  =
  \int_0^\omg \left(|\aa f(\tt)-g'(\tt)-\aa\cot\tt g(\tt)|^2 + |f'(\tt)-\aa\cot\tt f(\tt) -\aa g(\tt)|^2\right)\dd\tt\\
  &\qquad 
	=
	\int_0^\omg 
	\left(
	|f'(\tt)|^2+|g'(\tt)|^2+
	\aa^2(1+\cot^2\tt)
	\big(|f(\tt)|^2+|g(\tt)|^2\big)\right)\dd\tt\\
	&\qquad\qquad -2\aa\Re\int_0^\omg
	\left(f(\tt)\ov{g'(\tt)}+ f'(\tt)\ov{g(\tt)}-\cot\tt\big(
	g'(\tt)\ov{g(\tt)}-f'(\tt)\ov{f(\tt)}\big)\right)\dd\tt
\\
  &\qquad=
     \int_0^\omega \left[
      |f'(\theta)|^2 + |g'(\theta)|^2
    \right] \,\dd\theta\\
    &\qquad\qquad+
    \int_0^\omega \left[
       \frac{\alpha(\alpha-1)}{\sin^2\theta} |f(\theta)|^2
      + \frac{\alpha(\alpha+1)}{\sin^2\theta}
       |g(\theta)|^2 \right] \,\dd\theta 
      - \alpha B(f,g)(\omega),
   \end{split}
 \end{equation}
 with
 \begin{equation*}
      B(f,g)(\omega) := \Re \left[
      2f(\omega) \overline{g(\omega)} + \cot \omega
      |f(\omega)|^2
      - \cot \omega |g(\omega)|^2
    \right].
\end{equation*}
We underline that, performing the integration by parts, the boundary term $B(f,g)(\omega)$
only has a contribution in $\omega$, since the functions $f$ and
$g$ are
supported outside the origin.
From the boundary conditions in \eqref{eq:matrix_A} and
\eqref{eq:cD}, we get that
\begin{equation}\label{eq:quadratic.form.3}
\begin{aligned}
  &B(\psi_1,\psi_2)(\omega) + B(\psi_3,\psi_4)(\omega) \\
  &\quad= 2\Re\left[
  \psi_1(\omg)\ov{\psi_2(\omg)}
  +\psi_3(\omg)\ov{\psi_4(\omg)}\right]
  \\
  &\quad\quad+\cot\omg\left[|\psi_1(\omg)|^2+|\psi_3(\omg)|^2-|\psi_2(\omg)|^2-|\psi_4(\omg)|^2\right]\\
  &\quad = 2\Re\left[-\left(\sin\omg\psi_3(\omg)-\cos\omg\psi_4(\omg)\right)\left(\cos\omg\ov{\psi_3(\omg)}+\sin\omg\ov{\psi_4(\omg)}\right)+\psi_3(\omg)\ov{\psi_4(\omg)}
  \right] \\
  &\quad\quad \!+\!\cot\omg\left[|\sin\omg\psi_3(\omg)-\cos\omg\psi_4(\omg)|^2\!+\!|\psi_3(\omg)|^2\!-\!|\cos\omg\psi_3(\omg)+\sin\omg\psi_4(\omg)|^2\!-\!|\psi_4(\omg)|^2\right]\\
  &\quad = 
  2\sin\omg\cos\omg\left[|\psi_4(\omg)|^2-
  |\psi_3(\omg)|^2\right]+4\cos^2\omg\Re[\psi_3(\omg)\ov{\psi_4(\omg)}]\\
  &\quad\quad+\cot\omg\left[2\sin^2\omg\big(
  |\psi_3(\omg)|^2-|\psi_4(\omg)|^2\big)
  -4\sin\omg\cos\omg\Re[\psi_3(\omg)\ov{\psi_4(\omg)}]\right] = 0,
\end{aligned}
\end{equation}
so from \eqref{eq:quadratic.form.1}, \eqref{eq:quadratic.form.2} and \eqref{eq:quadratic.form.3} we
conclude that
\begin{equation}\label{eq:quadratic.form.final}
  \begin{split}
  \norm{(\sfT_k -\tfrac12) \psi}_{L^2(\cI_\omg;\dC^4)}^2 
 = & \int_0^\omega \left[
      |\psi_1'(\theta)|^2 + |\psi_2'(\theta)|^2
      +
      |\psi_3'(\theta)|^2 + |\psi_4'(\theta)|^2
    \right] \,\dd\theta
    \\
    &  +
    \int_0^\omega \left[
       \frac{\alpha(\alpha-1)}{\sin^2\theta}\left(|\psi_1(\theta)|^2 +
         |\psi_3(\theta)|^2\right)
     \right] \,\dd\theta \\
     &  +
    \int_0^\omega \left[
       \frac{\alpha(\alpha+1)}{\sin^2\theta}
       \left(|\psi_2(\theta)|^2 + |\psi_4(\theta)|^2\right)
    \right] \,\dd\theta.
  \end{split}
\end{equation}
We have that $\min[\alpha(\alpha-1),\alpha(\alpha+1)]
\ge -1/4$: thanks to Lemma \ref{lem:hardy.for.us}, from \eqref{eq:quadratic.form.final}
 we conclude that for $k\in\ZZ$ and $0<\omega < \pi/2$
\begin{equation}\label{eq:estimate.below}
\norm*{\left(\sfT_k -\tfrac{1}{2}\right)\psi }^2_{L^2(\cI_\omg;\dC^4)} \geq
\frac{\pi^2}{16\omega^2}\norm{\psi}_{L^2(\cI_\omg;\dC^4)}^2, \quad \text{ for all }
\psi \in \cM.
\end{equation}
From \eqref{eq:estimate.below} it is immediate to conclude \eqref{eq:estimate.below.final}.
\end{proof}

Our next aim is to construct an operator
in a weighted $L^2$-space which is unitarily equivalent to $\sfT_k$. This construction is performed in order to fit better to the application to the MIT bag model on the cone. Consider the unitary transform
\[
	\sfU \colon L^2(\cI_\omg;\sin\tt\dd\tt;\dC^4)\arr L^2(\cI_\omg;\dC^4),\qquad
	(\sfU\psi)(\tt) := (\sin\tt)^{1/2}\psi(\tt).
\]
We define the differential expressions
\[
	\wt{\tau}_k 	:= 
	\begin{pmatrix} 
	k+1 &-\frac{\dd}{\dd\tt} 	-(k+1)\cot\tt\\
	\frac{\dd}{\dd\tt} - k\cot\tt & -k
	\end{pmatrix},
	\qquad
		\wt\cT_k :=\begin{pmatrix} \wt{\tau}_k & 0\\
	0 &\wt{\tau}_k\end{pmatrix},\qquad k\in\dZ.		 
\]
The self-adjoint operator 
\begin{equation}\label{eq:wtT_k}
	\wt{\sfT}_k := \sfU^{-1}\sfT_k\sfU
\end{equation}
in the Hilbert space $L^2(\cI_\omg;\sin\tt\dd\tt;\dC^4)$
can be alternatively characterised as
\begin{equation}\label{eq:wtTk}
	\begin{aligned}
		\wt{\sfT}_k \psi &=\wt{\cT}_k\psi,\\
		\dom\wt\sfT_k &
		\!= \!
		\left\{
		\psi\!\in\! {\rm AC}_{\rm loc}(\cI_\omg;\dC^4)\colon\!
		\psi,
		\wt\cT_k \psi\in L^2(\cI_\omg;\sin\tt\dd\tt;\dC^4),
		\left(\begin{smallmatrix}
			\psi_1(\omg)
			\\
			\psi_2(\omg)
		\end{smallmatrix}\right)\!=\! 
		\sfA_\omg	\left(\begin{smallmatrix}\psi_3(\omg)
			\\
			\psi_4(\omg)
		\end{smallmatrix}\right)
		\right\}\!.
	\end{aligned}
\end{equation}
Our next aim is to find explicitly the eigenvalues and the eigenfunctions of $\wt\sfT_k$. 
In the formulation and the proof of this result we use the so-called Ferrers functions
$\sfP^\mu_\nu(x)$ and $\sfQ^\mu_\nu(x)$ with $x\in(-1,1)$; see \cite[\S 5.15]{O97},~\cite[Chap. 14]{DLMF},
and \cite[\S 8.7 and 8.8]{GR07}. 
\begin{remark}
	The functions $\sfP^\mu_\nu(x)$ and $\sfQ^\mu_\nu(x)$, $x\in(-1,1)$, can be expressed through the associated Legendre functions $P^\mu_\nu(z)$ and $Q^\mu_\nu(z)$ via the identities~\cite[14.23.4 and 14.23.5]{DLMF}.
\end{remark}
\begin{prop}\label{prop:model_Dirac}
	Let $\omg \in (0,\pi)\sm\{\frac{\pi}{2}\}$ and let the operator $\wt\sfT_k$, $k\in\dZ$, 
	be as in~\eqref{eq:wtT_k}. Then the following hold.
	\begin{myenum}
	\item If $k \ge 0$, then $\lm\in\dR$ is
	an eigenvalue of $\wt\sfT_k$ if and only if it is the root of at least one of the following two transcendental equations
	\begin{subequations}\label{eq:evs_k_pos}
	\begin{align}
				(\lm+k+1)\sfP_{\lm}^{-k-1}(\cos\omg) &= \sfP_{\lm-1}^{-k}(\cos\omg),\label{eq:evs_k_pos1}\\
				(\lm+k+1)\sfP_{\lm}^{-k-1}(\cos\omg) &=- \sfP_{\lm-1}^{-k}(\cos\omg).\label{eq:evs_k_pos2}
	\end{align}
	\end{subequations}
	If $\lm$ is a root of~\eqref{eq:evs_k_pos1} (respectively, of~\eqref{eq:evs_k_pos2}) then the associated eigenfunction is given by
	$\psi = \ii\varphi\oplus\varphi$
	(respectively, $\psi = 	-\ii\varphi\oplus\varphi$)
	where 
	\begin{equation}\label{eq:varphi}
		\varphi(\tt) := \begin{pmatrix}\sfP^{-k}_{\lm-1}(\cos\tt)\\
		(k-\lm+1)\sfP^{-k-1}_{\lm-1}(\cos\tt)
		\end{pmatrix}\in L^2(\cI_\omg;\sin\tt\dd\tt;\dC^2).
	\end{equation}
	Moreover, the spectrum of $\wt{\sfT}_k$ is simple.
	\item If $k \le -1$ then $\lm\in\dR$ is
	an eigenvalue of $\wt\sfT_k$ if and only if it is the root of at least one of the following two transcendental equations
	\begin{subequations}\label{eq:evs_k_neg}
	\begin{align}
				(\lm+k)\sfP_{\lm-1}^{k}(\cos\omg) &= -\sfP_{\lm}^{k+1}(\cos\omg),
\label{eq:evs_k_neg1}\\
				(\lm+k)\sfP_{\lm-1}^{k}(\cos\omg) &= \sfP_{\lm}^{k+1}(\cos\omg).
	\label{eq:evs_k_neg2}
	\end{align}
	\end{subequations}
	If $\lm$ is a root of~\eqref{eq:evs_k_neg1} (respectively, of~\eqref{eq:evs_k_neg2}) then the associated eigenfunction is given by
	$\psi = \ii\varphi\oplus\varphi$
		(respectively, $\psi = 	-\ii\varphi\oplus\varphi$)
		where 
		\[
		\varphi(\tt) := \begin{pmatrix}
		(\lm+k)\sfP^{k}_{\lm-1}(\cos\tt)\\
		\sfP^{k+1}_{\lm-1}(\cos\tt)\\
		\end{pmatrix}\in L^2(\cI_\omg;\sin\tt\dd\tt;\dC^2).
		\]
	Moreover, the spectrum of $\wt{\sfT}_k$ is simple.
	\end{myenum}
\end{prop}
\begin{proof}
   Let $\lm\in\dR$ and consider the ordinary differential equation 
	$\wt\cT_k\psi = \lm\psi$ with $\psi\in\dom\wt\sfT_k$. From this equation we find the system
	\begin{equation}\label{eq:system}
	\begin{cases}
		(k+1)\psi_1(\tt) - \psi_2'(\tt) - (k+1)\cot\tt\psi_2(\tt) = \lm\psi_1(\tt),\\
		\psi_1'(\tt) - k\cot\tt\psi_1(\tt) - k\psi_2(\tt) = \lm\psi_2(\tt).
	\end{cases}
	\end{equation}
	From the first equation in the above system we find that 
	\[
		(\lm-k-1)\psi_1(\tt) = -\psi_2'(\tt) - (k+1)\cot\tt\psi_2(\tt),
	\]
	and differentiating the above equation with respect to $\tt$ we get
	\[
		(\lm-k-1)\psi_1'(\tt) = -\psi_2''(\tt) - (k+1)\cot\tt\psi_2'(\tt) + \frac{k+1}{\sin^2\tt}\psi_2(\tt).
	\]
	Multiplying the second equation in~\eqref{eq:system} by $(\lm-k-1)$
	and substituting there the above formulae for $(\lm-k-1)\psi_1(\tt)$ and $(\lm-k-1)\psi'_1(\tt)$ we get after a simplification
	%
	\begin{equation}\label{eq:diff_eq_psi2}
		\psi_2''(\tt) + \cot\tt\psi_2'(\tt) +\left(\lm(\lm-1)-\frac{(k+1)^2}{\sin^2\tt}\right)\psi_2(\tt) = 0.
	\end{equation}
Making the change of variables $\tt = \arccos x$ in the last equation we get the equation of the form~\cite[Eq. (12.02)]{O97} 
with $\nu = |\lm-\tfrac12| -\tfrac12$ and $\mu = |k+1|$ settled therein. 
Hence, by~\cite[\S 5.15]{O97} and \cite[Eq.~(14.2.6)]{DLMF} the general solution of the differential equation~\eqref{eq:diff_eq_psi2} is given by
	\begin{equation}\label{key}
		\psi_2(\tt) = c_2\sfP_{\nu}^{-\mu}(\cos\tt) + c_2'\sfQ_\nu^{\mu}(\cos\tt),
	\end{equation}
	where $c_2,c_2'\in\dC$ are arbitrary constants.
	Thanks to \cite[Eq.~(14.9.5)]{DLMF}, $\sfP_{\nu}^{-\mu}(x) = \sfP_{\lambda -1}^{-\mu}(x) = \sfP_{- \lambda}^{-\mu}(x)$ for all 
	$x \in (-1,1)$, so in the following we will just write $\sfP_{\lambda -1}^{-\mu}(\cos\tt)$ in our formulae.
The assumption $\psi_2\in L^2(\cI_\omg;\sin\tt\dd\tt)$
	and the asymptotics~\cite[Eqs. (15.06) and (15.07)]{O97} of
	$\sfP_\nu^{-\mu}(x)$, $\sfQ_\nu^\mu(x)$ as $x\arr 1^-$ yield that $c_{2}' = 0$ provided that $k \ne -1$. In the case $k = -1$ 
	we compute using~\cite[Eq. (14.10.4)]{DLMF} the derivative of $\psi_2$ with respect to $\tt$
	\begin{equation}\label{eq:psi2_der}
	\begin{aligned}
		\psi_2'(\tt) &= 
		\frac{c_2\lm}{\sin\tt}\left(
		\sfP_{\lm}^0(\cos\tt) - \cos\tt \sfP_{\lm-1}^0(\cos\tt)\right)\\
		&\qquad+ 
		\frac{c_2'(1+\nu)}{\sin\tt}\left(
		\sfQ_{\nu+1}^0(\cos\tt) - \cos\tt \sfQ_\nu^0(\cos\tt)\right).
	\end{aligned}
	\end{equation}
	The condition $\wt{\cT}_{-1}\psi\in L^2(\cI_\omg;\sin\tt\dd\tt;\dC^4)$ yields $\psi_2'\in L^2(\cI_\omg;\sin\tt\dd\tt)$.
	Combining~\eqref{eq:psi2_der} with the series~\cite[9.100]{GR07}, the representation of the Ferrers function of the first kind in~\cite[8.704]{GR07} and with the asymptotics~\cite[Eq. (14.8.3)]{DLMF}
	we conclude that $c_{2}' = 0$ also for $k = -1$.
	Hence, we end up with 
	\[
		\psi_2(\tt) = c_{2}\sfP^{-|k+1|}_{\lm-1}(\cos\tt).
	\]
	From the second equation in ~\eqref{eq:system} we find that 
	\[
	\begin{cases}
	(\lm+k)\psi_2(\tt) = \psi_1'(\tt) - k\cot\tt\psi_1(\tt),\\
	(\lm+k)\psi_2'(\tt) = \psi_1''(\tt) - k\cot\tt\psi_1'(\tt) + \frac{k}{\sin^2\tt}\psi_1(\tt).
	\end{cases}	
	\]
	Multiplying the first equation in~\eqref{eq:system} by $(\lm+k)$
	and substituting there the above formulae for $(\lm+k)\psi_2(\tt)$ and $(\lm+k)\psi_2'(\tt)$ we get
	\[
		\psi''_1(\tt) + \cot\tt\psi_1'(\tt) +
		\left(\lm(\lm-1) -\frac{k^2}{\sin^2\tt}\right)\psi_1(\tt) = 0.
	\]
	Analogously we find that
	\[
		\psi_1(\tt) = c_{1}\sfP_{\lm-1}^{-|k|}(\cos\tt)
	\]
	with some constant $c_{1}\in\dC$.
	Our next aim is to find the relation that connects the constants $c_{1}$ and $c_{2}$.
	To this aim we consider two cases $k\ge 0$ and $k\le - 1$. First, assume that $k \ge 0$. In this case we have
	\[
		\psi_1(\tt) = c_{1}\sfP_{\lm-1}^{-k}(\cos\tt)\and \psi_2(\tt) = c_{2}\sfP_{\lm-1}^{-k-1}(\cos\tt).
	\]
	From the first identity in~\eqref{eq:system}
	and using the formula for the derivative of the Ferrers functions of the first kind~\cite[Eq. (14.10.4)]{DLMF} we find
	\[
	\begin{aligned}
	&(\lm-k-1)c_{1}\sfP_{\lm-1}^{-k}(\cos\tt) \\
	&\qquad =c_{2}\left[-\frac{\lm+k+1}{\sin\tt}
		\sfP_{\lm}^{-k-1}(\cos\tt) + (\lm - k-1)\cot\tt\sfP_{\lm-1}^{-k-1}(\cos\tt)\right].
	\end{aligned}
	\]
	Making use of the recurrence relation~\cite[Eq. (14.10.2)]{DLMF} in the above formula we get
	\[
		 c_{2} =	(k-\lm+1)c_{1}.
	\]
	Now we pass to the case $k\le -1$.
	In this case we have
	\[
	\psi_1(\tt) = c_{1}\sfP_{\lm-1}^{k}(\cos\tt)\and \psi_2(\tt) = c_{2}\sfP_{\lm-1}^{k+1}(\cos\tt).
	\]
	From the second identity in~\eqref{eq:system}
	and using again the formula for the derivative of the Ferrers functions of the first kind~\cite[Eq. (14.10.4)]{DLMF} we find
	\[
		(\lm+k)c_{2}\sfP_{\lm-1}^{k+1}(\cos\tt) = c_{1}
		\left[
		\frac{\lm-k}{\sin\tt}\sfP_{\lm}^k(\cos\tt) - (\lm+k)\cot\tt\sfP_{\lm-1}^k(\cos\tt)
		\right].
	\]
	Making again use of the recurrence relation~\cite[Eq. (14.10.2)]{DLMF} in the above formula we get
	\[
		c_{1} = (\lm+k)c_{2}
	\]
	Hence, we find that
	\begin{equation}\label{eq:psi12}
	\begin{aligned}
	\begin{pmatrix}
		\psi_1(\tt)\\
		\psi_2(\tt)
	\end{pmatrix} 
	&= 
	c_{12}\begin{pmatrix}
		\sfP^{-k}_{\lm-1}(\cos\tt)\\
		(k-\lm+1)\sfP_{\lm-1}^{-k-1}(\cos\tt)
	\end{pmatrix},&\text{for}&\,\, k \ge0.\\
	\begin{pmatrix}
		\psi_1(\tt)\\
		\psi_2(\tt)
	\end{pmatrix} &= 
		c_{12}\begin{pmatrix}
		(\lm+k)\sfP^{k}_{\lm-1}(\cos\tt)\\
		\sfP_{\lm-1}^{k+1}(\cos\tt)
		\end{pmatrix},&\text{for}&\,\, k \le -1.
\end{aligned}
	\end{equation}
	with some constant $c_{12}\in\dC$.
	Performing similar analysis for the last two components of the vector-valued function $\psi$ we find that
	\begin{equation}\label{eq:psi34}
	\begin{aligned}
	\begin{pmatrix}
	\psi_3(\tt)\\
	\psi_4(\tt)
	\end{pmatrix} 
	&= 
	c_{34}\begin{pmatrix}
	\sfP^{-k}_{\lm-1}(\cos\tt)\\
	(k-\lm+1)\sfP_{\lm-1}^{-k-1}(\cos\tt)
	\end{pmatrix},&\text{for}&\,\, k \ge0.\\
	\begin{pmatrix}
	\psi_3(\tt)\\
	\psi_4(\tt)
	\end{pmatrix} &= 
	c_{34}\begin{pmatrix}
	(\lm+k)\sfP^{k}_{\lm-1}(\cos\tt)\\
	\sfP_{\lm-1}^{k+1}(\cos\tt)
	\end{pmatrix},&\text{for}&\,\, k \le -1.
	\end{aligned}
	\end{equation}
	with some constant $c_{34}\in\dC$.
	
	Now we recall the boundary conditions satisfied by $\psi$
	\begin{equation}\label{eq:bc_recalled}
	\begin{pmatrix}
		\psi_1(\omg)
		\\
		\psi_2(\omg)
	\end{pmatrix}= 
	\sfA_\omg	\begin{pmatrix}\psi_3(\omg)
		\\
		\psi_4(\omg)
	\end{pmatrix},
	\end{equation}
	where the matrix $\sfA_\omg$ is given by~\eqref{eq:matrix_A}. It is straightforward to verify that 
	$\sfA_\omg$ has two eigenvalues $\lm = \pm \ii$ and that the respective eigenvectors are given by
	\[
		\xi_\ii = \begin{pmatrix} 1 \\
		\frac{\sin\omg-1}{\cos\omg}\end{pmatrix}
		\and \xi_{-\ii} = \begin{pmatrix}1\\ \frac{\sin\omg+1}{\cos\omg}\end{pmatrix}.
	\]
	Using these properties of the matrix $\sfA_\omg$, the boundary condition~\eqref{eq:bc_recalled} and the
	representations~\eqref{eq:psi12},~\eqref{eq:psi34} for $k\ge0$ we conclude that
	if $k\ge0$ then $\lm\in\dR$ is an eigenvalue of $\wt{\sfT}_k$ if and only if it is the root of at least of one of the following equations 
	\[
	\begin{aligned}
		\begin{cases}
		(k-\lm+1)\sfP^{-k-1}_{\lm-1}(\cos\omg) & = \frac{\sin\omg-1}{\cos\omg}\sfP_{\lm-1}^{-k}(\cos\omg),\\
		(k-\lm+1)\sfP^{-k-1}_{\lm-1}(\cos\omg) & = \frac{\sin\omg+1}{\cos\omg}\sfP_{\lm-1}^{-k}(\cos\omg).
		\end{cases}
	\end{aligned}
	\]
	Applying~\cite[8.735 (1)]{GR07} we simplify the above equations as
	\[
	\begin{aligned}
			\begin{cases}
			(\lm+k+1)\sfP_{\lm}^{-k-1}(\cos\omg) = \sfP_{\lm-1}^{-k}(\cos\omg),
			\\
			(\lm+k+1)\sfP_{\lm}^{-k-1}(\cos\omg) =- \sfP_{\lm-1}^{-k}(\cos\omg).
			\end{cases}
	\end{aligned}
	\]
	It follows from the first equations in~\eqref{eq:psi12} and~\eqref{eq:psi34} that the respective eigenfunctions are given as in item~(i).	
		
	Again using these properties of the matrix $\sfA_\omg$, the boundary condition~\eqref{eq:bc_recalled} and the
	representations~\eqref{eq:psi12},~\eqref{eq:psi34} for $k\le -1$ we conclude that
	if $k\le-1$ then $\lm\in\dR$ is an eigenvalue of $\wt{\sfT}_k$ if it is the root of at least one of the following equations 
	\[
		\begin{cases}
		(\lm+k)\sfP^{k}_{\lm-1}(\cos\omg) = \frac{\cos\omg}{\sin\omg-1}\sfP_{\lm-1}^{k+1}(\cos\omg),\\
		(\lm+k)\sfP^{k}_{\lm-1}(\cos\omg) = \frac{\cos\omg}{\sin\omg+1}\sfP_{\lm-1}^{k+1}(\cos\omg).
		\end{cases}
	\]
	Applying~\cite[8.735 (4)]{GR07} we simplify the above equations as
		\[
		\begin{aligned}
				\begin{cases}
				(\lm+k)\sfP_{\lm-1}^{k}(\cos\omg) = -\sfP_{\lm}^{k+1}(\cos\omg),
				\\
				(\lm+k)\sfP_{\lm-1}^{k}(\cos\omg) = \sfP_{\lm}^{k+1}(\cos\omg).
				\end{cases}
		\end{aligned}
		\]
	It follows from the second equations in~\eqref{eq:psi12} and~\eqref{eq:psi34} that the respective eigenfunctions are given as in item~(ii).
	
	It remains to show that the spectrum of $\wt{\sfT}_k$ is simple.
	We will consider the case $k \ge 0$ only, because the case  $k\le -1$ can be analysed analogously.	
	Since the eigenfunctions are characterized via the first equations in~\eqref{eq:psi12},~\eqref{eq:psi34} by two constants $c_{12}$ and $c_{34}$ leaving two degrees of freedom the multiplicity is bounded from above by $2$. Suppose that $\lm$ is an eigenvalue of $\wt{\sfT}_k$ of multiplicity $2$. Then $\lm$ is the root of both equations in~\eqref{eq:evs_k_pos}
	and the respective linear independent
	eigenfunctions are given by $(\pm\ii\varphi)\oplus \varphi$ with $\varphi$ as in~\eqref{eq:varphi}.
	Hence, we conclude that $0\oplus\varphi$ is also an eigenfunction corresponding to $\lm$ and from the boundary conditions in~\eqref{eq:wtTk} we get that $\varphi(\omg) = (0,0)^\top$. In view of $\tau_k((\sin\tt)^{1/2}\varphi(\tt)) = \lm(\sin\tt)^{1/2}\varphi(\tt)$, $\tt\in(0,\omg)$ we get by the unique solvability result \cite[Satz. 15.4 (b)]{W03} that $\varphi \equiv 0$,
	which leads to a contradiction.  
\end{proof}

\begin{cor}\label{cor:model_Dirac}
	Let $\omg\in (0,\pi)\sm \{\frac{\pi}{2}\}$.
	The spectrum of $\wt{\sfT}_k$, $k\in\dZ$, is symmetric with respect to the origin. In particular,
	\begin{myenum}
	\item For $k \ge0$, $\lm\in\dR$ is a root of~\eqref{eq:evs_k_pos1} if and only if $-\lm$ is a root of~\eqref{eq:evs_k_pos2}.
	\item For $k \le-1$, $\lm\in\dR$ is a root of~\eqref{eq:evs_k_neg1} if and only if $-\lm$ is a root of~\eqref{eq:evs_k_neg2}.
	\end{myenum}
	In particular, $0$ is not an eigenvalue of $\wt{\sfT}_k$.
\end{cor}
\begin{proof}
	First of all, we remark that in view of (i) and (ii) the point $\lm = 0$ can not be an eigenvalue of $\wt{\sfT}_k$ as otherwise it would be a double eigenvalue, which contradicts the simplicity of the spectrum shown in Proposition~\ref{prop:model_Dirac}.
	Now we pass to the proofs of the items (i) and (ii).

	\noindent(i) Let $\lm\in\dR$ be an eigenvalue of $\wt{\sfT}_k$, $k\ge0$. Then $\lm$ is a root either of~\eqref{eq:evs_k_pos1} or of~\eqref{eq:evs_k_pos2}. Assume for definiteness that $\lm$ is a root of~\eqref{eq:evs_k_pos1}; \ie
	\begin{equation}\label{eq:initial}
			(\lm+k+1)\sfP_{\lm}^{-k-1}(\cos\omg) = \sfP_{\lm-1}^{-k}(\cos\omg).
	\end{equation}
	 Our aim is to show that $-\lm$ is a root of the complementary equation~\eqref{eq:evs_k_pos2}; \ie
	 \[
		(-\lm+k+1)\sfP_{-\lm}^{-k-1}(\cos\omg) =- \sfP_{-\lm-1}^{-k}(\cos\omg).
	 \]
	 Transforming the above equation using~\cite[8.733 (5)]{GR07} we need to show that
 	 \begin{equation}\label{eq:goal}
	 		(\lm-k-1)\sfP_{\lm-1}^{-k-1}(\cos\omg) = \sfP_{\lm}^{-k}(\cos\omg).
 	 \end{equation}
 	 In view of~\cite[8.735 (2)]{GR07} and using~\eqref{eq:initial} we get
 	 \begin{equation}\label{eq:eq1}
 	 \begin{split}
 	 	(\lm-k-1)\sfP_{\lm-1}^{-k-1}(\cos\omg) 
 	 	&\! =\! (\lm+k+1)\cos\omg\sfP_{\lm}^{-k-1}(\cos\omg) -\sin\omg\sfP_\lm^{-k}(\cos\omg)\\
 	 	&
 	 	\!=\! \cos\omg\sfP_{\lm-1}^{-k}(\cos\omg) -\sin\omg\sfP_\lm^{-k}(\cos\omg).
 	 \end{split}	
 	  \end{equation}
   Applying further the identity \cite[8.735~(4)]{GR07}
   we find
	\begin{equation}\label{eq:eq2}
 	 	\sfP_\lm^{-k}(\cos\omg) = \cos\omg\sfP_{\lm-1}^{-k}(\cos\omg) - (\lm-k-1)\sin\omg\sfP_{\lm-1}^{-k-1}(\cos\omg)
	\end{equation} 	  	 %
 	 Taking the difference of~\eqref{eq:eq1} and~\eqref{eq:eq2} we get
 	 \[
 	 	(\lm-k-1)\sfP_{\lm-1}^{-k-1}(\cos\omg) - \sfP_\lm^{-k}(\cos\omg) =\sin\omg\left[
 	 	(\lm-k-1)\sfP_{\lm-1}^{-k-1}(\cos\omg) -\sfP_\lm^{-k}(\cos\omg)
 	 	\right]. 
 	 \]
 	 Hence, we conclude that~\eqref{eq:goal} holds.

\noindent (ii) 
	Let $\lm\in\dR$ be an eigenvalue of $\wt{\sfT}_k$, $k\le-1$. Then $\lm$ is a root either of~\eqref{eq:evs_k_neg1} or of~\eqref{eq:evs_k_neg2}. Assume for definiteness that $\lm$ is a root of~\eqref{eq:evs_k_neg1}; \ie
\begin{equation}\label{eq:initial2}
	(\lm+k)\sfP_{\lm-1}^k(\cos\omg) 
			= -\sfP_{\lm}^{k+1}(\cos\omg).
\end{equation}
Our aim is to show that $-\lm$ is a root of the complementary equation~\eqref{eq:evs_k_neg2}; \ie
\[
	(-\lm+k)\sfP_{-\lm-1}^{k}(\cos\omg) = \sfP_{-\lm}^{k+1}(\cos\omg).
\]
Transforming the above equation using~\cite[8.733 (5)]{GR07} we need to show that
\begin{equation}\label{eq:goal2}
	(k-\lm)\sfP_{\lm}^{k}(\cos\omg) = \sfP_{\lm-1}^{k+1}(\cos\omg).
\end{equation}
In view of~\cite[8.735 (1)]{GR07} and using~\eqref{eq:initial2} we get
\begin{equation}\label{eq:eq3}
\begin{split}
(k-\lm)\sfP_{\lm}^k(\cos\omg)
&\! =\! 
-(\lm+k)\cos\omg\sfP_{\lm-1}^k(\cos\omg) -\sin\omg\sfP_{\lm-1}^{k+1}(\cos\omg)\\
&
\!=\! \cos\omg\sfP_{\lm}^{k+1}(\cos\omg) -\sin\omg\sfP_{\lm-1}^{k+1}(\cos\omg).
\end{split}	
\end{equation}
Applying further the identity \cite[8.735~(3)]{GR07}
we find
\begin{equation}\label{eq:eq4}
\sfP_{\lm-1}^{k+1}(\cos\omg) = \cos\omg\sfP_{\lm}^{k+1}(\cos\omg) + (\lm-k)\sin\omg\sfP_{\lm}^{k}(\cos\omg)
\end{equation} 	  	 %
Taking the difference of~\eqref{eq:eq3} and~\eqref{eq:eq4} we get
\[
(k-\lm)\sfP_{\lm}^{k}(\cos\omg) - \sfP_{\lm-1}^{k+1}(\cos\omg) =\sin\omg\left[
(k-\lm)\sfP_{\lm}^{k}(\cos\omg) -\sfP_{\lm-1}^{k+1}(\cos\omg)
\right]. 
\]
Hence, we conclude that~\eqref{eq:goal2} holds.
\end{proof}
\begin{remark}\label{rem:anti}
	The symmetry of the spectrum of $\wt\sfT_k$, $k\in\dZ$, with respect to the origin can be alternatively  shown via an anti-commutation relation. In this argument it is more convenient to work with the unitarily equivalent operator $\sfT_k$. 
	Consider the matrix-valued function $M\colon [0,\omg]\arr\dC^{2\times2}$
	\[
		M = M(\tt) := \begin{pmatrix} \cos\tt & \sin\tt\\
		\sin\tt & -\cos\tt\end{pmatrix}
	\]
	and define the unitary and self-adjoint operator
	\[
		\cU\colon L^2(\cI_\omg;\dC^4)\arr L^2(\cI_\omg;\dC^4),\qquad 
		(\cU\psi)(\tt) := 
		\big(M(\tt)\oplus(-M(\tt)\big)\psi(\tt). 
	\]
	It is straightforward to check the following anti-commutation relations $\tau_k M= -M\tau_k$ and $\sfA_\omg M(\omg) = -M(\omg)\sfA_\omg$. Hence, we conclude that $\dom\sfT_k = \dom(\sfT_k\cU)$ and that
	$\sfT_k\cU = -\cU\sfT_k$.
	Thus, for any $\lm\in\dR$ the mapping $\cU$ is a bijection between $\ker(\sfT_k - \lm)$ and $\ker(\sfT_k+\lm)$ and therefore the spectrum of $\sfT_k$ is symmetric about the origin.
\end{remark}

Now we can combine the symmetry of the spectrum shown above with the estimate in Proposition~\ref{prop:spectrum.outside} in order to get a lower bound on the spectral gap for $\sfT_k$.
\begin{prop}\label{prop:spectrum_outside2}
	For all $k \in \ZZ$ and $\omega<\pi/2$,
	\begin{equation}\label{eq:not_sharp}
		\|\sfT_k\psi\|_{L^2(\cI_\omg;\dC^4)} \ge \left(\frac{\pi}{4\omg} + \frac12\right)\|\psi\|_{L^2(\cI_\omg;\dC^4)},\qquad\text{for all}\,\psi\in\dom\sfT_k.
	\end{equation}
	In particular, $\s(\sfT_k)\cap[-\frac12,\frac12] = \emptyset$.
\end{prop}
\begin{proof}
	Recall that $\sfT_k$ is unitarily equivalent to $\wt{\sfT}_k$.
	Hence, by Corollary~\ref{cor:model_Dirac} the spectrum of $\sfT_k$ is symmetric with respect to the origin
	and $0$ is not an eigenvalue of $\sfT_k$. Let $\lm > 0$ be the smallest positive eigenvalue of $\sfT_k$. According to the bound in Proposition~\ref{prop:spectrum.outside}
	we conclude that $|\lm - \frac12| \ge \frac{\pi}{4\omg}$. This condition implies that $\lm \ge \frac12 + \frac{\pi}{4\omg}$ and the inequality in the formulation of the proposition follows by the min-max principle. 
\end{proof}

\begin{remark}\label{rem:sharpness}
	Depending on $k \in \ZZ$, we can provide better estimates 
	allowing us to show that $\sigma(\sfT_k) \cap \left[-\frac12,\frac12\right] = \emptyset$
	dropping the condition 
	$\omega < \pi/2$.
	In detail, for $|k+\tfrac12| = |\alpha| \geq 3/2$ 
	we have 
	$\min[\alpha(\alpha-1),\alpha(\alpha+1)]  \ge 3/4$, so
	from \eqref{eq:quadratic.form.final} we conclude that 
	\begin{equation}\label{eq:estimate.below.1}
	\norm*{\left(\sfT_k -\tfrac{1}{2}\right)\psi }_{L^2(\cI_\omg;\dC^4)}^2 \geq
	\frac{3}{4}\norm{\psi}_{L^2(\cI_\omg;\dC^4)}^2, \quad \text{ for all }
	\psi \in \cM.
	\end{equation}
	For the smallest positive eigenvalue $\lm$ of $\sfT_k$ we get $\lm \ge \frac{\sqrt{3}}{2} + \frac12$ and in view of symmetry of the spectrum of $\sfT_k$ with respect to the origin this is enough to show that $\sigma(\sfT_k) \cap \left[-\frac12,\frac12\right] = \emptyset$
	for $k \leq -2$ or $k\geq 1$ and $\omega\in(0,\pi)\sm\{\frac{\pi}{2}\}$.
\end{remark}

\section{Decomposition of $\sfD_\omg$ in spherical coordinates}\label{sec:decomp}
In this section we decompose the Dirac operator $\sfD_\omg$ into the orthogonal sum of one-dimensional Dirac operators on the half-line and using this decomposition we prove the main results of the paper formulated in Theorems~\ref{thm:main1},~\ref{thm:main2}, and~\ref{thm:perturbation}.

First of all, we observe that the cone can be expressed in the spherical coordinates as $\cC_\omg = \dR_+\times \cM_\omg$ where  $\cM_\omg := [0,\omg)\times\dS^1$ is the spherical cap.
We introduce the spherical $L^2$-space on $\cC_\omg$
\[
	L^2_{\rm sph}(\cC_\omg;\dC^4) = 
	L^2(\dR_+\times\cM_\omg;r^2\sin\tt\dd\tt\dd\phi).
\]
The  Hilbert space $L^2_{\rm sph}(\cC_\omg;\dC^4)$ can be decomposed into the tensor product of weighted $L^2$-spaces
\begin{equation}\label{eq:tensor_product}
	L^2_{\rm sph}(\cC_\omg;\dC^4) = L^2(\dR_+;r^2\dd r)\otimes
	L^2(\cM_\omg; \sin\tt\dd\tt\dd\phi;\dC^4).
\end{equation}
Along with the spherical $L^2$-spaces we define the spherical first-order Sobolev space
\begin{equation}\label{eq:tensor_product}
H^1_{\rm sph}(\cC_\omg;\dC^4) = 
\big\{u\in L^2_{\rm sph}(\cC_\omg;\dC^4)\colon \p_r u, r^{-1}\nabla_{\dS^2}u\in L^2_{\rm sph}(\cC_\omg;\dC^4)\big\},
\end{equation}	
where $\nabla_{\dS^2}$ is the gradient on $\dS^2$.
Next, we introduce the unitary map
\[
	\sfV\colon L^2(\cC_\omg;\dC^4)\arr L^2_{\rm sph}(\cC_\omg;\dC^4),\qquad (\sfV u)(r,\tt,\phi) := u(r\sin\tt\cos\phi,r\sin\tt\sin\phi,r\cos\tt).
\]
In particular, we have $\sfV(H^1(\cC_\omg;\dC^4)) = H^1_{\rm sph}(\cC_\omg;\dC^4)$.
By means of the map $\sfV$ we define the operator which is unitarily equivalent to $\sfD_\omg$
\[
	 \wt{\sfD}_\omg := \sfV\sfD_\omg\sfV^{-1},\qquad 
	 \dom\wt{\sfD}_\omg := \sfV(\dom\sfD_\omg).
\]
According to~\cite[Satz 20.6]{W03} the Dirac differential expression $\cD$ can be written in the spherical coordinates as follows
\[
	\cD = \ii\aa_r\left(-\p_r-\frac{1}{r} +\frac{\cK}{r}\right),
\]
where $\aa_r$ is given by~\eqref{eq:alpha_r}
and where $\cK$ is the spin-orbit differential expression in~\eqref{eq:spin_orbit}.
With respect to the tensor product representation~\eqref{eq:tensor_product} we can decompose the Dirac operator $\wt{\sfD}_\omg$ as
\begin{equation}\label{eq:Dirac_cone_spherical}
	\wt{\sfD}_\omg u = \ii\aa_r\left(\left(-\p_r-\frac{1}{r}\right)\otimes\sfI +\frac{1}{r}\otimes\sfK_\omg\right)u,\qquad \dom\wt{\sfD}_\omg = \sfV\big(
	\dom\sfD_\omg\big),
\end{equation}
where the spin-orbit operator $\sfK_\omg$ acts in the Hilbert space $L^2(\cM_\omg;\sin\tt\dd\tt\dd\phi;\dC^4)$ and is defined as
\begin{equation}\label{eq:Komg}
\begin{aligned}
	\sfK_\omg f & \!=\! \cK f,\\
	\dom \sfK_\omg & \!=\! \left\{f\!\in\! C^\infty(\ov{\cM_\omg};\dC^4)\colon\!
	\left(\begin{smallmatrix}
	f_1(\omg,\phi)\\
	f_2(\omg,\phi)
	\end{smallmatrix}\right)\! =\!
	\left(
	\begin{smallmatrix}  
	\ii\sin\omg &-\ii\cos\omg e^{-\ii\phi} \\	-\ii\cos\omg e^{\ii\phi} &-\ii\sin\omg 
	\end{smallmatrix}\right)
	\left(\begin{smallmatrix}
	f_3(\omg,\phi)\\
	f_4(\omg,\phi)
	\end{smallmatrix}\right)
	\right\}.
\end{aligned}	
\end{equation}
In view of the above definition the domain of $\wt{\sfD}_\omg$ can be characterised more explicitly as
\begin{equation}\label{eq:explicit_domain}
	\dom\wt{\sfD}_\omg = 
	\{u\in C^\infty_0((0,\infty)\times\ov\cM_\omg)\colon
	u(r,\cdot,\cdot)\in\dom\sfK_\omg,\,\forall r >0\}.
\end{equation}

Our next goal is to show that $\sfK_\omg$ has an orthonormal basis of eigenfunctions that correspond to real eigenvalues. The latter will also imply as a by-product that the operator $\sfK_\omg$ is essentially self-adjoint in the Hilbert space $L^2(\cM_\omg;\sin\tt\dd\tt\dd\phi;\dC^4)$.
The Hilbert space $L^2(\cM_\omg;\sin\tt\dd\tt\dd\phi;\dC^4)$ can be further decomposed as the tensor product
\[
	L^2(\cM_\omg;\sin\tt\dd\tt\dd\phi;\dC^4) = 
	L^2(\cI_\omg;\sin\tt\dd\tt)\otimes L^2(\dS^1;\dC^4).
\]
We consider the orthonormal basis $\{h_{k1},h_{k2},h_{k3},h_{k4}\}_{k\in\dZ}$ of $L^2(\dS^1;\dC^4)$
given by
\[
\begin{aligned}
	h_{k1}(\phi) & = 
	\frac{1}{\sqrt{2\pi}}\begin{pmatrix} e^{\ii k\phi} \\ 0\\0\\0\end{pmatrix},\qquad
	&h_{k2}(\phi) &= 
	\frac{1}{\sqrt{2\pi}}\begin{pmatrix} 0 \\ e^{\ii (k+1)\phi}\\ 0 \\ 0\end{pmatrix},\\
	h_{k3}(\phi) & = 
	\frac{1}{\sqrt{2\pi}}\begin{pmatrix} 0 \\ 0\\e^{\ii k\phi}\\0\end{pmatrix},\qquad
	&h_{k4}(\phi)& = 
	\frac{1}{\sqrt{2\pi}}\begin{pmatrix} 0 \\ 0\\ 0\\ e^{\ii (k+1)\phi}\end{pmatrix}.
\end{aligned}	
\]
Consider the subspace $\cF_k = {\rm span}\,\{h_{k1}, h_{k2},h_{k3},h_{k4}\}$, $k\in\dZ$, of $L^2(\dS^1;\dC^4)$. Now, we have the decomposition
\begin{equation}\label{eq:ortho}
	L^2(\cM_\omg;\sin\tt\dd\tt\dd\phi;\dC^4) =
	\bigoplus_{k\in\dZ} \big(L^2(\cI_\omg;\sin\tt\dd\tt)\otimes\cF_k\big) 
\end{equation}
and the isomorphism
\begin{equation}\label{eq:isomorph}
	L^2(\cI_\omg;\sin\tt\dd\tt)\otimes\cF_k\simeq  L^2(\cI_\omg;\sin\tt\dd\tt;\dC^4).
\end{equation}
By the spectral theorem the family of eigenfunctions of $\wt{\sfT}_k$ constructed in Proposition~\ref{prop:model_Dirac} upon normalization constitutes an orthonormal basis in the Hilbert space $L^2(\cI_\omg;\sin\tt\dd\tt;\dC^4)$.
In view of the orthogonal decomposition~\eqref{eq:ortho}
and the isomorphism~\eqref{eq:isomorph} we can construct an orthonormal basis in the Hilbert space $L^2(\cM_\omg;\sin\tt\dd\tt\dd\phi;\dC^4)$ by transplanting the basis of normalized eigenfunctions of $\wt{\sfT}_k$ into the respective fiber $L^2(\cI_\omg;\sin\tt\dd\tt)\otimes\cF_k$ in the decomposition~\eqref{eq:ortho}.
To this aim recall the definition of the following discrete subsets of the real axis:
\begin{equation}\label{eq:Zk}
\begin{aligned}
	\cZ_k &:= \big\{\lm\in\dR\colon (\lm+k+1)\sfP_\lm^{-k-1}(\cos\omg) = \sfP_{\lm-1}^{-k}(\cos\omg)\big\}, & k&\ge 0,\\
	\cZ_k &:=\big\{\lm\in\dR\colon (\lm+k)\sfP_{\lm-1}^{k}(\cos\omg) = -\sfP_{\lm}^{k+1}(\cos\omg)\big\}, & k&\le -1.
\end{aligned}
\end{equation}	
Let us introduce the following  two-component functions $\{\Phi_{k,\lm}^\pm\}_{k\in\dZ,\lm\in\cZ_k}$ on $\cM_\omg$
\[
\begin{aligned}
	\Phi_{k,\lm}^\pm(\tt,\phi) & := 
	\begin{pmatrix}
	\sfP^{-k}_{\pm\lm-1}(\cos\tt)e^{\ii k\phi}\\
	(k\mp\lm+1)\sfP^{-k-1}_{\pm\lm-1}(\cos\tt)e^{\ii (k+1)\phi}
	\end{pmatrix},\qquad  &k \ge 0,\,\lm\in\cZ_k,\\
	\Phi_{k,\lm}^\pm(\tt,\phi) & := 
	\begin{pmatrix}
	(\lm\pm k)\sfP^{k}_{\pm\lm-1}(\cos\tt)e^{\ii k\phi}\\
	\pm\sfP^{k+1}_{\pm\lm-1}(\cos\tt)e^{\ii (k+1)\phi}
	\end{pmatrix},\qquad & k \le-1,\,\lm\in\cZ_k.
\end{aligned}
\]
From Proposition~\ref{prop:model_Dirac} and Corollary~\ref{cor:model_Dirac}
we get that the following family $\{\Psi_{k,\lm}^+,\Psi_{k,\lm}^-\}_{k\in\dZ,\lm\in\cZ_k}$ of vector-valued functions on $\cM_\omg$: 
\[
	\Psi_{k,\lm}^\pm = c_{k,\lm}(\pm\ii\Phi_{k,\lm}^\pm)\oplus\Phi_{k,\lm}^\pm,\qquad k\in\dZ,\,\lm\in\cZ_k, 
\]
is an orthonormal basis of $L^2(\cM_\omg;\sin\tt\dd\tt\dd\phi;\dC^4)$, where $\{c_{k,\lm}\}_{k\in\dZ,\lm\in\cZ_k}$ are just normalizing constants.

\begin{lem}\label{lem:spinors}
	Let $k\in\dZ$ and let $\lm\in\cZ_k$ be fixed. Then the following hold:
	\begin{myenum}
		\item $\Psi_{k,\lm}^\pm\in\dom\sfK_\omg$;
		\item $\sfK_\omg\Psi_{k,\lm}^\pm = \pm\lm \Psi_{k,\lm}^\pm$;
		\item $\ii\aa_r\Psi_{k,\lm}^+ = -\Psi_{k,\lm}^-$ and 
		$\ii\aa_r\Psi_{k,\lm}^- = \Psi_{k,\lm}^+$ with $\aa_r$ as in~\eqref{eq:alpha_r}.
	\end{myenum}
\end{lem}
\begin{proof}
	(i) The fact that $\Psi_{k,\lm}^\pm\in C^\infty(\ov{\cM_\omg})$ follows from the
	$C^\infty$-smoothness of the Ferrers functions and the behaviour of the Ferrers functions as $x\arr0^+$ (\cite[Eqs. (14.3.1) and (15.2.1)]{DLMF}. Thus, it only remains to check the boundary condition.
	Let use the shorthand $\Psi := \Psi_{k,\lm}^\pm$. In view of~\eqref{eq:Komg} we need first to check that $\Psi$ satisfies
	\[
	\begin{pmatrix}
	\Psi_1(\omg,\phi)\\
	\Psi_2(\omg,\phi)
	\end{pmatrix} 
	=
	\begin{pmatrix}  
	\ii\sin\omg &-\ii\cos\omg e^{-\ii\phi} \\	-\ii\cos\omg e^{\ii\phi} &-\ii\sin\omg 
	\end{pmatrix}
	\begin{pmatrix}
	\Psi_3(\omg,\phi)\\
	\Psi_4(\omg,\phi)
	\end{pmatrix}.
	\] 
	Taking the structure of $\Psi$ into account, the latter is equivalent to
	\begin{equation}\label{eq:Psi_cond}
	\begin{pmatrix}
	\Psi_1(\omg,0)\\
	\Psi_2(\omg,0)
	\end{pmatrix} 
	=
	\begin{pmatrix}  
	\ii\sin\omg &-\ii\cos\omg \\	-\ii\cos\omg  &-\ii\sin\omg 
	\end{pmatrix}
	\begin{pmatrix}
	\Psi_3(\omg,0)\\
	\Psi_4(\omg,0)
	\end{pmatrix}.
	\end{equation}
	By its construction $\Psi(\cdot,0)\in\dom\wt{\sfT}_k$ with $\dom\wt{\sfT}_k$ specified in~\eqref{eq:wtTk}. Hence, it follows from~\eqref{eq:wtTk} that~\eqref{eq:Psi_cond} holds.
	
	\noindent (ii)
	Let $\Psi = \Psi_{k,\lm}^+$. The case $\Psi = \Psi_{k,\lm}^-$ is analogous.
	\[
		\sfK_\omg\Psi = \begin{pmatrix}
		(\wt{\sfT}_k\Psi(\cdot,0))_1 e^{\ii k\phi} \\
		(\wt{\sfT}_k\Psi(\cdot,0))_2 e^{\ii (k+1)\phi}\\
		(\wt{\sfT}_k\Psi(\cdot,0))_3 e^{\ii k\phi}\\
		(\wt{\sfT}_k\Psi(\cdot,0))_4 e^{\ii (k+1)\phi}
		\end{pmatrix} = 
		\lm\begin{pmatrix}
		\Psi_1(\cdot,0) e^{\ii k\phi} \\
		\Psi_2(\cdot,0) e^{\ii (k+1)\phi}\\
		\Psi_3(\cdot,0) e^{\ii k\phi}\\
		\Psi_4(\cdot,0) e^{\ii (k+1)\phi}
		\end{pmatrix} = \lm\Psi,
	\]
	where we used in the penultimate step that $\Psi(\cdot,0)$ is an eigenfunction of $\wt{\sfT}_k$ corresponding to the eigenvalue $\lm$.
	
	\noindent (iii) Let $k\ge0$. We will show first the identity $\ii\aa_r\Psi_{k,\lm}^+ = -\Psi_{k,\lm}^-$. To this aim we first verify that
	\[
	\begin{aligned}
		\begin{pmatrix}
		\cos\tt & e^{-\ii\phi}\sin\tt\\
		e^{\ii\phi}\sin\tt	&-\cos\tt
		\end{pmatrix}\Phi_{k,\lm}^+
		&=\begin{pmatrix}
		\big(\cos\tt\sfP_{\lm-1}^{-k}(\cos\tt)+\sin\tt(k-\lm+1)\sfP_{\lm-1}^{-k-1}(\cos\tt)\big)e^{\ii k\phi}\\
		\big(\sin\tt\sfP_{\lm-1}^{-k}(\cos\tt)-\cos\tt(k-\lm+1)\sfP_{\lm-1}^{-k-1}(\cos\tt)\big)e^{\ii (k+1)\phi}
		\end{pmatrix}\\
		&=
		\begin{pmatrix} \sfP_{-\lm-1}^{-k}(\cos\tt) e^{\ii k\phi}\\
		(\lm+k+1)\sfP^{-k-1}_{-\lm-1}(\cos\tt)e^{\ii(k+1)\phi}
		\end{pmatrix} = \Phi_{k,\lm}^{-},
	\end{aligned}	
	\]
	where we applied the identity \cite[8.735 (4)]{GR07} in the first  row and the identity~\cite[8.735(1)]{GR07} in the second row and used the relation $\sfP_\nu^\mu = \sfP_{-\nu-1}^\mu$.
	Hence, we obtain that
	\[
	\begin{aligned}
		\ii\aa_r\Psi_{k,\lm}^+&=
		c_{k,\lm}
		\left(\ii\begin{pmatrix}
		\cos\tt & e^{-\ii\phi}\sin\tt\\
		e^{\ii\phi}\sin\tt	&-\cos\tt
		\end{pmatrix}\Phi_{k,\lm}^+
		\right)\oplus
		\left(-\begin{pmatrix}
		\cos\tt & e^{-\ii\phi}\sin\tt\\
		e^{\ii\phi}\sin\tt	&-\cos\tt
		\end{pmatrix}\Phi_{k,\lm}^+
		\right)\\
		& = c_{k,\lm}(\ii\Phi_{k,\lm}^-)\oplus(-\Phi_{k,\lm}^-)=  -\Psi_{k,\lm}^-,	\end{aligned}	
	\]
	Next, we will show identity $
	\ii\aa_r\Psi_{k,\lm}^- = \Psi_{k,\lm}^+$.
         To this aim we  verify that
		\[
		\begin{aligned}
		\begin{pmatrix}
		\cos\tt & e^{-\ii\phi}\sin\tt\\
		e^{\ii\phi}\sin\tt	&-\cos\tt
		\end{pmatrix}\Phi_{k,\lm}^-
		&=\begin{pmatrix}
		\big(\cos\tt\sfP_{-\lm-1}^{-k}(\cos\tt)+\sin\tt(k+\lm+1)\sfP_{-\lm-1}^{-k-1}(\cos\tt)\big)e^{\ii k\phi}\\
		\big(\sin\tt\sfP_{-\lm-1}^{-k}(\cos\tt)-\cos\tt(k+\lm+1)\sfP_{-\lm-1}^{-k-1}(\cos\tt)\big)e^{\ii (k+1)\phi}
		\end{pmatrix}\\
		&=
		\begin{pmatrix} \sfP_{\lm-1}^{-k}(\cos\tt) e^{\ii k\phi}\\
		(k-\lm+1)\sfP^{-k-1}_{\lm-1}(\cos\tt)e^{\ii(k+1)\phi}
		\end{pmatrix} = \Phi_{k,\lm}^{+},
		\end{aligned}	
	\]
	where we again applied the identity \cite[8.735 (4)]{GR07} in the first row and the identity~\cite[8.735(1)]{GR07} in the second row and used the relation $\sfP_{\nu}^\mu = \sfP_{-\nu-1}^\mu$.
	
Hence, we obtain that
\[
\begin{aligned}
\ii\aa_r\Psi_{k,\lm}^-&= c_{k,\lm}
\left(\ii\begin{pmatrix}
\cos\tt & e^{-\ii\phi}\sin\tt\\
e^{\ii\phi}\sin\tt	&-\cos\tt
\end{pmatrix}\Phi_{k,\lm}^-
\right)\oplus
\left(\begin{pmatrix}
\cos\tt & e^{-\ii\phi}\sin\tt\\
e^{\ii\phi}\sin\tt	&-\cos\tt
\end{pmatrix}\Phi_{k,\lm}^-
\right)\\
& = c_{k,\lm}(\ii\Phi_{k,\lm}^+)\oplus(\Phi_{k,\lm}^+)=  \Psi_{k,\lm}^+,	\end{aligned}	
\]
	The proof in the case that $k\leq -1$ is analogous, using \cite[8.735 (1)]{GR07}, \cite[8.735 (4)]{GR07} and \cite[Eq. (14.9.5)]{DLMF}, so it will be omitted.
\end{proof}	
Let us introduce the orthogonal projectors
in the Hilbert space $L^2_{\rm sph}(\cC_\omg;\dC^4)$
\[
	(\Pi_{k,\lm}^\pm u)(r,\tt,\phi)\! := \!
	\Psi_{k,\lm}^\pm(\tt,\phi)
	\int_{\cM_\omg}
	u(r,\tt,\phi)\ov{\Psi_{k,\lm}^\pm(\tt,\phi)}\sin\tt\dd\tt\dd\phi,\quad k\in\dZ, \lm\in\cZ_k.
\]
These projectors induce the orthogonal decomposition
\begin{equation}\label{eq:decomposition}
	L^2_{\rm sph}(\cC_\omg;\dC^4) = \bigoplus_{k\in\dZ}\bigoplus_{\lm\in\cZ_k}
	\cE_{k,\lm},
\end{equation}
where the fiber spaces are defined as
\begin{equation}
	\cE_{k,\lm} := L^2(\dR_+;r^2\dd r)\otimes {\rm span}\,\{\Psi_{k,\lm}^+, \Psi_{k,\lm}^-\}\simeq L^2(\dR_+;r^2\dd r;\dC^2).
\end{equation}
For the sake of convenience we introduce the unitary transforms $\sfW_{k,\lm}\colon \cE_{k,\lm}\arr L^2(\dR_+;\dC^2)$, $\lm\in\cZ_k$, $k\in\dZ$
\[
	(\sfW_{k,\lm} u )(r) := r\begin{pmatrix} \left(u(r,\cdot,\cdot),\Psi_{k,\lm}^+\right)_{\cM_\omg}\\
	\left(u(r,\cdot,\cdot),\Psi_{k,\lm}^-\right)_{\cM_\omg}
	\end{pmatrix},
\]
where $(\cdot,\cdot)_{\cM_\omg}$ stands for inner product in $L^2(\cM_\omg;\sin\tt\dd\tt\dd\phi;\dC^4)$.
Now we have all the tools at our disposal to prove the main results of the paper.
\begin{proof}[Proof of Theorem~\ref{thm:main1}]
	Pick a function $u\in\dom\wt{\sfD}_\omg\cap\cE_{k,\lm}$. By definition, $u$ writes as
	\[
		u(r,\tt,\phi) = 
		\frac{\psi_+(r)}{r}\Psi_{k,\lm}^+(\tt,\phi)
		+
		\frac{\psi_-(r)}{r}\Psi_{k,\lm}^-(\tt,\phi)
	\]	
	with arbitrary $\psi_\pm \in C^\infty_0(\dR_+)$.
	Applying the operator $\wt{\sfD}_\omg$ in~\eqref{eq:Dirac_cone_spherical} to $u$
	and using Lemma~\ref{lem:spinors} we get
	\begin{equation}\label{eq:action}
	\begin{aligned}
		\wt\sfD_\omg u & = \frac{\ii\aa_r}{r}\left[
		\Psi_{k,\lm}^+\left(-\psi_+' + \frac{\lm\psi_+}{r}\right) + \Psi_{k,\lm}^-
		\left(-\psi_-' - \frac{\lm\psi_-}{r}\right)
		\right] \\
		&=
		\frac{1}{r}\left[
				\Psi_{k,\lm}^-\left(\psi_+' - \frac{\lm\psi_+}{r}\right) + \Psi_{k,\lm}^+
				\left(-\psi_-' - \frac{\lm\psi_-}{r}\right)
				\right].
	\end{aligned}
	\end{equation}
	Hence, we conclude that the inclusion $\wt{\sfD}_\omg\big(\dom\wt{\sfD}_\omg\cap\cE_{k,\lm}\big)\subset\cE_{k,\lm}$ holds. Thus, for any $\lm\in\cZ_k$ and $k\in\dZ$ the operators 
	\[
		d_{k,\lm}u:=\wt\sfD_\omg u,\qquad \dom d_{k,\lm} := \dom\wt{\sfD}_\omg\cap\cE_{k,\lm},
	\]
	are well defined. Moreover, relying on formula~\eqref{eq:action} we find that
	\[
		\sfW_{k,\lm}d_{k,\lm}\sfW_{k,\lm}^{-1} = \mathbf{d}_{\lm},\qquad \forall \lm\in\cZ_k, k\in\dZ.,
	\]
	where $\mathbf{d}_{\lm}$ is defined in~\eqref{eq:fibre}.
	Hence, we conclude that
	\[
		\wt{\sfD}_\omg\simeq\bigoplus_{k\in\dZ}\bigoplus_{\lm\in\cZ_k}\mathbf{d}_{\lm}.
	\]
	In view of the identity $\mathbf{d}_\lm = \s_2\mathbf{d}_{-\lm}\s_2$ we get that $\mathbf{d}_\lm$ and $\mathbf{d}_{-\lm}$ are unitarily equivalent. By 
	Proposition~\ref{prop:model_Dirac} and Corollary~\ref{cor:model_Dirac} we have
	that the simple spectrum of $\sfT_k$ is  given by  $\s(\sfT_k) = \cZ_k\cup(-\cZ_k)$, where $\cZ_k\cap(-\cZ_k)=\emptyset$ for all $k\in\dZ$.  Hence, we end up with
	\[
	\begin{aligned}
		\wt\sfD_\omg&\simeq \bigoplus_{k\in\dZ}
		\left[\left(\medoplus_{\lm\in\cZ_k\cap(0,\infty)}\mathbf{d}_\lm\right)\oplus
		\left(\medoplus_{\lm\in\cZ_k\cap(-\infty,0)}\mathbf{d}_\lm\right)\right]\\
		&\simeq
		\bigoplus_{k\in\dZ}
		\left[\left(\medoplus_{\lm\in\cZ_k\cap(0,\infty)}\mathbf{d}_\lm\right)\oplus
		\left(\medoplus_{\lm\in-\cZ_k\cap(0,\infty)}\mathbf{d}_\lm\right)\right]\simeq
	\bigoplus_{k\in\dZ}\bigoplus_{\lm\in\s(\sfT_k)\cap(0,\infty)}\mathbf{d}_\lm.
	\qedhere
	\end{aligned}
	\]
\end{proof}
\begin{proof}[Proof of Theorem~\ref{thm:main2}] 
	It follows from Proposition~\ref{prop:spectrum_outside2} that $\cZ_k\cap [-\frac12,\frac12] = \varnothing$ for any opening angle $\omg\in(0,\pi/2)$ and all $k\in\dZ$. Hence, all the fiber operators $\mathbf{d}_{\lm}$ in the orthogonal decomposition~\eqref{eq:Dirac_ortho} of $\wt{\sfD}_\omg$ are essentially self-adjoint by
	Proposition~\ref{prop:1D_Dirac}\,(ii). Thus, the Dirac operator $\wt{\sfD}_\omg$ and hence also $\sfD_\omg$ are essentially self-adjoint. 
	
	Now it remains to characterise the closure of $\sfD_\omg$.
	Let $\omg \in(0,\pi/2)$ be fixed.
	Let us introduce the symmetric densely defined operator
	in the Hilbert space $L^2(\cC_\omg;\dC^4)$
	\[
		\sfT u := \cD u,\qquad\dom\sfT = \big\{u\in H^1(\cC_\omg;\dC^4)\colon u|_{\p\cC_\omg} + \ii\beta(\aa\cdot\nu_\omg)u|_{\p\cC_\omg} = 0\big\},
	\]
	\cf Subsection~\ref{ssec:def}.  Our aim is to prove that $\sfT = \ov{\sfD_\omg}$. Since $\sfD_\omg$ is an essentially self-adjoint restriction of $\sfT$, it suffices to show that $\sfT$ is self-adjoint.  	Passing to the unitary equivalent operator $\wt\sfT = \sfV\sfT\sfV^{-1}$,
	we can repeat the construction in the proof of Theorem~\ref{thm:main1} with $\wt{\sfD}_\omg$ replaced by $\wt\sfT$.
	In this way we will get an orthogonal decomposition
	\[
		\wt{\sfT} \simeq \bigoplus_{k\in\dZ}\bigoplus_{\lm\in\cZ_k}\mathbf{t}_{k,\lm}
	\]
	where
	\[
	\begin{aligned}
		\mathbf{t}_{k,\lm}\psi &\!=\!\begin{pmatrix} 0 & -\frac{\dd}{\dd x} - \frac{\lm}{x}\\
		\frac{\dd}{\dd x} -\frac{\lm}{x} & 0
		\end{pmatrix}\psi,\\
		\dom\mathbf{t}_{k,\lm} &= \left\{\psi \!=\! (\psi_+,\psi_-)\colon\dR_+\arr\dC^2\colon \frac{\psi_+(r)}{r}\Psi_{k,\lm}^+(\tt,\varphi)\! +\! \frac{\psi_-(r)}{r}\Psi_{k,\lm}^-(\tt,\varphi)\in H^1_{\rm sph}(\cC_\omg;\dC^4)\right\},
		\end{aligned}
	\]
	are symmetric operators in the Hilbert space $L^2(\dR_+;\dC^2)$.
	On the other hand we know that $\ov{\wt{\sfD}_\omg}$ is self-adjoint and that the orthogonal decomposition
	\[
		\ov{\wt{\sfD}_\omg}\simeq\bigoplus_{k\in\dZ}\bigoplus_{\lm\in\cZ_k}\ov{\mathbf{d}_{\lm}}.
	\]
	holds.
	Since by Proposition~\ref{prop:spectrum_outside2} it holds that $\cZ_k\cap[-\frac12,\frac12] = \emptyset$, in view of Proposition~\ref{prop:1D_Dirac}\,(ii) in order to conclude the proof it it suffices to check that
	\[
		\dom\ov{\mathbf{d}_\lm} = H^1_0(\dR_+;\dC^2)\subset \dom\mathbf{t}_{k,\lm},\qquad k\in\dZ,\,\lm\in\cZ_k.
	\]
	To this aim it is enough show that
	for any $\psi \in H^1_0(\dR_+)$ we have
	\[
		v(r,\tt,\phi) = \frac{\psi(r)}{r}\Psi_{k,\lm}^\pm(\tt,\varphi)\in H^1_{\rm sph}(\cC_\omg;\dC^4).
	\]
	Using the expression for the gradient in the spherical coordinates we find that
	\[
	\begin{aligned}
		\|v\|^2_{H^1_{\rm sph}(\cC_\omg;\dC^4)}
		&=\int_0^\infty
		\left(\left|\psi'(r) -\frac{\psi(r)}{r}\right|^2 
		+ \frac{|\psi(r)|^2
		}{r^2}\int_{\cM_\omg}
		|(\nabla_{\dS^2}\Psi_{k,\lm}^\pm)(\tt,\varphi)|^2\sin\tt\dd\tt\dd\phi\right)\dd r	\\
		&\le\int_0^\infty\left(
		2|\psi'(r)|^2 
		+ \frac{|\psi(r)|^2
		}{r^2}\left(2+\int_{\cM_\omg}
		|(\nabla_{\dS^2}\Psi_{k,\lm}^\pm)(\tt,\varphi)|^2\sin\tt\dd\tt\dd\phi\right)\right)\dd r	\\
		& < \infty,
	\end{aligned}	
	\]
	where we used the Hardy inequality
	and that $\nabla_{\dS^2}\Psi_{k,\lm}^\pm$ is bounded in the last step.
\end{proof}

The proof of Proposition~\ref{prop:quantumdot} is inspired by the arguments for the analogous problem in the two-dimensional setting in \cite{PV19} and is a consequence of the following proposition and Remark~\ref{rmk:mit_minus}.
\begin{prop}
Let $\theta \in [0,2\pi) \setminus \{\frac{\pi}2,\frac{3\pi}2\}$; let moreover $\sfD_\omg$, $\sfD_\omg^{-}$ and $\sfD_\omg^\theta$ be defined as 
in~\eqref{eq:Dirac_operator}, \eqref{eq:Dirac_operator_-} and \eqref{eq:Dirac_operator_quantumdot} respectively. Then there exists a self-adjoint positive real matrix $M_{\tt}\in\dC^{4\times 4}$ such that 
\begin{equation*}
\sfD_\omg^{\theta} = 
\begin{cases}
M_\theta \sfD_\omg M_\theta, \quad &\theta \in [0,\pi/2) \cup (3\pi/2,2\pi), \\
M_\theta \sfD_\omg^- M_\theta, \quad &\theta \in (\pi/2, 3\pi/2).
\end{cases}
\end{equation*}
\end{prop}
\begin{proof}
From \eqref{eq:Dirac_operator_quantumdot}, for all $u \in \dom\sfD_\omg^{\theta}$ we have
\begin{equation*}
[I_4  - (\sin\theta)\beta ] \, u|_{\p\cC_\omg} = (\cos\theta) \, [\ii(\aa\cdot\nu_\omg)\beta] \, u|_{\p\cC_\omg}.
\end{equation*}
In our assumptions the matrix $I_4 - (\sin\theta)\beta$ is invertible, so we conclude that 
\eqref{eq:Dirac_operator_quantumdot} is equivalent to
\begin{equation}\label{eq:boundary_computation}
u|_{\p\cC_\omg} = 
\begin{pmatrix}
\frac{\cos\theta}{1-\sin\theta} I_2 & 0 \\
0 & \frac{\cos\theta}{1+\sin\theta} I_2
\end{pmatrix}
[\ii(\aa\cdot\nu_\omg)\beta] \, u|_{\p\cC_\omg}.
\end{equation}
Let $M_\theta$ be the self-adjoint real matrix
\begin{equation*}
M_\theta := 
\begin{pmatrix}
\sqrt{\frac{|\cos\theta|}{1+\sin\theta}} I_2 & 0 \\
0 & \sqrt{\frac{|\cos\theta|}{1-\sin\theta}} I_2
\end{pmatrix}.
\end{equation*}
In the case that $\theta \in (\pi/2, 3\pi/2)$, we get that
\[
\begin{pmatrix}
\frac{\cos\theta}{1-\sin\theta} I_2 & 0 \\
0 & \frac{\cos\theta}{1+\sin\theta} I_2
\end{pmatrix}
[\ii(\aa\cdot\nu_\omg)\beta] 
=
M_\theta^{-1} 
[ -\ii(\aa\cdot\nu_\omg)\beta]  M_\theta,
\]
so from \eqref{eq:boundary_computation} we conclude that
\[
M_\theta u|_{\p\cC_\omg} = [ -\ii(\aa\cdot\nu_\omg)\beta]  M_\theta  u|_{\p\cC_\omg}.
\]
Thanks to the last equation, it is easy to show that
$\sfD_\omg^{\theta} = M_\theta \sfD_\omg^{-} M_\theta$. 
 In the case that $\theta \in [0,\pi/2) \cup (3\pi/2,2\pi)$, we have from \eqref{eq:boundary_computation} that
\[
M_\theta u|_{\p\cC_\omg} = [ \ii(\aa\cdot\nu_\omg)\beta]  M_\theta  u|_{\p\cC_\omg}, 
\]
and from this it is then easy to conclude that $\sfD_\omg^{\theta} = M_\theta \sfD_\omg M_\theta$.
\end{proof}

\begin{proof}[Proof of Proposition \ref{prop:estimate.norm.resolvent}]
Let $u\in\dom\wt{\sfD}_\omg$: from \eqref{eq:decomposition}, $u$ writes as
\begin{equation}\label{eq:hardy.decomposition}
		u(r,\tt,\phi) = \sum_{k\in\dZ}\sum_{\lm\in\cZ_k}
		\frac{\psi_{k,\lm}^+(r)}{r}\Psi_{k,\lm}^+(\tt,\phi)
		+
		\frac{\psi_{k,\lm}^-(r)}{r}\Psi_{k,\lm}^-(\tt,\phi)
\end{equation}
with $\psi_{k,\lm}^\pm\in C^\infty_0(\dR_+)$ for $k\in\dZ$ and $\lm\in\cZ_k$. First of all, observe that
\begin{equation}\label{eq:rhs}
	\int_{\cC_\omg} \frac{|u(x)|^2}{|x|^2}\dd x
	= \sum_{k\in\dZ}\sum_{\lm\in\cZ_k}
	\int_0^\infty\left[\frac{|\psi_{k,\lm}^+(r)|^2}{r^2}+\frac{|\psi_{k,\lm}^-(r)|^2}{r^2}   \right]\dd r. 
\end{equation}
According to the proof of Theorem \ref{thm:main1}, we have that
\begin{equation}\label{eq:hardy.starting}
\begin{split}
\int_{\cC_\omg}|(\sfD_\omg u)(x)|^2 \,\dd x 
 =  \sum_{k\in\dZ}\sum_{\lm\in\cZ_k} &
 \bigg[
 \int_0^\infty
\bigg\vert(\psi_{k,\lm}^+)'(r) - \frac{\lm\psi_{k,\lm}^+(r)}{r}\bigg\vert^2  \dd r
\\
&  +
\int_0^\infty \bigg\vert (\psi_{k,\lm}^-)'(r) + \frac{\lm\psi_{k,\lm}^-(r)}{r}\bigg\vert^2   \dd r
\bigg].
\end{split}
\end{equation}
With an explicit computation, for all $k \in \dZ$ and $\lm \in \cZ_k$ we have
\begin{equation*}
\begin{split}
\int_0^\infty
\left\vert (\psi_{k,\lm}^\pm)'(r) \mp \frac{\lm\psi_{k,\lm}^\pm(r)}{r}\right\vert^2  \dd r
&= \int_0^\infty
\left|(\psi_{k,\lm}^\pm)'(r)\right|^2\dd r+ \lm^2\int_0^\infty
	\frac{|\psi_{k,\lm}^\pm(r)|^2}{r^2}\dd r\\
&\qquad\mp 2\lm\Re\int_0^\infty 
\frac{(\psi_{k,\lm}^\pm)'(r)\ov{\psi_{k,\lm}^\pm(r)}}{r}\dd r. 
\end{split}
\end{equation*}
Since 
\begin{equation*}
2\Re \int_0^{\infty}
\frac{(\psi_{k,\lm}^\pm)'(r)\ov{\psi_{k,\lm}^\pm(r)}}{r}
 \dd r
=
\int_0^{\infty}
\frac{(|\psi_{k,\lm}^\pm(r)|^2)'}{r}
\dd r
=
\int_0^{\infty}
\frac{|\psi_{k,\lm}^\pm(r)|^2}{r^2}\dd r
\end{equation*}
we get using the Hardy inequality for all $k \in \dZ$ and $\lm \in \cZ_k$
\begin{equation}\label{eq:hardy.each.component}
\int_0^\infty
\left\vert (\psi_{k,\lm}^\pm)'(r) \mp \frac{\lm\psi_{k,\lm}^\pm(r)}{r}\right\vert^2  \dd r
\geq 
\left(\lm\mp \frac12\right)^2 \int_0^\infty
\frac{|\psi_{k,\lm}^\pm(r)|^2}{r}\dd r.
\end{equation}
By Proposition \ref{prop:spectrum_outside2}, $\lm^2\geq (\tfrac{\pi}{4\omega} + \tfrac12 )^2$  for all $k \in \dZ$ and $\lambda \in \cZ_k$, so gathering~\eqref{eq:rhs}, \eqref{eq:hardy.decomposition}, \eqref{eq:hardy.starting} and \eqref{eq:hardy.each.component}, we get 
\[
	\int_{\cC_\omg}|(\sfD_\omg u)(x)|^2\dd x
	\ge \left(\frac{\pi}{4\omg}\right)^2 \int_{\cC_\omg}\frac{|u(x)|^2}{|x|^2}\dd x
\]
 and conclude the proof.
\end{proof}
\begin{proof}[Proof of Theorem~\ref{thm:perturbation}]
	Let $\omg \in (0,\pi/2)$.
	Recall that by Theorem~\ref{thm:main1} the operator $\sfD_\omg$ is essentially self-adjoint in the Hilbert space $L^2(\cC_\omg;\dC^4)$. By the inequality
	in Proposition~\ref{prop:estimate.norm.resolvent} for any Hermitian $\V\colon\cC_\omg\arr\dC^{4\times 4}$ satisfying the bound~\eqref{eq:bnd_V} the condition
	\[
		\int_{\cC_\omg}\big|\V(x) u(x)\big|^2\dd x
		\le
		\int_{\cC_\omg}\big|\V(x)\big|^2 |u(x)|^2\dd x \leq
		\left( \frac{4\omg}{\pi} \sup_{x \in \cC_\omg} |x| |\V(x)|  \right )^2 
		\int_{\cC_\omg}|(\sfD_\omg u)(x)|^2\dd x
	\]
	holds for all $u\in \dom\sfD_\omg$.
In other words, the operator of multiplication with the matrix-valued function $\V$ is bounded with respect to the operator $\sfD_\omg$ with the bound $<1$ or $\leq 1$ respectively when $\nu < \frac{\pi}{4\omg}$ or $\nu \leq \frac{\pi}{4\omg}$.
	Hence, we conclude the statements from the Kato-Rellich theorem~\cite[Chap. V, Thm. 4.4]{Kato} and the W\"{u}st theorem~\cite[Chap. V, Thm. 4.6]{Kato},
respectively.
\end{proof}

\subsection*{Acknowledgement}
BC is member of GNAMPA (INDAM). VL acknowledges the support by the grant No.~21-07129S 
of the Czech Science Foundation (GA\v{C}R).
The authors are also grateful to the anonymous referees for valuable suggestions.
In particular, one of these suggestions
led to Remark~\ref{rem:anti} with an alternative argument based on the anti-commutation relation 
for the symmetry of the spectra of the fiber operators.

\begin{appendix}

\section{Partial derivatives in the spherical coordinates}\label{app:pderiv}
The aim of this appendix is to express partial derivatives $\p_j$, $j=1,2,3$ in terms of $\p_r,\p_\tt,\p_\phi$. This material is standard and we provide it only for convenience of the reader.

Employing the chain rule for the differentiation
and using the identities~\eqref{eq:x123} we can express $\p_r,\p_\tt$ and $\p_\phi$ through $\p_1,\p_2$ and $\p_3$ as follows
\begin{equation}\label{eq:prttphi}
\begin{cases}
\p_r  = \cos\phi\sin\tt\p_1 + \sin\phi\sin\tt\p_2 + \cos\tt\p_3,\\
\p_\tt  = r\cos\phi\cos\tt\p_1 +r\sin\phi\cos\tt\p_2 - r\sin\tt\p_3,\\
\p_\phi = -r\sin\phi\sin\tt\p_1 + r\cos\phi\sin\tt\p_2.
\end{cases}	
\end{equation}
It is a standard routine procedure to express $\p_1,\p_2,\p_3$ through $\p_r,\p_\tt,\p_\phi$.
Since it is an important ingredient of our analysis, we outline this derivation below.
First, we divide the expression for $\p_\tt$ and $\p_\phi$ in~\eqref{eq:prttphi} by $r$
\begin{equation}\label{eq:pttphir}
\begin{cases}
\frac{\p_\tt}{r}  = \cos\phi\cos\tt\p_1 +\sin\phi\cos\tt\p_2 - \sin\tt\p_3,\\
\frac{\p_\phi}{r} = -\sin\phi\sin\tt\p_1 + \cos\phi\sin\tt\p_2.
\end{cases}
\end{equation}
Combining the expression for $\p_r$ in~\eqref{eq:prttphi} and the expression for $\frac{\p_\tt}{r}$ in \eqref{eq:pttphir} we obtain
\begin{equation}
\sin\tt\p_r + \frac{\cos\tt\p_\tt}{r}  = \cos\phi\p_1 + \sin\phi\p_2.
\end{equation}
From the above equation and the expression for $\frac{\p_\phi}{r}$ in~\eqref{eq:pttphir} we get
\[
\sin\phi\sin^2\tt\p_r + \frac{\sin\phi\sin\tt\cos\tt\p_\tt}{r} + \frac{\cos\phi\p_\phi}{r} = \sin\tt\p_2. 
\]
Hence, we get the expression for $\p_2$  dividing the above equation by $\sin\tt$
\begin{equation}\label{eq:p2}
\p_2 = \sin\phi\sin\tt\p_r + \frac{\sin\phi\cos\tt\p_\tt}{r} + \frac{\cos\phi}{\sin\tt}\frac{\p_\phi}{r}.
\end{equation}
Combining the above expression for $\p_2$ with the expression for $\frac{\p_\phi}{r}$ in~\eqref{eq:pttphir} we get the expression for $\p_1$
\begin{equation}\label{eq:p1}
\begin{aligned}
\p_1&  =
\frac{1}{\sin\phi\sin\tt}\left(
- \frac{\p_\phi}{r}+ \cos\phi\sin\tt\p_2 \right)\\
& = - \frac{1}{\sin\phi\sin\tt}\frac{\p_\phi}{ r} + \frac{\cos\phi}{\sin\phi}\left(\sin\phi\sin\tt\p_r + \frac{\sin\phi\cos\tt\p_\tt}{r} + \frac{\cos\phi}{\sin\tt}\frac{\p_\phi}{r}\right)\\
& = - \frac{1}{\sin\phi\sin\tt}\frac{\p_\phi} {r} +
\cos\phi\sin\tt \p_r + \frac{\cos\phi\cos\tt\p_\tt}{r} + \frac{\cos^2\phi}{\sin\phi\sin\tt}\frac{\p_\phi}{r}\\
& =
\cos\phi\sin\tt \p_r
+
\frac{\cos\phi\cos\tt\p_\tt}{r}  
- 
\frac{\sin\phi}{\sin\tt}\frac{\p_\phi}{ r}.
\end{aligned}	
\end{equation}
Finally, we obtain a formula for $\p_3$ substituting~\eqref{eq:p1} and~\eqref{eq:p2} into the expression for $\p_r$ in~\eqref{eq:prttphi}
\[
\begin{aligned}
\p_3 & =\frac{1}{\cos\tt}\p_r - \frac{\cos\phi\sin\tt}{\cos\tt}\left(
\cos\phi\sin\tt \p_r
+
\frac{\cos\phi\cos\tt\p_\tt}{r}  
- 
\frac{\sin\phi}{\sin\tt}\frac{\p_\phi}{ r} 	
\right)\\
& \qquad - \frac{\sin\phi\sin\tt}{\cos\tt}\left(\sin\phi\sin\tt\p_r + \frac{\sin\phi\cos\tt\p_\tt}{r} + \frac{\cos\phi}{\sin\tt}\frac{\p_\phi}{r}\right)\\
& =
\left(\frac{1}{\cos\tt} - \frac{\cos^2\phi\sin^2\tt}{\cos\tt}-\frac{\sin^2\phi\sin^2\tt}{\cos\tt} \right)\p_r-
\left(\cos^2\phi\sin\tt +\sin^2\phi\sin\tt\right)\frac{\p_\tt}{r}
\\
&\qquad + \left(\frac{\cos\phi\sin\phi}{\cos\tt}- \frac{\sin\phi\cos\phi}{\cos\tt}\right)\frac{\p_\phi}{r}\\
& = \cos\tt\p_r -\frac{\sin\tt\p_\tt}{r}.
\end{aligned}
\]
Now we summarize the expressions that we obtained
\begin{equation}\label{eq:p123app}
\begin{cases}
\p_1 = \cos\phi\sin\tt \p_r
+
\frac{\cos\phi\cos\tt\p_\tt}{r}  
- 
\frac{\sin\phi}{\sin\tt}\frac{\p_\phi}{ r}, 
\\
\p_2 = \sin\phi\sin\tt\p_r + \frac{\sin\phi\cos\tt\p_\tt}{r} + \frac{\cos\phi}{\sin\tt}\frac{\p_\phi}{r},
\\
\p_3 = \cos\tt\p_r -\frac{\sin\tt\p_\tt}{r}.
\end{cases}
\end{equation}

\section{Dirac systems}\label{app:Dirac}
In this appendix we provide basic facts on deficiency indices of one-dimensional Dirac operators on intervals. We follow the constructions given in \cite{W87} and \cite[Chap. 15]{W03}.

First, we introduce on the interval $\cI := (0,b)$ with $b\in(0,\infty]$ a class of one-dimensional first-order Dirac differential expressions
\begin{equation}\label{eq:differential_expression}
	\tau f(x) = \begin{pmatrix} 0 & 1\\-1 & 0\end{pmatrix} f'(x) + q(x)f(x),
\end{equation}
where $q\colon \cI\arr\dR^{2\times 2}$ is measurable with $q(x)$ symmetric almost everywhere in $\cI$ and $\|q(\cdot)\| \in L^1_{\rm loc}(\cI)$. Throughout this section we denote the inner product in $L^2(\cI;\dC^2)$ by $\langle\cdot,\cdot\rangle$.

We associate with the differential expression $\tau$ the maximal operator in the Hilbert space $L^2(\cI;\dC^2)$
\begin{equation}\label{eq:maximal}
	\sfT f := \tau f,\qquad\dom\sfT = \{f\in L^2(\cI;\dC^2)\colon f\in {\rm AC}_{\rm loc}(\cI;\dC^2), \tau f\in L^2(\cI;\dC^2)\},
\end{equation}
where ${\rm AC}_{\rm loc}(\cI;\dC^2)$ stands for the space of locally absolutely continuous two-component functions on the interval $\cI$.

We define the preminimal operator in $L^2(\cI;\dC^2)$ by
\begin{equation}\label{eq:preminimal}
	\sfT_0'' f := \tau f,\qquad
		\dom\sfT_0'' = C^\infty_0(\cI;\dC^2).
\end{equation}
According to \cite[Chap. 3]{W87} the preminimal operator is closable and its closure 
\begin{equation}\label{eq:minimal}
	\sfT_0 := \ov{\sfT_0''}
\end{equation}
is the minimal symmetric operator in $L^2(\cI;\dC^2)$.

Let us define the Lagrange brackets for (locally) absolutely continuous functions $f,g\colon \cI\arr\dC^2$ by
\begin{equation}\label{eq:bracket}
	[f,g]_x := f_2(x)\ov{g_1(x)} - f_1(x)\ov{g_2(x)}.
\end{equation}
Using the concept of the Lagrange bracket we can provide the Green's identity for $\sfT$.
\begin{prop}[Green's formula, {\cite[Satz 15.3]{W03}}]\label{prop:Green} 
	Let the maximal operator $\sfT$ be as in~\eqref{eq:maximal}. Let the Lagrange bracket $[f,g]_x$ be as in~\eqref{eq:bracket}.
	Then for any $f,g\in\dom\sfT$ there exist the limits
	\[
		[f,g]_0 := \lim_{x\arr 0^+} [f,g]_x,\qquad
		[f,g]_b := \lim_{x\arr b^-}[f,g]_x,
	\]
	and it holds that
	\[
		\langle \sfT f,g\rangle - \langle f,\sfT g\rangle = [f,g]_b - [f,g]_0.
	\]
\end{prop}
The domain of the closed symmetric operator $\sfT_0$ can be explicitly characterised.
\begin{prop}\cite[Thm. 3.11]{W87}\label{prop:domT0}
	For $z\in\dC\sm\dR$ we set $\cN_z := \ker(\sfT-z)$, then
	\[
	\begin{split}
		\dom\sfT_0 & = \left\{f\in \dom\sfT\colon [g,f]_0 = [g,f]_b = 0\,\, \forall\, g\in\dom\sfT\right\}\\
		&=\left\{f\in \dom\sfT\colon [g,f]_0 = [g,f]_b = 0\,\, \forall\, g\in\cN_z +\cN_{\ov z}\right\}.
	\end{split}	
	\]
\end{prop}

The endpoints of the interval $\cI$ admit some classifications. First, one can classify whether the endpoints are regular or singular.
\begin{dfn}\label{dfn:singular_regular}
	Let the differential $\tau$ be as in~\eqref{eq:differential_expression}.
	Then $\tau$ is called \emph{regular} at $x = 0$ if there exists $c \in\cI$ so that $\|q(\cdot)\|$ is integrable on $[0,c]$. 
	The expression $\tau$ is called \emph{regular} at $x = b$ if $b <\infty$ and there exists $c \in\cI$ so that $\|q(\cdot)\|$ is integrable on $[c,b]$. 
	If $\tau$ is not regular, respectively, at $x = 0$ and/or at $x=b$, then it is called \emph{singular} at $x = 0$ and/or $x = b$, respectively.
\end{dfn}
The functions in $\dom\sfT$ have certain regularity in the neighbourhood of a regular endpoint. This is clarified in the next proposition.
\begin{prop}\cite[Satz 15.4]{W03}\label{prop:regular}
	Let the maximal operator $\sfT$ be as in~\eqref{eq:maximal}.
	Assume that $\tau$ is regular at $x = b$. Then the following hold.
	\begin{myenum}
        \item For any $f\in\dom\sfT$ there exists a finite
         limit $f(b):=\lim_{x\arr b^-}f(x)$
		\item The mapping $\dom\sfT\ni f\mapsto f(b)\in\dC^2$ is surjective.
	\end{myenum}
	Analogous statements hold if $\tau$ is regular at $x = 0$.
\end{prop}
The connection between the maximal operator and the minimal and preminimal operators is revealed in the next proposition.
\begin{prop}\label{prop:T0adjoint}
	Let the maximal operator $\sfT$, the preminimal operator $\sfT_0''$ and the minimal operator $\sfT_0$ be as in~\eqref{eq:maximal},~\eqref{eq:preminimal}, and~\eqref{eq:minimal}, respectively. Then 
	\[
		(\sfT_0'')^* = \sfT_0^* = \sfT.
	\]	
\end{prop}
Next we define the concepts of being right or left in $L^2(\cI;\dC^2)$.
\begin{dfn}\label{dfn:leftright}
	Let $f\colon\cI\arr\dC^2$ be measurable. One says that $f$ lies "left in $L^2(\cI;\dC^2)$" if there exists $c\in\cI$ such that $f|_{(0,c)} \in L^2((0,c);\dC^2)$. One says that $f$ lies "right in $L^2(\cI;\dC^2)$" if there exists $c\in\cI$ such that $f|_{(c,b)} \in L^2((c,b);\dC^2)$
\end{dfn}
The solutions of the differential equation $(\tau - z)f = 0$ obey a dichotomy with respect to the property stated in Definition~\ref{dfn:leftright}.
\begin{prop}[Weyl's alternative]
	\label{prop:Weyl}
	For any Dirac differential expression $\tau$ as in~\eqref{eq:differential_expression} on the interval $\cI$ one of the following two alternatives holds.
	\begin{myenum}
	\item For all $z\in\dC$ every solution of $(\tau-z)f = 0$ lies left in $L^2(\cI;\dC^2)$.
	\item For all $z\in\dC$ there is at least one solution of $(\tau -z)f = 0$ that does not lie left in $L^2(\cI;\dC^2)$. In this case there exists for any $z\in\dC\sm\dR$ exactly one solution of $(\tau -z)f = 0$ (up to a constant factor) that lies "left in $L^2(\cI;\dC^2)$".
	\end{myenum}
	The same alternative holds for ``right in $L^2(\cI;\dC^2)$''.
\end{prop}
The Weyl's alternative stated above yields yet another classification of the endpoints of the interval $\cI$ with respect to the differential expression $\tau$.
\begin{dfn}\label{dfn:lplc}
	If the first alternative in Proposition~\ref{prop:Weyl} holds, then one says that $\tau$ is limit-circle at $x = 0$ (respectively, at $x = b$). If the second alternative in Proposition~\ref{prop:Weyl} holds, the one says that $\tau$ is limit-point at $x = 0$ (respectively, at $x = b$).
\end{dfn}
The limit-point/limit-circle classification provides extra properties of functions in $\dom\sfT$ at the endpoints and allows to compute the deficiency indices of the closed symmetric operator $\sfT_0$.
\begin{prop}\label{prop:lplc}
	Let $\tau$ be the Dirac differential expression as in~\eqref{eq:differential_expression} on the interval $\cI$. Let the maximal operator $\sfT$ and the minimal operator $\sfT_0$ be as in~\eqref{eq:maximal} and~\eqref{eq:minimal}, respectively.
	\begin{myenum}
	\item If $\tau$ is limit-point at $x = 0$ (respectively, at $x = b$) then $[f,g]_0 = 0$ (respectively, $[f,g]_b = 0$) for all $f,g\in\dom\sfT$.
	\item The deficiency indices of $\sfT_0$ are
	\begin{myenum}
	\item [-] $(2,2)$ if $\tau$ is limit-circle at $x = 0$ and $x = b$.
	\item [-] $(1,1)$ if $\tau$ is limit-circle at either $x = 0$ or at $x= b$ and limit-point at the other endpoint of $\cI$.
	\item [-] $(0,0)$ if $\tau$ is limit-point at both endpoints of $\cI$.
	\end{myenum}
	\end{myenum}
\end{prop}

\section{Numerical results}\label{app:plots}
In Remark \ref{rmk:conjecture} we conjectured that Theorem~\ref{thm:main2} is in fact valid for all $\omg \in (0,\pi)\setminus\{\pi/2\}$; in this appendix we provide a numerical evidence to support this claim.

The key step in the proof of Theorem \ref{thm:main2} for all $\omega \in (0,\pi)\setminus\{\pi/2\}$ would be showing the property
\begin{equation}\label{eq:to_hold}
\sigma(\sfT_k) \cap \left[-\frac12,\frac12\right] = \emptyset \quad \text{ for all }k \in \mathbb{Z} \text{ and } \omg \in (0,\pi)\setminus\left\{\frac{\pi}2\right\}.
\end{equation}
In our analytical results we are able to show it only for $\omg \in (0,\pi/2)$: this poses a limitation in the statement of Theorem \ref{thm:main2}.

In fact, \eqref{eq:to_hold} holds for $k \in \mathbb{Z} \setminus \{-1,0\}$ and all $\omg \in (0,\pi)\setminus\{\pi/2\}$, since it is true that $\lambda \in \sigma(\sfT_k) = \cZ_k\cup(-\cZ_k) \implies |\lambda|\geq (\sqrt{3}+1)/2 \approx 1.37$, see Remark \ref{rem:sharpness}.  It remains to consider the cases $k = -1,0$ and $\omg \in (\pi/2,\pi)$.
We remark that $\sigma(\sfT_k) = \cZ_k\cup(-\cZ_k)$ is a discrete set and it is expressed by the
solutions of the transcendental equation in \eqref{eq:Zk}: thanks to  \cite[Eq.~(14.9.5)]{DLMF}, it is true that  $\cZ_{-1} = -\cZ_{0}$, so it is enough to treat the case $k = 0$ and $\omg \in (\pi/2,\pi)$. We provide here a plot where we draw the set $\cZ_0$ for all $\omg\in(0,\pi)\setminus \{\frac{\pi}{2}\}$ and understand if \eqref{eq:to_hold} holds true.

In Figure \ref{fig:plot0}, the set $\cZ_{0}$ is plotted with a black line for $\omega \in (0,\pi)\setminus\{\frac{\pi}2\}$. For $\omega < \frac{\pi}2$,  Proposition~\ref{prop:spectrum_outside2} implies the spectral bound $\lambda \in \cZ_{0} \implies |\lambda| \geq \frac{\pi}{4 \omega} + \frac12$: the curves $\lambda = \pm\left(  \frac{\pi}{4 \omega} + \frac12\right)$ are plotted with dotted green curves. Finally, the lines $\lambda = 1/2$ and $\lambda = -1/2$ are plotted with  dashed red lines.

It is evident from Figure \ref{fig:plot0} that $\cZ_0 \cap [-1/2,1/2] = \emptyset$, so Theorem \ref{thm:main2} should be true for all $\omg \in (0,\pi)\setminus\{\pi/2\}$.
\begin{figure}[H]
\includegraphics[width=100mm]{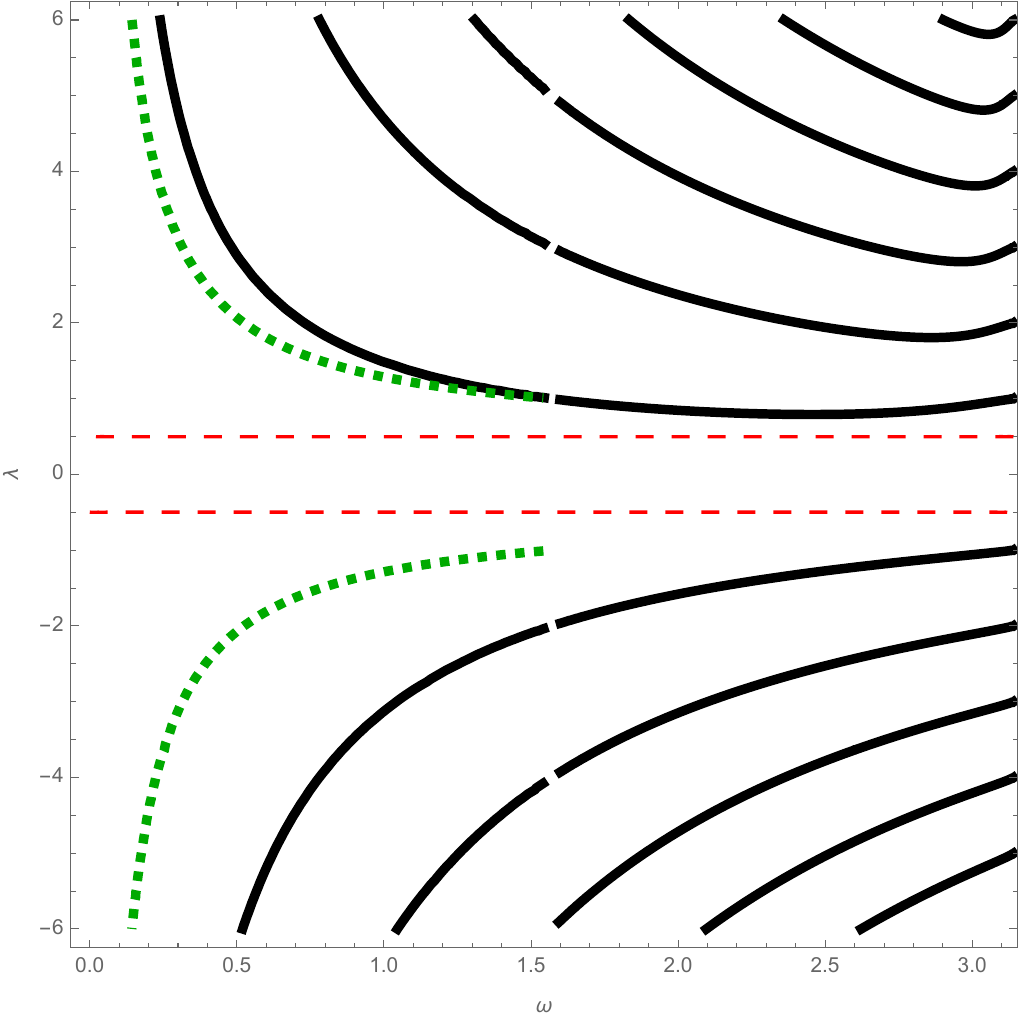}
\caption{The set $\cZ_{0}$ is plotted with black lines; the curves $\lambda = \pm( \frac{\pi}{4 \omega} + \frac12)$ are plotted with dotted green curves;  the lines $\lambda = \pm1/2$ are plotted with  dashed red lines.}
\label{fig:plot0}
\end{figure}
\begin{remark}
	We observe from numerics that the eigenvalues of $\sfT_0$ converge to the eigenvalues of the principal fiber of the spin-orbit operator corresponding to the Dirac operators on the half-space $\dR^3_+$ (with MIT bag boundary conditions) and in the full space $\dR^3$ as $\omg\arr\frac{\pi}{2}$ and $\omg\arr\pi^-$, respectively. This phenomenon can be an indication of the underline convergence of the spin-orbit operators in the generalized norm resolvent sense.
\end{remark}
\end{appendix}

\newcommand{\etalchar}[1]{$^{#1}$}


\begin{thebibliography}{\textsc{BCD{\etalchar{+}}72}}
\bibitem[AS64]{AS64}
M.~Abramowitz and I.~Stegun,
\emph{Handbook of mathematical functions with formulas, graphs, and
	mathematical tables}, Washington, D.C., 1964.



\bibitem[ALR17]{ALR17}
N.~Arrizabalaga, L.~Le Treust, and N.~Raymond, 
On the MIT bag model in the
non-relativistic limit, {\it Commun. Math. Phys.} {\bf 354} (2017),  641--669.

\bibitem[ALR20]{ALR20}
N.~Arrizabalaga, L.~Le Treust, and N.~Raymond, 
Extension operator for the MIT bag model,
{\it Ann. Fac. Sci. Toulouse, Math. (6)} {\bf 29} (2020), 135--147.

\bibitem[AMV14]{AMV14}
N.~Arrizabalaga, A.~Mas, and L.~Vega, 
Shell interactions for Dirac operators, 
{\it J. Math. Pures Appl.} {\bf 102} (2014),
617--639.

\bibitem[AMV15]{AMV15}
N.~Arrizabalaga, A.~Mas, and L.~Vega, 
Shell interactions for Dirac operators: on
the point spectrum and the confinement,
{\it SIAM J. Math. Anal.} {\bf 47} 
(2015), 1044--1069.

\bibitem[AMV16]{AMV16}
N.~Arrizabalaga, A.~Mas, and L.~Vega, 
An isoperimetric-type inequality for electrostatic shell interactions for Dirac operators, {\it Commun. Math. Phys.} {\bf 344}
 (2016), 483--505.


\bibitem[BEHL16]{BEHL16}
J.~Behrndt, P.~Exner, M.~Holzmann, and V.~Lotoreichik,
On the spectral properties of Dirac operators with electrostatic $\dl$-shell interactions, 
{\it J. Math. Pures Appl.} {\bf 111} (2018),
47--78.

\bibitem[BEHL19]{BEHL19}
J.~Behrndt, P.~Exner, M.~Holzmann, and V.~Lotoreichik, On Dirac operators in $\dR^3$
with electrostatic and Lorentz scalar $\dl$-shell interactions, {\it Quantum Stud. Math.
Found.} {\bf 6}  (2019), 295--314.

\bibitem[BHdS20]{BHdS20}
J.~Behrndt, S.~Hassi, and H.~de Snoo, 
\emph{Boundary value problems, Weyl functions, and differential operators}, Birkh\"{a}user, Cham, 2020.

\bibitem[BHM20]{BHM20}
J.~Behrndt, M.~Holzmann, and A.~Mas,
Self-adjoint Dirac operators on domains in $\dR^3$, {\it Ann. Henri Poincar\'{e}} {\bf 21} (2020), 2681--2735.
%

\bibitem[BHSS21]{BHSS21}
J.~Behrndt, M.~Holzmann, C.~Stelzer,
and G.~Stenzel,
A class of singular perturbations of the Dirac operator: boundary triplets and Weyl functions,
\emph{Acta Wasaensia} {\bf 462} (2021), Festschrift in honor of Seppo Hassi, 15--36.


\bibitem[B22a]{B21a}
B.~Benhellal, Spectral properties of the Dirac operator coupled with $\delta$-shell interactions, {\it Lett. Math. Phys.}
{\bf 112} (2022), 52.

\bibitem[B22b]{B21b}
B.~Benhellal, Spectral analysis of Dirac operators with delta interactions supported on the boundaries of rough domains, 
{\it J. Math. Phys.} {\bf 63}  (2022), 011507.


\bibitem[BDY99]{BDY99}
C.~Bernardi, M.~Dauge, and Y.~Maday, 
\emph{Spectral methods for axisymmetric domains. Numerical algorithms and tests due to Mejdi Aza\"{i}ez.}
Gauthier-Villars/North Holland, Paris, 1999.


\bibitem[BS62]{BS62}
M.\,Sh.~Birman and G.\,E.~Skvortsov, On square summability of highest derivatives of the solution of the Dirichlet problem in a domain with piecewise smooth boundary, {\it Izv. Vyssh. Uchebn. Zaved. Mat.} (1962), 12--21.

\bibitem[Bru76]{Br76} V.\,M.~Bruk,
A certain class of boundary value problems with a spectral parameter in the
boundary condition,
\textit{Mat.\ Sb.\ (N.S.)} \textbf{100 (142)} (1976), 210--216 (in Russian);
translated in: \textit{Math.\ USSR-Sb.} \textbf{29} (1976), 186--192.

\bibitem[CL20]{CL20}
 B.~Cassano and V.~Lotoreichik, 
Self-adjoint extensions of the two-valley Dirac operator with
discontinuous infinite mass boundary conditions,
{\it Operators and Matrices} {\bf 14} (2020), 667--678.

\bibitem[CLMT21]{CLMT21}
B.~Cassano, V.~Lotoreichik, A.~Mas and M.~Tu\v{s}ek, General $\delta $-shell interactions for the two-dimensional Dirac
operator: self-adjointness and approximation, {\it to appear in Rev. Mat. Iberoam.}, \texttt{arXiv:2102.09988}.

\bibitem[CP18]{CP18}
B.~Cassano and F.~Pizzichillo, 
Self-adjoint extensions for the Dirac operator with Coulomb-type
spherically symmetric potentials, {\it Lett. Math. Phys.} {\bf 108}
(2018), 2635--2667.

\bibitem[CP19]{CP19}
B.~Cassano and F.~Pizzichillo, 
Boundary triples for the Dirac operator with Coulomb-type spherically
symmetric perturbations,
{\it J.~Math.~Phys.}, {\bf 60} (2019), 041502. 

\bibitem[C75]{C75}
A.~Chodos, Field-theoretic Lagrangian with baglike solutions, \emph{Phys. Rev. D}
{\bf12} (1975),  2397--2406.

\bibitem[CJJT74]{CJJT74}
A.~Chodos, R.\,L.~Jaffe, and K.~Johnson, 
and C.\,B.~Thorn, 
Baryon structure in the bag
theory, {\it Phys. Rev. D} {\bf 10} (1974),
2599--2604.

\bibitem[CJJ$^+$74]{CJJTW74}
A.~Chodos, R.\,L.~Jaffe, K.~Johnson, C.\,B.~Thorn, and V.\,F.~Weisskopf, 
New extended model of hadrons, {\it Phys. Rev. D} {\bf 9}  (1974), 3471--3495.

\bibitem[D88]{D88}
M.~Dauge, \emph{Elliptic boundary value problems on corner domains. Smoothness and asymptotics of solutions.},  Springer-Verlag,  Berlin, 1988.



\bibitem[DJJK75]{DJJK75}
T.~DeGrand, R.\,L.~Jaffe, K.~Johnson, and J.~Kiskis, Masses and other parameters
of the light hadrons, {\it Phys. Rev. D} {\bf 12} (1975), 2060--2076.

\bibitem[DLMF]{DLMF} NIST \emph{Digital Library of Mathematical Functions}. http://dlmf.nist.gov/, Release 1.1.3 of 2021-09-15. 


\bibitem[DM91]{DM91} V.\,A.~Derkach and M.\,M.~Malamud,
Generalized resolvents and the boundary value problems for Hermitian operators
with gaps,
\textit{J.\ Funct.\ Anal.} \textbf{95} (1991), 1--95.

\bibitem[DM95]{DM95} V.\,A.~Derkach and M.\,M.~Malamud,
The extension theory of Hermitian operators and the moment problem,
\textit{J.\ Math.\ Sci.\ (New York)} \textbf{73} (1995), 141--242.

\bibitem[ES88]{ES88}
P. Exner and P. \v{S}eba, A simple model of thin-film point contact in two and three dimensions, {\it Czechoslovak J. Phys. B}
{\bf 38} (1988), 1095--1110.

\bibitem[FL23]{FL21}
D.~Frymark and V.~Lotoreichik,
Self-adjointness of the 2D Dirac operator with singular interactions supported on star graphs,
{\it Ann. Henri Poincar\'{e}} {\bf 24}
	 (2023), 179--221. 


\bibitem[G85]{G85}
P.~Grisvard, \emph{Elliptic problems in nonsmooth domains}, Pitman, Boston-London-Melbourne, 1985.
\bibitem[G92]{G92}
P.~Grisvard, \emph{Singularities in boundary value problems}, Springer-Verlag, Berlin, 1992.


\bibitem[GPS21]{GPS21}
F.~Gesztesy,  M.~Pang, and J.~Stanfill,
Bessel-type operators and a refinement of Hardy's inequality,
 In: Gesztesy F., Martinez-Finkelshtein A. (eds) From Operator Theory to Orthogonal Polynomials, Combinatorics, and Number Theory. Operator Theory: Advances and Applications, 285 (2021).

\bibitem[GR07]{GR07}
I.\,S.~Gradshteyn and I.\,M.~Ryzhik, 
\emph{Table of integrals, series, and products. 7th edition}
Elsevier/Academic Press, Amsterdam, 2007.

\bibitem[GR73]{GR73}
K.~Gustafson and P.~Rejto, 
Some essentially self-adjoint Dirac operators with spherically symmetric potentials,
{\it Isr. J. Math.} {\bf 14} (1973),
63--75.

\bibitem[H21]{H21}
M.~Holzmann, A note on the three dimensional Dirac operator with zigzag type boundary conditions, 
{\it Complex Anal. Oper. Theory} {\bf 15}  (2021), 47.

\bibitem[HS67]{HS67}
M. S. Hanna and K. T. Smith, Some remarks on the Dirichlet problem in piecewise smooth domains, 
{\it Commun. Pure Appl. Math.} {\bf 20} (1967), 575--593.

\bibitem[HMZ01]{HMZ01}
O.~Hijazi, S.~Montiel, and X.~Zhang,
Eigenvalues of the Dirac operator on manifolds with boundary,
{\it Commun. Math. Phys.} {\bf 221} (2001),
255--265.

\bibitem[HOP18]{HOP18}
M.~Holzmann, T.~Ourmi\`{e}res-Bonafos, and K.~Pankrashkin, Dirac operators with
Lorentz scalar shell interactions, {\it Rev. Math. Phys.} {\bf 30} (2018),  1850013.




\bibitem[K95]{Kato}
	T.~Kato,
	\emph{Perturbation theory for linear operators},
	Classics in Mathematics, Berlin: Springer-Verlag, 1995.

\bibitem[Ko75]{Ko75} A.\,N.~Kochubei,
Extensions of symmetric operators and symmetric binary relations,
\textit{Math.\ Zametki} \textbf{17} (1975), 41--48 (in Russian);
translated in: \textit{Math.\ Notes} \textbf{17} (1975), 25--28.

\bibitem[K67]{K67}
V.\,A.~Kondratiev, 
Boundary problems for elliptic equations in domains with conical or angular points, 
{\it Trans. Mosc. Math. Soc.} {\bf 16} (1967),
227--313.

\bibitem[J75]{J75}
K.~Johnson, The MIT bag model, {\it Acta Phys. Pol. B} {\bf 12}  (1975), 865--892.
	


\bibitem[McL]{McL}
W.~McLean, 
\emph{Strongly elliptic systems and boundary integral equations},
Cambridge University Press,  Cambridge, 2000.	


\bibitem[LO18]{LO18}
L.~Le Treust and T.~Ourmi\`{e}res-Bonafos, 
Self-adjointness of Dirac operators with infinite mass boundary conditions in sectors,
{\it Ann. Henri Poincar\'{e}} {\bf 19}  (2018),
1465--1487.

\bibitem[MR09]{MR09}
M. Marletta and G. Rozenblum, 
A Laplace operator with boundary conditions singular at one point, {\it J. Phys. A} {\bf 42} (2009), 125204, 11 pp.
\bibitem[NP18]{NP18}
S. A. Nazarov and N. Popoff, 
Self-adjoint and skew-symmetric extensions of the Laplacian
with singular Robin boundary condition, 
{\it C. R. Math. Acad. Sci. Paris}  {\bf 356} (2018), 927--932.
\bibitem[O97]{O97}
F.~Olver, 
\emph{Asymptotics and special functions},
A K Peters, Wellesley, 1997.


\bibitem[OV18]{OV18}
T.~Ourmi\`{e}res-Bonafos and L.~Vega, 
A strategy for self-adjointness of Dirac operators: applications to the MIT bag model and $\dl$-shell interactions,
{\it Publ. Mat.} {\bf 62} (2018), 397--437.

\bibitem[OP21]{OP21}
T.~Ourmi\`{e}res-Bonafos and F.~Pizzichillo,
Dirac operators and shell interactions: a survey, in: Mathematical Challenges of Zero-Range Physics. Springer, Cham, (2021) 105--131.

\bibitem[PV19]{PV19}
F.~Pizzichillo and H.~Van Den Bosch, Self-adjointness of two
dimensional Dirac operators on corner domains,
{\it J. Spectral Theory} {\bf 11} (2021), 1043--1079.


\bibitem[P13]{P13}
A.~Posilicano,
On the many Dirichlet Laplacians on a non-convex polygon and their approximations by point interactions,
{\it J. Funct. Anal.} {\bf 265} (2013), 303--323.

\bibitem[R20]{R2020}
V.~Rabinovich, Boundary problems for three-dimensional Dirac operators and generalized MIT bag models for unbounded domains, {\it Russian J. Math. Phys.}, vol. {\bf 27} (2020), 504--519.

\bibitem[R22]{R2022}
V.~Rabinovich, Fredholm property and essential spectrum of 3-D Dirac operators with regular and singular potentials, {\it Complex Var. Elliptic Equ.} {\bf 67} (2022), 938--961.

\bibitem[S12]{S12} K.~Schm\"{u}dgen,
\textit{Unbounded Self-adjoint Operators on Hilbert Space},
Graduate Texts in Mathematics, vol.~265, Springer, Dordrecht, 2012.

\bibitem[T92]{T92}
B.~Thaller, \emph{The Dirac Equation}, Springer, Berlin, 1992.

\bibitem[W87]{W87}
J.~Weidmann, \textit{Spectral theory of ordinary differential operators}, Springer-Verlag, Berlin, 1987.

\bibitem[W03]{W03}
J.~Weidmann, {\it Lineare Operatoren in Hilbertr\"{a}umen. Teil II: Anwendungen}, Teubner, Stuttgart, 2003.

\end{thebibliography}
\end{document}